\documentclass[12pt]{amsart}
\usepackage{amscd,amsmath,bm,latexsym,amssymb,amscd}
\usepackage[mathscr]{euscript} 
\usepackage[all]{xy}
\usepackage{tikz}
\usetikzlibrary{patterns,cd}
\usepackage{layout} 
\usepackage{enumerate}
\usepackage{dsfont}

\makeatletter							
\newtheorem*{rep@theorem}{\rep@title}
\newcommand{\newreptheorem}[2]{%
\newenvironment{rep#1}[1]{%
 \def\rep@title{#2 \ref{##1}}%
 \begin{rep@theorem}}%
 {\end{rep@theorem}}}
\makeatother

\newtheorem{theorem}{Theorem}[section]
\newreptheorem{theorem}{Theorem}
\newtheorem{corollary}[theorem]{Corollary}
\newreptheorem{corollary}{Corollary}
\newtheorem*{corollary*}{Corollary}
\newtheorem{lemma}[theorem]{Lemma}
\newtheorem{proposition}[theorem]{Proposition}
\theoremstyle{definition}
\newtheorem{definition}[theorem]{Definition}
\theoremstyle{definition}
\newtheorem{example}[theorem]{Example}
\theoremstyle{definition}

\newtheorem{remark}[theorem]{Remark}

\theoremstyle{remark}
\newtheorem*{notation*}{Notation}

\def\Z{\mathbb{Z}}
\def\D{\mathbb{D}}

\def\N{\mathbb{N}}
\def\C{\mathbb{C}}
\def\Q{\mathbb{Q}}
\def\R{\mathbb{R}}

\def\SS{\mathbb{S}}
\def\I{\mathbb{I}}
\def\BC{B_{com}}

\newcommand{\BONE}{\mathds 1}

\def\g{\mathfrak{g}}
\def\t{\mathfrak{t}}
\def\ti{\mathfrak{t}}
\def\s{\mathfrak{s}}
\def\x{\mathbf{x}}

\def\H{\mathcal{H}}
\def\K{\mathcal{K}}

\def\M{\mathcal{M}}
\def\U{\mathcal{U}}

\def\colim{\displaystyle\mathop{\textup{colim}}}
\def\hocolim{\displaystyle\mathop{\textup{hocolim}}}
\DeclareMathOperator{\Hom}{\textup{Hom}}

\def\Im{\textup{Im}}
\def\Ker{\textup{Ker}}
\def\Rep{\textup{Rep}}

\def\Ob{\textup{Ob}}
\def\Ad{\textup{Ad}}
\def\Mor{\textup{Mor}}
\def\op{\textup{op}}
\newcommand{\co}{\colon\thinspace}

\begin{document}

\title[On the second homotopy group of spaces of commuting elements] 
{On the second homotopy group of spaces of commuting elements in Lie groups}

\author[A.~Adem]{Alejandro Adem}
\address{Department of Mathematics,
University of British Columbia, Vancouver BC V6T 1Z2, Canada}
\email{adem@math.ubc.ca}

\author[J.~M.~G\'omez]{Jos\'e Manuel G\'omez}
\address{ Escuela de Matem\'aticas\\
Universidad Nacional de Colombia sede Medell\'in, 
Medell\'in, Colombia }       
\email{jmgomez0@unal.edu.co}

\author[S.~Gritschacher]{Simon Gritschacher}
\address{Department of Mathematical Sciences \\
University of Copenhagen, Copenhagen, Denmark }       
\email{gritschacher@math.ku.dk}

\begin{abstract}
Let $G$ be a compact connected Lie group and $n\geqslant 1$ an integer. 
Consider the space of ordered commuting $n$-tuples in $G$, $\Hom(\Z^n,G)$, and its quotient under the adjoint action, $\Rep(\Z^n,G):=\Hom(\Z^n,G)/G$. In this article we study and in many cases compute the homotopy groups $\pi_2(\Hom(\Z^n,G))$.
For $G$ simply--connected and simple we show that $\pi_2(\Hom(\Z^2,G))\cong \Z$ and $\pi_2(\Rep(\Z^2,G))\cong \Z$,
and that on these groups the quotient map $\Hom(\Z^2,G)\to \Rep(\Z^2,G)$ induces multiplication by the Dynkin index of $G$. More generally we show that if $G$ is simple and $\Hom(\Z^2,G)_{\BONE}\subseteq \Hom(\Z^2,G)$ is the path--component of the trivial homomorphism, then $H_2(\Hom(\Z^2,G)_{\BONE};\Z)$ is an extension of the Schur multiplier of $\pi_1(G)^2$ by $\Z$.
We apply our computations to prove that if $B_{com}G_{\BONE}$ is the classifying space for commutativity at the identity component, then $\pi_4(B_{com}G_{\BONE})\cong \Z\oplus\Z$, and we construct examples of 
non-trivial transitionally commutative structures on the trivial principal $G$-bundle 
over the sphere $\SS^{4}$.
\end{abstract}

\maketitle

\section{Introduction}

Suppose that $G$ is a compact connected Lie group. For an integer $n\geqslant 1$ let $\Hom(\Z^{n},G)$ be the space of ordered commuting $n$-tuples in 
$G$ endowed with the subspace topology as a subset of $G^{n}$. The group $G$ acts by conjugation on 
$\Hom(\Z^{n},G)$ and the space of representations $\Rep(\Z^{n},G):=\Hom(\Z^{n},G)/G$ can be identified with the moduli space of isomorphism classes of flat connections on principal $G$-bundles over the torus $(\SS^{1})^{n}$.

When  $n=2$ or $n=3$ these moduli spaces appear naturally in quantum field theories such as
Yang–Mills and Chern–Simons theories. Motivated by this a systematic study of the spaces 
$\Rep(\Z^{n},G)$ was initiated by Borel, Friedman and Morgan in \cite{BFM} 
and also by Kac and Smilga in \cite{KS}. In both of these papers the authors showed that 
the representation spaces $\Rep(\Z^{n},G)$ for $n=2,3$ can be described in terms of 
the root system associated to a choice of maximal torus in $G$.  If $G$ is simply--connected and simple, then $\Rep(\Z^2,G)$ can be furthermore identified with the moduli space of semistable principal bundles over an elliptic curve with structure group the complexification of $G$, and this is known to be a weighted projective space \cite{BS79, FMW, Laszlo, Looijenga}. On the other hand, in \cite{AC}
Adem and Cohen started a systematic study of the spaces of homomorphisms $\Hom(\Z^{n},G)$  
from the point of view of homotopy theory. Since then a variety of authors have studied 
these spaces using techniques from geometry and homotopy theory. See for example 
\cite{ACG,Baird,BLR,GPS,GH19,PS,RS,STG}, 
among others. 

The spaces of ordered commuting pairs $\Hom(\Z^2, G)$ turn out to have a suprisingly complicated structure; their integral homology is  not known for $G=SU(m)$ when $m>2$, and torsion can appear at primes which divide the order of the Weyl group. From the point of view of algebraic geometry, 
$\Hom(\Z^2,G)$ can be identified with $\Lambda (G)$, the inertia stack of $G$ with the adjoint action. In this context $ \Rep(\Z^2,G)$ can be regarded as the coarse moduli space of the stack, and as the local isotropy groups are not finite, the geometry can be rather intricate. In this article we describe the second homology group of the principal path--component $\Hom(\Z^2,G)_{\BONE}$ as an extension of the Schur multiplier of $\pi_1(G)^2$.

\begin{theorem}[Theorem \ref{thm:mainpairs2}] \label{thm:mainpairs2intro}
Suppose that $G$ is a semisimple compact connected Lie group. Then there is an extension
\[
0\to \Z^s \to H_2(\Hom(\Z^2,G)_{\BONE};\Z) \to H_2(\pi_1(G)^2;\Z)\to 0\,,
\]
where $s\geqslant 0$ is the number of simple factors in the Lie algebra of $G$.
\end{theorem}

Suppose that $G$ is simply--connected and simple. As $\Rep(\Z^2,G)$ is a weighted projective space one has $\pi_2(\Rep(\Z^2,G))\cong \Z$. The preceding theorem will be deduced from a calculation of $\pi_2(\Hom(\Z^2,G))$. One of the fundamental results in Lie group theory is that $\pi_2(G)=0$, and $\pi_3(G)\cong\Z$. There is a canonical 3-dimensional integral cohomology class that can be realized through a group homomorphism $SU(2)\to G$. In this paper we obtain the rather surprising result that for commuting pairs there is a canonical class which now appears in dimension two.

\begin{theorem}[Theorem \ref{thm:mainpairs}] \label{thm:mainpairsintro}
Let $G$ be a simply--connected and simple compact Lie group. Then
\[
\pi_2(\Hom(\Z^2,G)) \cong \Z\,,
\]
and on this group the quotient map $\Hom(\Z^2,G)\to \Rep(\Z^2,G)$ induces multiplication by the Dynkin index $\textnormal{lcm}\{n_0^\vee,\dots,n^\vee_r\}$ where 
$n_0^\vee,\dots,n_r^\vee \geqslant 1$ are the coroot integers of $G$.
\end{theorem}

The Dynkin index of 
$G$ can be defined as the greatest common divisor of the degrees of $\pi_3(G)\to \pi_3(SU(N))$ for all representations $G\to SU(N)$. The values of the Dynkin index of $G$ are explicitly computed in \cite[Proposition 4.7]{KN97} and \cite[Proposition 2.6]{Laszlo-Sorger1997} and agree with the expressions in terms of coroot integers which are tabulated below.

In Theorem \ref{thm:su2represents} we show that any embedding $SU(2)\to G$ corresponding to a long root of $G$ induces an isomorphism $\pi_2(\Hom(\Z^2,SU(2)))\cong \pi_2(\Hom(\Z^2,G))$. 

\begin{table}[b]
\centering
\def\arraystretch{1.5}
\begin{tabular}{|c||c|c|c|c|c|c|c|c|} \hline
$G$ & $SU(m)$ & $Spin(k)$ & $Sp(l)$ & $E_6$ & $E_7$ & $E_8$ & $F_4$ & $G_2$ \\ \hline \hline
coroot integers & $1$ & $1,2$ & $1$ & $1,2,3$ & $1,2,3,4$ & $1,2,3,4,5,6$ & $1,2,3$ & $1,2$ \\ \hline
lcm & $1$ & $2$ & $1$ & $6$ & $12$ & $60$ & $6$ & $2$ \\ \hline
\end{tabular}
\vspace{10pt}
\caption{The set of coroot integers and their least common multiple for the simple simply--connected 
compact Lie groups ($m\geqslant 1$, $k\geqslant 7$, $l\geqslant 2$).} \label{table:coroot}
\end{table}

We note that Theorem \ref{thm:mainpairsintro} effectively computes $\pi_2(\Hom(\Z^2,G)_{\BONE})$ when $G$ is the group of complex or real points of any reductive algebraic group. Indeed, the main result of \cite{PS} asserts that $\Hom(\Z^2,G)_{\BONE}$ deformation retracts onto $\Hom(\Z^2,K)_{\BONE}$ where $K$ is a maximal compact subgroup of $G$. So we may assume that $G$ is a compact connected Lie group. From the standard classification theorems we know that $G\cong \tilde{G}/L$, where $\tilde{G}
=(\SS^1)^k\times G_1\times\dots\times G_s$ is a product of a torus and simply--connected simple 
compact Lie groups $G_1,\dots , G_s$, and $L$ is a finite subgroup in the center of $\tilde{G}$. By \cite[Lemma 2.2]{Goldman} the quotient map $\tilde{G}\to G$ induces a 
covering map $\Hom(\Z^2,\tilde{G})\to \Hom(\Z^2,G)_{\BONE}$. From Theorem \ref{thm:mainpairsintro} we then deduce that
\[
\pi_{2}(\Hom(\Z^{2},G)_{\BONE})\cong 
\pi_{2}(\Hom(\Z^{2},G_{1}))\times\cdots\times \pi_{2}(\Hom(\Z^{2},G_{s}))\cong \Z^{s}.
\]

\begin{corollary}
Let $G$ be the component of the identity of the group of complex or real points of a reductive algebraic group, which we assume is defined over $\R$ in the latter case. Then 
\[
\pi_2(\Hom(\Z^2,G)_{\BONE})\cong \Z^{s}\,,
\]
where $s\geqslant 0$ is the number of simple factors in the Lie algebra of $G$.
\end{corollary}

For the groups $SU(m)$  and $Sp(k)$ we extend our computations to commuting $n$-tuples for $n>2$. 

\begin{theorem}[Theorem \ref{thm:generalsecondhomology}] \label{thm:generalsecondhomologyintro}
Suppose that $G$ is $SU(m)$ or $Sp(m)$ with $m\geqslant 1$. For every $n\geqslant 1$ 
the quotient map $\Hom(\Z^{n},G)\to \Rep(\Z^{n},G)$ induces an isomorphism 
\[
\pi_{2}(\Hom(\Z^{n},G))\cong \pi_{2}(\Rep(\Z^{n},G))\, .
\]
\end{theorem}

We then calculate $\pi_2(\Rep(\Z^n,G))$ and obtain the following result.

\begin{theorem}[Theorem \ref{thm:casesymplecticandunitary}] \label{thm:casesymplecticandunitaryintro}
Let $n\geqslant 1$. Then
\begin{itemize}
\item[(i)] for every $m\geqslant 3$ there is an isomorphism
\[
\pi_{2}(\Hom(\Z^{n},SU(m))) \cong \Z^{\binom{{n}}{{2}}}\,,
\]
and the standard inclusion $SU(m)\to SU(m+1)$ induces an isomorphism 
\[
\pi_2(\Hom(\Z^n,SU(m))) \xrightarrow{\cong} \pi_2(\Hom(\Z^n,SU(m+1)))\,,
\]
\item[(ii)] for every $k\geqslant 1$ there is an isomorphism
\[
\pi_2(\Hom(\Z^n,Sp(k))) \cong \Z^{\binom{{n}}{{2}}}\oplus (\Z/2)^{2^{n}-1-n-\binom{{n}}{{2}}}\,,
\]
and the standard inclusion $Sp(k)\to Sp(k+1)$ induces an isomorphism 
\[
\pi_2(\Hom(\Z^n,Sp(k)))\xrightarrow{\cong} \pi_2(\Hom(\Z^n,Sp(k+1)))\, .
\]
\end{itemize}
\end{theorem}

By the same reasoning as above, using the covering space $S^1\times SU(m)\to U(m)$ we observe that for every $m,n\geqslant 1$ the inclusion map 
$SU(m)\to U(m)$ induces an isomorphism 
\[
\pi_{2}(\Hom(\Z^{n},SU(m)))\cong \pi_{2}(\Hom(\Z^{n},U(m)))
\]
so that the preceding theorem also covers the case of the unitary groups. 

For spaces of commuting pairs in $Spin(m)$ we also explore the behavior of $\pi_2$ with respect to stabilization. 

\begin{theorem}[Theorem \ref{thm:spinstability}] \label{thm:spinstabilityintro}
For all $m\geqslant 5$ the standard map $Spin(m)\to Spin(m+1)$ induces an isomorphism
\[
\pi_2(\Hom(\Z^2,Spin(m)))\xrightarrow{\cong} \pi_2(\Hom(\Z^2,Spin(m+1)))\, .
\]
\end{theorem}

Along the way,  in Theorem \ref{thm:stabilityrepspin} we prove an integral homology stability result for the moduli spaces $\Rep(\Z^2,Spin(m))$.

A motivation for the computations provided in this article is the construction of 
non-trivial transitionally commutative structures  (TC structures) on a trivial principal 
$G$-bundle. For a Lie group $G$ the classifying bundle for commutativity, $E_{com}G\to B_{com}G$, 
is a principal $G$--bundle constructed out of the spaces of ordered commuting 
$n$-tuples in $G$ for all
$n\geqslant 0$ (see Section \ref{TC structures} for the definition). In this setting, the analogue of Steenrod's homotopy classification of bundles is as follows: the space $B_{com}G$ classifies 
principal $G$-bundles that come equipped with a TC structure. A TC structure on a principal $G$-bundle 
$p\co E\to X$ is a choice of a lifting $\tilde{f}\co X\to B_{com}G$, up to homotopy, of the classifying map 
$f\co X\to BG$ of the bundle $p\co E\to X$. Therefore, the same underlying bundle can admit different 
TC structures. In what follows we focus on the associated bundle for the component of the identity
$E_{com}G_{\BONE}\to B_{com}G_{\BONE}\to BG$.

In the classical setting, the computation $\pi_3(G)\cong \Z$ implies that $\pi_4(BG)\cong \Z$ and this can be used to construct non--trivial principal $G$--bundles over the 4--sphere $\SS^4$. Our main
result in the commutative context (see Corollary \ref{cor:pi4}) is the following
computation for any  simply--connected simple compact Lie group $G$:
\[
\pi_4(E_{com}G_{\BONE})\cong \Z\, \textnormal{ and }\, \pi_4(B_{com}G_{\BONE})\cong
\pi_4( E_{com}G_{\BONE})\oplus \pi_4(BG)\cong \Z\oplus \Z\, .
\]

We provide explicit generators for these groups. Moreover we construct examples of non-trivial TC structures on the trivial principal $G$-bundle over the 
sphere $\SS^{4}$. As a by-product of this construction we obtain geometric representatives for the generators 
of the reduced commutative $K$-theory group $\widetilde{K}_{com}(\SS^{4})$. \medskip

This article is organized as follows. In Section \ref{sec:cohomology} we present a cohomology computation with the goal of determining the rank of $\pi_2(\Hom(\Z^n,G)_{\BONE})$. In Section 
\ref{SectionBredon} we set up the spectral sequence for the homology of the homotopy
orbit space
$EG\times_G \Hom(\Z^{n},G)_{\BONE}$, which will be the main device used for our computations. In Section \ref{sec:sumspm} we calculate $\pi_2$ of the space of ordered commuting $n$-tuples in $G=SU(m)$ and $G=Sp(k)$. 
Section \ref {sec:pairs} begins with a review of some known facts concerning centralizers of pairs of 
commuting elements. Then we prove some of the main technical results of the paper, eventually arriving at the calculation of $\pi_2(\Hom(\Z^2,G))$. In Section \ref{sec:stability} we study the 
stability behavior for spaces of commuting pairs in the Spin groups. In 
Section \ref{sec:rolesu2} we study the distinguished role that the group $SU(2)$ plays in the computation 
of the groups $\pi_{2}(\Hom(\Z^{2},G))$. Finally, in Section \ref{TC structures} we provide 
geometric interpretations for the results obtained in this article in terms of transitionally 
commutative principal $G$-bundles.

\section{Cohomological computations}\label{sec:cohomology}

Let $G$ be a compact connected Lie group, let $T\leqslant G$ be a fixed maximal torus, and let $W=N(T)/T$ be the Weyl group. We begin by recalling some known facts concerning the path--components and the fundamental group of $\Hom(\Z^n,G)$ which will be used in the sequel. 

In general, $\Hom(\Z^n,G)$ is not path--connected. If $\pi_{1}(G)$ is torsion free, however, then $\Hom(\Z^2,G)$ is path--connected, because the centralizer of any element of $G$ is connected (see \cite[IX  \S 5.3 Corollary 1]{Bourbaki}). In fact, classical theorems of Borel \cite{Bo60} show that $\Hom(\Z^2,G)$ is path--connected if and only if $\pi_1(G)$ is torsion free, and for $n\geqslant 3$ the space $\Hom(\Z^n,G)$ is path--connected if and only if $H_\ast(G;\Z)$ is torsion free. For example, this is the case for $G=SU(m)$, $G=U(l)$ and $G=Sp(k)$.

When $\Hom(\Z^{n},G)$ is not path--connected, we denote by $\Hom(\Z^{n},G)_{\BONE}$ the path--connected 
component that contains the trivial homomorphism $\BONE\co \Z^n\to G$, equivalently the tuple $\BONE=(1_G,\dots,1_G)$. Then $\Hom(\Z^{n},G)_{\BONE}$ consists precisely of those $n$-tuples that are contained in some maximal torus of $G$. For simply--connected $G$ this follows, because the rank of the centralizer of a commuting $n$-tuple is locally constant (see \cite[Corollary 2.3.2]{BFM}). For general $G$ it follows by the same argument applied to the universal cover of $G$ using \cite[Lemma 2.2]{Goldman}. Consequently,
\[
\Rep(\Z^n,G)_\BONE:=\Hom(\Z^n,G)_\BONE/G \cong T^n/W\,,
\]
where on the right hand side the Weyl group acts diagonally on $T^n$. All of our statements will concern the path-components $\Hom(\Z^n,G)_{\BONE}$ and $\Rep(\Z^n,G)_{\BONE}$.

The fundamental group of $\Hom(\Z^n,G)_{\BONE}$ is also well-understood. By \cite[Theorem 1.1]{GPS} the inclusion $\Hom(\Z^n,G)_{\BONE}\subseteq G^n$ induces an isomorphism
\[
\pi_1(\Hom(\Z^n,G)_{\BONE})\cong \pi_1(G)^n\, .
\]
In particular, $\Hom(\Z^n,G)_{\BONE}$ is simply--connected if $G$ is. In this situation also $\Rep(\Z^n,G)_{\BONE}$ is simply--connected by \cite[Theorem 1.1]{BLR}. We will use these facts frequently without mention.

Now we turn our attention to the rational homology of $\Hom(\Z^n,G)_{\BONE}$. As explained below the rational cohomology ring of $\Hom(\Z^n,G)_{\BONE}$ admits a well known description as a ring of $W$-invariants. A closed formula for the Poincar{\'e} series of $\Hom(\Z^n,G)_{\BONE}$ was derived in \cite{RS} (see Remark \ref{rem:poincare}). For a fixed homological degree $i\in \N$, however, determination of the rank of $H_i(\Hom(\Z^n,G)_{\BONE};\Q)$ can pose a combinatorial challenge. In this section we present a calculation of $H_2(\Hom(\Z^n,G)_{\BONE};\Q)$. This determines the rank of $\pi_2(\Hom(\Z^n,G)_{\BONE})$ as an abelian group.

To describe the rational cohomology of $\Hom(\Z^n,G)_{\BONE}$ consider the (surjective) map 
\begin{alignat*}{1}
G\times T^{n}&\to \Hom(\Z^{n},G)_{\BONE}\\
(g,t_{1},\dots,t_{n})&\mapsto (gt_{1}g^{-1},\dots,gt_{n}g^{-1})\, .
\end{alignat*}
It is invariant under the action of the normalizer $N(T)$ on $G$ by right translation and on $T^n$ diagonally by conjugation. Since $G\times_{N(T)} T^n\cong G/T\times_W T^n$, we obtain an induced map
\[
G/T\times_{W}T^{n}\to \Hom(\Z^{n},G)_{\BONE}\, .
\]
Baird showed in \cite[Theorem 4.3]{Baird} that if $k$ is a field in which $|W|$ is invertible, then this map induces an isomorphism
\[
H^\ast(\Hom(\Z^n,G)_{\BONE};k)\cong H^\ast(G/T\times_W T^n;k)\cong H^\ast(G/T\times T^n;k)^W\, .
\]

\begin{theorem}\label{2nd homology}
Suppose that the compact connected Lie group $G$ is simple. Then
\[ 
H_{2}(\Hom(\Z^{n},G)_{\BONE};\Q)\cong \Q^{\binom{n}{2}}.
\]
\end{theorem}
\begin{proof}
Because of the universal coefficient theorem, we may as well use complex coefficients throughout this proof. Let $W$ act diagonally on $H^{1}(T;\C)\otimes H^{1}(T;\C)$. As a first step we are going to prove that $(H^{1}(T;\C)\otimes H^{1}(T;\C))^{W}\cong \C$. Let $\ti$ be the Lie algebra of $T$ and $\ti_{\C}$ its 
complexification. 
There is an isomorphism of $W$-representations $H^{1}(T;\C)\cong \ti^{*}_{\C}$, where 
$\ti_{\C}^{*}$ denotes the dual of $\ti_{\C}$. Therefore 
$(H^{1}(T;\C)\otimes H^{1}(T;\C))^{W}\cong (\ti_{\C}^{*}\otimes \ti_{\C}^{*})^{W}$. 
To complete the first step we need to prove that 
$(\ti_{\C}^{*}\otimes \ti_{\C}^{*})^{W}$ is a complex vector space of dimension $1$. 
Note that $\text{dim}_{\C}(\ti_{\C}^{*}\otimes \ti_{\C}^{*})^{W}$ 
is the number of times that the trivial representation appears in  
the $W$-representation $\ti_{\C}^{*}\otimes \ti_{\C}^{*}$. 
This number is given by 
\[
\text{dim}_{\C}(\ti_{\C}^{*}\otimes \ti_{\C}^{*})^{W}=\left<\chi_{\ti_{\C}^{*}}^{2},1\right>
=\frac{1}{|W|}\sum_{w\in W}\overline{\chi_{\ti_{\C}^{*}}(w)}^{2}\,,
\]
i.e., the Frobenius-Schur indicator of the character $\chi_{\ti_{\C}^*}$ of the $W$-representation $\ti_{\C}^{*}$. Here $\left<\cdot,\cdot\right>$ denotes the usual inner product of characters. 
As $\ti_{\C}^{*}$ is (the complexification of) a real $W$-representation we have $\chi_{\ti_{\C}^{*}}(w)\in \R$. Moreover, 
$\ti_{\C}^{*}$ is irreducible by \cite[Proposition 14.31]{FultonHarris} since $\g_{\C}$, the complexification of the Lie algebra of $G$, is simple. Therefore,
\[
\text{dim}_{\C}(\ti_{\C}^{*}\otimes \ti_{\C}^{*})^{W}
=\left<\chi_{\ti_{\C}^{*}}^{2},1\right>
=\left<\chi_{\ti_{\C}^{*}},\chi_{\ti_{\C}^{*}}\right>
=1.
\]

Next we prove that  $H^{2}(\Hom(\Z^{n},G)_{\BONE};\C)\cong \C^{\binom{n}{2}}$. 
With complex coefficients there are isomorphisms
\[
H^{2}(\Hom(\Z^{n},G)_{\BONE};\C)\cong H^{2}(G/T\times_{W}T^{n};\C)
\cong H^{2}(G/T\times T^{n};\C)^{W}.
\]
Using the K{\"u}nneth theorem we obtain a $W$-equivariant isomorphism
\[
H^{2}(G/T\times T^{n};\C)\cong H^{2}(G/T;\C)\oplus 
(H^{1}(G/T;\C)\otimes H^{1}(T^{n};\C))\oplus  H^{2}(T^{n};\C).
\]
The flag variety $G/T$ is simply--connected and thus $H^{1}(G/T;\C)=0$. 
Furthermore, as an ungraded ring, $H^{*}(G/T;\C)$ is the regular $W$-representation
and the only copy of the trivial representation 
is in degree $0$. This implies that $H^{2}(G/T;\C)^{W}=0$. Therefore, 
\[
H^{2}(G/T\times T^{n};\C)^{W}\cong H^{2}(T^{n};\C)^{W}.
\]
By \cite[IX \S 5.2 Corollary 1]{Bourbaki}
the space $T/W$ is homeomorphic to $A/\pi_1(G)$, where $A$ is the closure of a Weyl alcove. Now $\pi_1(G)$ is finite as $G$ is simple, and $A$ is contractible, hence
$H^{2}(T;\C)^{W}\cong H^{2}(T/W;\C) \cong H^2(A/\pi_1(G);\C)\cong H^2(A;\C)^{\pi_1(G)}=0$. 
Using this and the K{\"u}nneth theorem we obtain an isomorphism 
\[
H^{2}(T^{n};\C)^{W}\cong \bigoplus^{\binom{n}{2}}(H^{1}(T;\C)\otimes H^{1}(T;\C))^{W}
\cong \C^{\binom{n}{2}}.
\]
Putting everything together we conclude that 
$H^{2}(\Hom(\Z^{n},G)_{\BONE};\C)\cong \C^{\binom{n}{2}}$. The universal coefficient theorem now finishes the proof.
\end{proof}

\begin{corollary} \label{cor:rank}
Let $G$ be a simple 
compact connected Lie group. Then 
$\pi_{2}(\Hom(\Z^{n},G)_{\BONE})$ has rank $\binom{n}{2}$ as an abelian group.
\end{corollary}
\begin{proof}
As explained in the introduction $\pi_2(\Hom(\Z^n,G)_{\BONE})\cong \pi_2(\Hom(\Z^n,\tilde{G})_{\BONE})$, where $\tilde{G}$ is the universal covering group of $G$. It is compact as $G$ is simple. Now $\Hom(\Z^{n},\tilde{G})_{\BONE}$ is simply--connected by \cite[Theorem 1.1]{GPS}. Therefore, by Hurewicz' theorem, we obtain   
\[
\pi_{2}(\Hom(\Z^{n},\tilde{G})_{\BONE})\cong H_{2}(\Hom(\Z^{n},\tilde{G})_{\BONE};\Z)\,.
\] The assertion follows from 
Theorem \ref{2nd homology}.
\end{proof}

\begin{remark} \label{rem:poincare}
Let $G$ be a simple compact connected Lie group. We can view the Weyl group $W$ 
as a reflection group on $\ti$ and every element $w\in W$ as a linear transformation $w:\ti\to \ti$. Let 
$d_{1}, d_{2},\dots , d_{r}$ be the characteristic degrees of $W$. 
By \cite[Theorem 1.1]{RS} the Poincar\'e series of $\Hom(\Z^{n},G)_{\BONE}$ is 
given by 					
\[
P(\Hom(\Z^{n},G)_{\BONE})(t)=\frac{\prod_{i=1}^{r}(1-t^{2d_{i}})}{|W|}
\left(\sum_{w\in W}\frac{\text{det}(1+tw)^{n}}{\text{det}(1-t^{2}w)}\right).
\]
Our Theorem \ref{2nd homology} shows that the coefficient of $t^{2}$ in this polynomial is precisely 
$\binom{n}{2}$. Moreover, as $G$ is simple $\pi_1(G)$ is finite. Then $\pi_1(\Hom(\Z^n,G)_{\BONE})\cong \pi_1(G)^n$ is finite too, so
\[
P(\Hom(\Z^{n},G)_{\BONE})(t)=1+{\binom{n}{2}}t^{2}+t^{3}q(t)\,,
\]
where $q(t)$ is some polynomial with non-negative integer coefficients.  
\end{remark}

\section{The Bredon spectral sequence for $\Hom(\Z^n,G)_{\BONE}$}\label{SectionBredon}

In this section we will analyze the Bredon spectral sequence associated to the 
$G$-space $\Hom(\Z^n,G)_{\BONE}$.

\subsection{The set-up} We begin with an observation:

\begin{lemma}\label{equihom1}
Let $G$ be a simply--connected compact Lie group. There is an isomorphism
\[
\pi_2(\Hom(\Z^n,G)_{\BONE})\cong  H_{2}(EG\times_G\Hom(\Z^n,G)_{\BONE};\Z).
\]
\end{lemma}
\begin{proof}
Since $G$ is assumed to be simply--connected, $BG$ is 3-connected, and 
$\Hom(\Z^n,G)_{\BONE}$ is simply--connected by \cite[Theorem 1.1]{GPS}. 
The long exact sequence of homotopy groups associated to the fibration sequence
\[
\Hom(\Z^n,G)_{\BONE}\to EG\times_G \Hom(\Z^n,G)_{\BONE}\to BG
\]
then implies that $EG\times_G \Hom(\Z^n,G)_{\BONE}$ is simply--connected, and 
that $\pi_2(\Hom(\Z^n,G)_{\BONE})\cong \pi_2(EG\times_G \Hom(\Z^n,G)_{\BONE})$. 
The result follows now from the Hurewicz theorem.
\end{proof}

As a consequence, the homotopy group $\pi_2(\Hom(\Z^n,G)_{\BONE})$ can be 
computed from the Borel equivariant homology of $\Hom(\Z^n,G)_{\BONE}$. Now if $X$ is a $G$-CW-complex, then the skeletal 
filtration of $X$ gives rise to an Atiyah-Hirzebruch style spectral sequence
\[
E^2_{p,q}=H^G_p(X;\mathcal{H}_q)\Longrightarrow H_{p+q}(EG\times_G X;\Z)\,,
\]
where $H^{G}_{p}(X;\H_{q})$ denotes Bredon homology with respect to the covariant coefficient system
\[
G/K\mapsto \mathcal{H}_q(G/K):=H_q(BK;\Z)\, ,
\]
for $K$ a closed subgroup of $G$. For Bredon homology in the setting of topological groups we refer the reader to the paper by Willson \cite{willson}. The spectral sequence is a special case of \cite[Theorem 3.1]{willson}. We will refer to it as the \emph{Bredon spectral sequence}.

\medskip

If $G$ is a compact Lie group, then $\Hom(\Z^n,G)$ may be viewed as an affine 
$G$-variety, and thus the underlying space admits a $G$-CW-structure by 
\cite[Theorem 1.3]{ParkSuh}. In fact, if $G$ is simply--connected, it is not too 
difficult to construct a $G$-CW-structure on $\Hom(\Z^n,G)_{\BONE}$ directly, and 
we will do this in Lemma \ref{lem:cw} (see also Remark \ref{rem:cwgeneral}). 
Thus, taking $X=\Hom(\Z^n,G)_{\BONE}$ the spectral sequence we will study takes the form
\begin{equation} \label{eq:bredonss}
E^2_{p,q}=H^G_p(\Hom(\Z^n,G)_{\BONE};\mathcal{H}_q) 
\Longrightarrow H_{p+q}(EG\times_G \Hom(\Z^n,G)_{\BONE};\Z)\, .
\end{equation}
Notice that when $q=0$ then $\mathcal{H}_q=\underline{\Z}$ is the constant coefficient system. Therefore,
\[
E^2_{p,0}=H^G_{p}(\Hom(\Z^n,G)_{\BONE};\underline{\Z})\cong H_p(\Rep(\Z^n,G)_{\BONE};\Z)\, .
\]

We will be interested not only in calculating the homotopy group $\pi_2(\Hom(\Z^n,G)_{\BONE})$, 
but also in describing the effect of the quotient map
\[
\pi\co \Hom(\Z^n,G)_{\BONE}\to \Rep(\Z^n,G)_{\BONE}
\]
on $\pi_2$. By Lemma \ref{equihom1}, this map can be identified with the one induced on $H_2$ by the projection
\[
EG\times_G\Hom(\Z^n,G)_{\BONE}\to \Hom(\Z^n,G)_{\BONE}/G= \Rep(\Z^n,G)_{\BONE}
\]
from the homotopy orbit to the strict orbit. In the spectral sequence this corresponds to the composite
\[
H_2(EG\times_G\Hom(\Z^n,G)_{\BONE};\Z)\to E^\infty_{2,0}\hookrightarrow 
E^2_{2,0}\cong H_2(\Rep(\Z^n,G)_{\BONE};\Z)\, .
\]

The case $n=2$ is an interesting special case; it is known that if $G$ is simply--connected and simple, then $\Rep(\Z^2,G)$ can be identified naturally with the moduli space of semistable principal bundles over an elliptic curve with structure group the complexification of $G$, and this moduli space is a weighted projective space $\mathbb{CP}(\mathbf{n}^\vee)$, where 
$\mathbf{n}^\vee=(n_0^\vee,\dots,n_r^\vee)$ is the tuple of coroot integers of $G$ 
(see Section \ref{sec:componentgroup} for the definition of coroot integers and Section 
\ref{sec:quotients} for the definition of a weighted projective space). This follows from the work of various authors \cite{BS79,FMW,Laszlo,Looijenga,NS,Ra}.

Our proofs will be self-contained; in Proposition \ref{prop:weightedprojective} we will prove independently, and by elementary topological methods, that $\Rep(\Z^2,G)$ is homotopy equivalent to $\mathbb{CP}(\mathbf{n}^\vee)$.

The homology of weighted projective spaces was computed by Kawasaki \cite{K73} from which one 
obtains, for simple simply--connected $G$ of rank $r$,
\[
H_p(\Rep(\Z^2,G);\Z)\cong \begin{cases} \Z & \textnormal{if }p\leqslant 2r \textnormal{ is even}, 
\\ 0 & \textnormal{otherwise.} \end{cases}
\]
Figure \ref{figure:e2page} depicts the spectral sequence in this situation. 
We will be interested only in the part of the spectral sequence that is relevant for 
homological degree two.

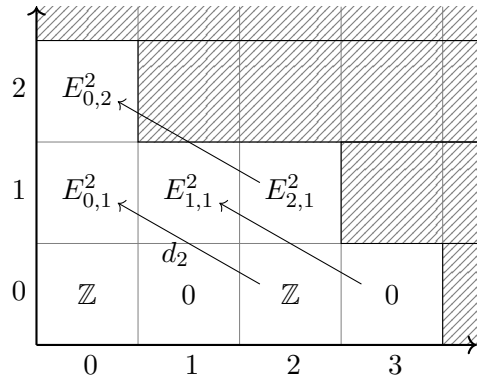
\begin{figure}[b]
\begin{tikzpicture}[scale =0.9]
\small

\clip (-0.5, -0.5) rectangle (6.5, 5); 
\draw[step=1.5, gray, very thin] (0,0) grid (7, 5);

\draw[pattern=north east lines, pattern color=gray] (0,4.5) rectangle (7,6);
\draw[pattern=north east lines, pattern color=gray] (1.5,3) rectangle (7,4.5);
\draw[pattern=north east lines, pattern color=gray] (4.5,1.5) rectangle (7,3);
\draw[pattern=north east lines, pattern color=gray] (6,0) rectangle (7,1.5);

\draw[->, thick](0,0)--(6.5, 0); 
\draw[->, thick](0,0)--(0, 5); 

\draw (0.75, 0.75) node[anchor=center]{$\Z$};
\draw (2.25, 0.75) node[anchor=center]{$0$};
\draw (3.75, 0.75) node[anchor=center]{$\Z$};
\draw (5.25, 0.75) node[anchor=center]{$0$};

\draw (0.75, 2.25) node[anchor=center]{$E^2_{0,1}$};
\draw (2.25, 2.25) node[anchor=center]{$E^2_{1,1}$};
\draw (3.75, 2.25) node[anchor=center]{$E^2_{2,1}$};

\draw (0.75, 3.75) node[anchor=center]{$E^2_{0,2}$};

\foreach \x in {0,1, 2, 3}
\draw (0.8+3*\x/2,0) node[anchor=north] {$\x$};
\foreach \y in {0,1, 2}
\draw (0,0.8+3*\y/2) node[anchor=east] {$\y$};

\draw (2.05, 1.35) node[anchor=center]{$d_{2}$};

\draw[->] (3.3,0.9)-- (1.2, 2.1);
\draw[->] (3.3,2.4)-- (1.2, 3.6);
\draw[->] (4.8,0.9)-- (2.7, 2.1);

\end{tikzpicture}
\caption{A portion of the $E^2$-page of the Bredon spectral sequence in the 
case of commuting pairs in a simple, simply--connected, compact Lie group of 
rank $\geqslant 1$.}
\label{figure:e2page}
\end{figure}

\subsection{First consequences}

Going back to the general situation, we will prove that $E^2_{0,2}$ vanishes in certain cases. 
For this we will use the following lemma.

\begin{lemma}\label{computation spectral sequence}
Suppose that $G$ is a compact Lie group and $X$ is a path--connected $G$-CW complex that has a 
basepoint $x_{0}$ fixed by $G$.  Assume that $\M$ is a covariant coefficient 
system such that $\M(G/G)=0$ and for every morphism $G/H\to G/K$ in the orbit category the induced map $\M(G/H)\to \M(G/K)$ is surjective. 
Then $H^{G}_{0}(X;\M)=0$.
\end{lemma}
\begin{proof}
After passing to an appropriate subdivision we may assume without loss 
of generality that the $G$-CW complex structure has a $0$-dimensional 
cell corresponding to the basepoint $x_{0}$. As this basepoint 
is fixed by $G$ this cell must be of the form $G/G$. 

By definition $H^{G}_{p}(X;\M)=H_{p}(C_{*}(X;\M))$. Here 
the group $C_{i}(X;\M)$ of cellular $i$-chains is defined by 
\[
C_{i}(X;\M)=\bigoplus_{\sigma\in S_{i}(X)} \M(G/G_{\sigma}),
\]
where $S_{i}(X)$ is an indexing set of all
$i$-dimensional $G$-cells of $X$, and $G_\sigma$ is the isotropy group of $\sigma\in S_i(X)$. To prove the lemma we are going to show that the differential 
\[
d_{1}\co \bigoplus_{\tau\in S_{1}(X)} \M(G/G_{\tau})\to 
\bigoplus_{\sigma\in S_{0}(X)} \M(G/G_{\sigma})
\]
is surjective. Let us first recall the definition of $d_1$. Suppose that $\tau\in S_1(X)$ and let $f_\tau\co G/G_\tau\times \partial D^1\to X^{(0)}$ be the attaching map of $\tau$. Write $\partial D^1=\{0,1\}$. Suppose that $f_\tau$ restricts to $G$-maps $f_{\tau,0}\co G/G_{\tau}\times \{0\}\to G/G_{\sigma_0}\times D^0$ and $f_{\tau,1}\co G/G_{\tau}\times \{1\} \to G/G_{\sigma_1}\times D^0$ for $\sigma_i\in S_0(X)$, $i=0,1$. Given $y\in \mathcal{M}(G/G_\tau)$, then
\[
d_1(y)=\mathcal{M}(f_{\tau,1})(y)-\mathcal{M}(f_{\tau,0})(y)\,.
\]
Now suppose that $\sigma\in S_{0}(X)$ and $z\in \M(G/G_{\sigma})$. We must show that $z\in \text{Im}(d_{1})$. As $X$ is assumed to be path--connected, we can find $\sigma_0,\dots,\sigma_n \in S_0(X)$ and $\tau_1,\dots,\tau_n\in S_1(X)$ such that the attaching map $f_{\tau_k}$ restricts to $f_{\tau_k,0}\co G/G_{\tau_k}\times \{0\}\to G/G_{\sigma_{k-1}}\times D^0$ and $f_{\tau_k,1}\co G/G_{\tau_k}\times \{1\}\to G/G_{\sigma_k}\times D^0$, and such that $\sigma_n=\sigma$ and $\sigma_0$ is the $G$-cell corresponding to the $G$-fixed basepoint $x_0$. By hypothesis $\mathcal{M}(f_{\tau_k,i})$ is surjective for every $k=1,\dots,n$ and $i=0,1$. Therefore, we find $y_{n}\in \M(G/G_{\tau_{n}})$ such that $\mathcal{M}(f_{\tau_n,1})(y_{n})=z_n:=z$, and $d_{1}(y_{n})=z_n-z_{n-1}$ for some $z_{n-1}\in \M(G/G_{\sigma_{n-1}})$. Continuing like this we find for every $k=1,\dots,n$ elements $y_{k}\in \M(G/G_{\tau_{k}})$ and $z_{k-1}\in \mathcal{M}(G/G_{\sigma_{k-1}})$ such that 
$\mathcal{M}(f_{\tau_k,1})(y_{k})=z_{k}$ and $d_{1}(y_{k})=z_{k}-z_{k-1}$. 
Since $\M(G/G_{\sigma_{0}})=\M(G/G)=0$, we have that $z_0=0$ and thus $d_1(y_1)=z_1$. Then
\[
d_{1}(y_{n}+y_{n-1}+\cdots+ y_{1})
=(z_{n}-z_{n-1})+(z_{n-1}-z_{n-2})+\cdots +(z_{2}-z_{1})+z_1=z_{n}=z.
\] 
This proves that $d_{1}$ is surjective. 
\end{proof}

In addition, we must make an assumption about the components of the isotropy groups of 
$\Hom(\Z^n,G)_{\BONE}$ under the conjugation action by $G$. Suppose that $\underline{x}=(x_1,\dots,x_n)\in \Hom(\Z^n,G)_{\BONE}$ is an $n$-tuple of commuting elements. Then the isotropy group of $\underline{x}$ is
\[
G_{\underline{x}}=Z_G(\underline{x})\,,
\]
i.e., the centralizer of the subset $\{x_1,\dots,x_n\}\subseteq G$. The following fact is explained in \cite[~Example~2.4]{AG1}.

\begin{lemma} \label{lem:connectedisotropy}
Let $G$ be $SU(m)$ or $Sp(m)$ for some $m\geqslant 1$, and let $n\geqslant 1$. Then for every $\underline{x}\in \Hom(\Z^n,G)_{\BONE}$ the centralizer $Z_G(\underline{x})$ is connected.
\end{lemma}

In general, $Z_{G}(\underline{x})$ may not be connected. For example, 
when $G=Spin(m)$ with $m\geqslant 7$ there exist pairs $(x,y)\in \Hom(\Z^2,G)$ such that $\pi_0(Z_G(x,y))\cong \Z/2$. The following result is a special case of 
\cite[Corollary 7.5.3]{BFM} and a proof will be given in Section \ref{sec:componentgroup}.

\begin{lemma} \label{lem:componentsisotropy}
Let $G$ be a simply--connected and simple compact Lie group. For every 
$(x,y)\in \Hom(\Z^2,G)_{\BONE}$ the group $\pi_0(Z_G(x,y))$ is a finite cyclic group (of order at most $6$).
\end{lemma}

We can now prove:

\begin{proposition} \label{prop:vanishinge02}
Suppose that either $n=2$, or that $n\geqslant 2$ and $G$ is $SU(m)$ or $Sp(m)$ for some 
$m\geqslant 1$. Then $E^2_{0,2}=H_0^G(\Hom(\Z^n,G)_{\BONE};\mathcal{H}_2)= 0$. 
\end{proposition}
\begin{proof}
For ease of notation we denote $\Hom(\Z^n,G)_{\BONE}$ simply by $X$. Let $H$ and $K$ be 
isotropy groups of $X$ and assume that $gHg^{-1}\leqslant K$ for some $g\in G$. 
If we can show that the map
\[
\H_2(G/H)=H_{2}(BH;\Z)\to H_{2}(BK;\Z)=\H_2(G/K)
\]
induced by conjugation and inclusion is 
surjective, then Lemma \ref{computation spectral sequence} finishes the proof, for 
$x_{0}=(1_{G},\dots, 1_{G})\in X$ is fixed by the conjugation action and 
$\H_2(G/G)=H_{2}(BG;\Z)=0$ since $BG$ is $3$-connected.

Factoring the map through the isomorphism induced by $H\to gHg^{-1}$, it suffices to prove 
surjectivity in the case where $H\leqslant K$ and $H_2(BH;\Z)\to H_2(BK;\Z)$ is induced by 
the inclusion of $H$ into $K$.

To see this we consider first the inclusion of the identity components $H_0\leqslant K_0$. 
There is a torus $T\leqslant H_0$ which is maximal for both $H_0$ and $K_0$ as both groups 
have maximal rank. By \cite[Theorem V 7.1]{BD} the maps 
$\pi_{1}(T)\to \pi_{1}(H_{0})$ and $\pi_{1}(T)\to \pi_{1}(K_{0})$ induced by the inclusions 
are both surjective. This in turn implies that the map $\pi_{1}(H_{0})\to \pi_{1}(K_{0})$ 
induced by the inclusion is surjective. From the natural isomorphisms
\[
H_{2}(BH_{0};\Z)\cong \pi_{2}(BH_{0})\cong \pi_{1}(H_{0})\,,
\]
and the corresponding ones for $K_0$, we derive that $H_{2}(BH_{0};\Z) \to H_{2}(BK_{0};\Z)$ 
is surjective.

Surjectivity of $H_2(BH;\Z)\to H_2(BK;\Z)$ follows then, because there is a commutative 
diagram with exact rows
\[
\xymatrix{
H_2(BH_0;\Z)_{\pi_0(H)} \ar[r] \ar[d] &  H_2(BH;\Z) \ar[d] \ar[r] & 0 \ar@{=}[d] \\
H_2(BK_0;\Z)_{\pi_0(K)} \ar[r] &  H_2(BK;\Z) \ar[r] & 0 
}
\]
resulting from the Serre exact sequence of the homotopy fiber sequence
\[
BK_{0}\to BK\to B\pi_{0}(K)\,,
\]
and the corresponding one for $H$. In more detail, consider the Serre spectral sequence 
\[
\tilde{E}^{2}_{p,q}=H_{p}(\pi_{0}(K);H_{q}(BK_{0};\Z))\Longrightarrow H_{p+q}(BK;\Z)\,.
\]
As $BK_0$ is path--connected and simply--connected, $H_0(BK_0;\Z)\cong \Z$ is the trivial 
$\pi_0(K)$-module and $H_1(BK_0;\Z)= 0$. In particular, both $\tilde{E}^{2}_{1,1}$ and 
$\tilde{E}_{2,1}^2$ vanish.

Now observe that $\pi_0(K)$ is a cyclic group; if $n=2$, then this is Lemma \ref{lem:componentsisotropy}, and 
if $G$ is either $SU(m)$ or $Sp(m)$, then $K$ is connected by Lemma \ref{lem:connectedisotropy}. As a consequence, we have that $H_2(\pi_0(K);\Z)= 0$ and, therefore, $\tilde{E}^{2}_{2,0}=0$. Furthermore, we can identify $\tilde{E}_{0,2}^{2}$ with the 
coinvariants $H_{2}(BK_{0};\Z)_{\pi_{0}(K)}$. The vanishing of $\tilde{E}_{2,1}^{2}$ implies 
that $\tilde{E}_{0,2}^{3}=\tilde{E}_{0,2}^{2}$. It follows that
\[
\tilde{E}^{4}_{0,2}= H_{2}(BK_{0};\Z)_{\pi_{0}(K)}/\Im(d_{3}\co\tilde{E}_{3,0}^{3}\to \tilde{E}_{0,2}^{3})\,.
\] 
For degree reasons, the differentials from or to $\tilde{E}^{r}_{0,2}$ are zero when $r\geqslant 4$, 
which implies that $\tilde{E}^{\infty}_{0,2}=\tilde{E}^{4}_{0,2}$. As a result, there is an exact sequence
\[
0\to \Im(d_{3}\co\tilde{E}_{3,0}^{3}\to \tilde{E}_{0,2}^{3}) 
\to H_{2}(BK_{0};\Z)_{\pi_{0}(K)} \to H_2(BK;\Z)\to 0\,,
\]
and a corresponding exact sequence for $H$. The diagram is obtained by naturality of the Serre 
spectral sequence.
\end{proof}

As a consequence of Proposition \ref{prop:vanishinge02} we obtain the following theorem, which was stated in the introduction in terms of $\pi_2$ rather than $H_2$ (the two statements being equivalent by the Hurewicz theorem).

\begin{theorem}\label{thm:generalsecondhomology}
Suppose that $G$ is $SU(m)$ or $Sp(m)$ with $m\geqslant 1$. For every $n\geqslant 1$ 
the quotient map $\pi\co \Hom(\Z^{n},G)\to \Rep(\Z^{n},G)$ induces an isomorphism 
\[
H_{2}(\Hom(\Z^{n},G);\Z)\cong H_{2}(\Rep(\Z^{n},G);\Z)\, .
\]
\end{theorem}
\begin{proof}
Let $X:=\Hom(\Z^{n},G)$ and consider the Bredon spectral sequence
\[
E^{2}_{p,q}=H^{G}_{p}(X;\H_{q})\Longrightarrow H_{p+q}(EG\times_G X;\Z)\,.
\]
As pointed out before, the map we aim to show is an isomorphism is identified with the composite
\[
H_2(EG\times_G X;\Z)\to E_{2,0}^\infty \hookrightarrow E_{2,0}^2\, .
\]
By Lemma \ref{lem:connectedisotropy}, $G$ acts on $X$ with connected isotropy groups, hence $E^{2}_{p,1}=H^G_{p}(X;\H_1)=0$ 
for all $p\geqslant 0$. This implies that $d_2\co E^2_{2,0}\to E^2_{0,1}$ is 
trivial, so that $E^{\infty}_{2,0}\cong E^{2}_{2,0}$. Now $E^2_{0,2}$ vanishes by Proposition 
\ref{prop:vanishinge02}, whence the only possibly non-zero term of total degree two on the 
$E^\infty$-page is $E^\infty_{2,0}$. Consequently, the projection 
$H_2(EG\times_G X;\Z)\to E^{\infty}_{2,0}$ is an isomorphism.
\end{proof}

As a result, the computation of $\pi_2(\Hom(\Z^n,G))$ for $G=SU(m)$ or $G=Sp(m)$ is reduced to 
the calculation of $H_2(\Rep(\Z^n,G);\Z)$. This will be adressed in Section \ref{sec:sumspm}.

\medskip

For the case of commuting pairs we record an intermediate result in the following lemma.

\begin{lemma} \label{lem:pretheorem}
Let $G$ be a simply--connected and simple compact Lie group. Then
\[
\pi_{2}(\Hom(\Z^{2},G))\cong \Z\oplus E^2_{1,1}\,,
\]
where $E^2_{1,1}=H_1^G(\Hom(\Z^2,G);\mathcal{H}_1)$ is a finite group. Moreover, the image of
\[
\pi_\ast\co \pi_2(\Hom(\Z^2,G))\to \pi_2(\Rep(\Z^2,G))
\]
is a finite index subgroup $d\Z\subseteq \Z$ where $d$ equals the order of 
$E^2_{0,1}=H_0^G(\Hom(\Z^2,G);\mathcal{H}_1)$.
\end{lemma}
\begin{proof}
In the Bredon spectral sequence for $\Hom(\Z^2,G)$ we have that $E^2_{2,0}\cong \Z$ (see Figure \ref{figure:e2page}), and $E^{2}_{0,2}=0$ by Proposition \ref{prop:vanishinge02}. Because $EG\times_G \Hom(\Z^2,G)$ is simply--connected, we must have $E^\infty_{0,1}= 0$. The only way this can happen is by $d_2\co E^2_{2,0}\to E^2_{0,1}$ being surjective.

As all isotropy groups of $\Hom(\Z^2,G)$ are compact, the coefficient system $\mathcal{H}_1$ 
takes values in finite abelian groups. Since $\Hom(\Z^2,G)$ is compact, 
$E^2_{0,1}=H_0^G(\Hom(\Z^2,G);\mathcal{H}_1)$ is finitely generated. Together this 
implies that $E^2_{0,1}$ is finite. Therefore, the subgroup $E^\infty_{2,0}\leqslant E^2_{2,0}$ is $\ker(d_2)\cong d\Z$, 
where $d$ is the order of the cyclic group $E^2_{0,1}$.

Finally, $E^{\infty}_{1,1}\cong E^2_{1,1}$ because $E^2_{3,0}=0$. Since $E^{\infty}_{2,0}=\ker(d_2)\cong \Z$ is free, we 
have that $\pi_2(\Hom(\Z^2,G))\cong E^\infty_{2,0}\oplus E^{2}_{1,1}$.
\end{proof}

The groups $H_p^G(\Hom(\Z^2,G);\H_1)$, $p\geqslant 0$ will be calculated in Section \ref{sec:pairs}.

\section{Commuting $n$-tuples in $SU(m)$ and $Sp(k)$} \label{sec:sumspm}

In this section we focus on the cases $G=SU(m)$ and $G=Sp(k)$. Since $\Hom(\Z^n,G)$ and $\Rep(\Z^n,G)$ are path--connected for all $n\geqslant 1$, we will omit the subscript $\mathds{1}$ in this section.

\begin{theorem}\label{thm:casesymplecticandunitary}
Let $n\geqslant 1$. Then
\begin{itemize}
\item[(i)] for every $m\geqslant 3$ there is an isomorphism
\[
\pi_{2}(\Hom(\Z^{n},SU(m))) \cong \Z^{\binom{{n}}{{2}}}\,,
\]
and the standard inclusion $SU(m)\to SU(m+1)$ induces an isomorphism 
\[
\pi_2(\Hom(\Z^n,SU(m))) \xrightarrow{\cong} \pi_2(\Hom(\Z^n,SU(m+1)))\,,
\]
\item[(ii)] for every $k\geqslant 1$ there is an isomorphism
\[
\pi_2(\Hom(\Z^n,Sp(k))) \cong \Z^{\binom{{n}}{{2}}}\oplus (\Z/2)^{2^{n}-1-n-\binom{{n}}{{2}}}\,,
\]
and the standard inclusion $Sp(k)\to Sp(k+1)$ induces an isomorphism 
\[
\pi_2(\Hom(\Z^n,Sp(k)))\xrightarrow{\cong} \pi_2(\Hom(\Z^n,Sp(k+1)))\, .
\]
\end{itemize}
\end{theorem}
\begin{proof}
Let us first consider the quaternionic unitary groups $Sp(k)$ with $k\geqslant 1$. 
Because of Theorem \ref{thm:generalsecondhomology} it suffices to calculate 
$H_{2}(\Rep(\Z^{n},Sp(k));\Z)$. By \cite[Proposition 6.3]{ACG} there is 
a homeomorphism $\Rep(\Z^{n},Sp(k))\cong SP^{k}((\SS^{1})^{n}/(\Z/2))$, 
where $SP^k$ denotes the $k$-th symmetric power. Here $\Z/2$ acts by complex 
conjugation on $\SS^{1}\subseteq \C$ and diagonally on $(\SS^{1})^{n}$. Under this 
identification the inclusion  $Sp(k)\to Sp(k+1)$ given by block sum with a $1\times 1$ 
identity matrix induces the natural 
inclusion $SP^{k}((\SS^{1})^{n}/(\Z/2))\to SP^{k+1}((\SS^{1})^{n}/(\Z/2))$. 
For the space $(\SS^{1})^{n}/(\Z/2)$ we know by 
\cite[Proposition 6.9]{ACG} that 
\[
H_{1}((\SS^{1})^{n}/(\Z/2);\Z)=0 \ \text{ and } \
H_{2}((\SS^{1})^{n}/(\Z/2);\Z)\cong \Z^{\binom{{n}}{{2}}}\oplus 
(\Z/2)^{2^{n}-1-n-\binom{{n}}{{2}}}\,.
\]
This together with the theorem of Dold and Thom imply that $SP^{\infty}((\SS^{1})^{n}/(\Z/2))$ 
is a simply--connected space, and 
\[
H_{2}(SP^{\infty}((\SS^{1})^{n}/(\Z/2));\Z)\cong \pi_{2}(SP^{\infty}((\SS^{1})^{n}/(\Z/2)))
\cong \Z^{\binom{{n}}{{2}}}\oplus (\Z/2)^{2^{n}-1-n-\binom{{n}}{{2}}}.
\] 
By Steenrod's splitting \cite[Section 22]{Steenrod} the map 
$SP^{k}((\SS^{1})^{n}/(\Z/2))\to SP^{k+1}((\SS^{1})^{n}/(\Z/2))$ is split injective in homology 
for every $k\geqslant 1$. As the groups $H_{2}(SP^{k}((\SS^{1})^{n}/(\Z/2));\Z)$ for $k=1$ and 
$k=\infty$ agree, each of the maps
\begin{equation*}
H_{2}(SP^{k}((\SS^{1})^{n}/(\Z/2));\Z)\to H_{2}(SP^{k+1}((\SS^{1})^{n}/(\Z/2));\Z)
\end{equation*}
must be an isomorphism. This proves part (ii) of the theorem.

Consider now the special unitary group $SU(m)$.  By \cite[Proposition 4.4]{LR15} the space 
$\Rep(\Z^n,SU(m))$ is homotopy equivalent to the universal cover of 
$\Rep(\Z^n,U(m))$. There is a homeomorphism $\Rep(\Z^n,U(m))\cong SP^m((\mathbb{S}^1)^n)$, 
so that we obtain
\[
\pi_2(\Rep(\Z^n,SU(m)))\cong \pi_2(SP^m((\mathbb{S}^1)^n))\, .
\]
By the Dold-Thom theorem, there is an isomorphism
\[
\pi_2(SP^\infty((\SS^1)^n))\cong H_2((\SS^1)^n;\Z)\cong \Z^{\binom{n}{2}}\, .
\]
To prove part (i) of the theorem it suffices to show that for every $m\geqslant 3$ the natural 
map $SP^m((\SS^1)^n)\to SP^{\infty}((\SS^1)^n)$ induces an isomorphism of second homotopy groups. 
But this follows at once from the unpublished preprint \cite[Section 8]{Tripathyd}, 
where it is shown that for a connected based CW complex $X$ and every $k\geqslant 0$ the map 
$\pi_k(SP^m(X))\to \pi_k(SP^{m+1}(X))$ is an isomorphism whenever $m> k$.
\footnote{For the particular case $X=(\mathbb{S}^1)^n$ this stabilization result can 
also be proved directly by constructing a suitable CW-structure on $SP^m((\mathbb{S}^1)^n)$. However, to avoid digression, we chose not to include a proof of this fact.}
\end{proof}

\begin{remark} \label{rem:symplecticandunitary}
As $SU(2)\cong Sp(1)$ also the case of $SU(2)$ is covered by Theorem \ref{thm:casesymplecticandunitary}. In the case of commuting pairs the inclusion 
$SU(2)\hookrightarrow SU(3)$ induces an isomorphism 
$\pi_2(\Hom(\Z^2,SU(2)))\cong \pi_2(\Hom(\Z^2,SU(3)))$. This follows from Theorem 
\ref{thm:generalsecondhomology} and the fact that the induced map of 
representation spaces $\Rep(\Z^2,SU(2))\to \Rep(\Z^2,SU(3))$ can be identified with the standard 
inclusion $\mathbb{CP}^1\subseteq \mathbb{CP}^2$ as shown in \cite[Lemma 5.4]{LR15}.
\end{remark}

\section{Commuting pairs in a simple and simply--connected group} \label{sec:pairs}

Throughout this section $G$ will denote a simply--connected and simple compact Lie group 
of rank $r\geqslant 1$, unless stated otherwise. In this case $\Hom(\Z^2,G)$ is path--connected and simply--connected, hence we will drop the subscript $\mathds{1}$ from the notation. One objective of this section is to prove Theorem \ref{thm:mainpairsintro}, which asserts that $\pi_2(\Hom(\Z^2,G))\cong \Z$ and on $\pi_2$ the map $\Hom(\Z^2,G)\to \Rep(\Z^2,G)$ induces multiplication by the Dynkin index of $G$.

In view of Lemma \ref{lem:pretheorem} the proof amounts to a 
calculation of $H_p^G(\Hom(\Z^2,G);\mathcal{H}_1)$ for $p=0,1$. Note that, by Lemma 
\ref{lem:componentsisotropy}, $\pi_0(H)$ is abelian as $H$ ranges over the isotropy groups of 
$\Hom(\Z^2,G)$. Thus,
\[
\H_1(G/Z_G(x,y))=H_1(BZ_G(x,y);\Z)\cong \pi_0(Z_G(x,y))\,,
\]
for all $(x,y)\in \Hom(\Z^2,G)$. This justifies introducing the coefficient system
\[
\underline{\pi}_0\co G/Z_G(x,y)\mapsto \pi_0(Z_G(x,y))\, ,
\]
which we will use from now on instead of $\H_1$.

\subsection{The component group $\pi_0(Z_G(x,y))$} \label{sec:componentgroup}

In order to calculate the homology groups $H_\ast^G(\Hom(\Z^2,G);\underline{\pi}_0)$ we must 
acquire a good understanding of the coefficient system $\underline{\pi}_0$. Thus, our first goal 
is to describe the group of components of the centralizer $Z_G(x,y)$ of a pair 
$(x,y)\in \Hom(\Z^2,G)$. The description will be in terms of the fundamental group of the derived 
group $DZ_G(x)$ (Lemma \ref{lem:pi0inpi1}). This description can be found in \cite{BFM}, but 
since the proofs may not be easy to find we include them here for convenience.

We begin by explaining some of our notation regarding root systems and the fundamental alcove. Let $\g_{\C}$ be the complexified Lie algebra of $G$. 
As a Cartan subalgebra of $\g_{\C}$ we choose $\t_{\C}$, the complexification of the Lie algebra 
$\t$ of the maximal torus $T$. We will work with the real roots of the root system associated with 
$(\g_{\C},\t_{\C})$. Thus the roots are $\R$-valued functionals
$\alpha:\t   \to \R$ such that the weight space
\[
\g_{\alpha}=\{Z\in \g_{\C}~|~ [H,Z]=2\pi i\alpha(H)Z \text{ for all } H\in \t_{\C}\}
\]
is non-trivial. Choose a $W$-invariant inner product $\langle\cdot,\cdot\rangle$ on $\t$. 
For each root $\alpha\in \t^\ast$ we can find a unique element $h_{\alpha}\in \t$ such that 
$\alpha(H)=\langle H,h_{\alpha}\rangle$ for every $H\in \t$. Define 
$\alpha^{\vee}=2h_{\alpha} /\langle h_{\alpha},h_{\alpha}\rangle \in \t$.  
The element $\alpha^{\vee}$ is called the coroot associated to 
the root $\alpha$ and is independent of the choice of inner product $\langle\cdot,\cdot\rangle$. 
Let  $\Delta=\{\alpha_{1},\dots,\alpha_{r}\}$ 
be a fixed set of simple roots for $(\g_\C,\t_\C)$ and $\alpha_{0}$ the lowest root. We can write 
$-\alpha_{0}^{\vee}$ in the form
\[
-\alpha_{0}^{\vee}=n_{1}^{\vee}\alpha_{1}^{\vee}+n_{2}^{\vee}\alpha_{2}^{\vee}
+\cdots+n_{r}^{\vee}\alpha_{r}^{\vee}
\]
for unique integers $n_{1}^{\vee},n_{2}^{\vee},\dots,n_{r}^{\vee}\geqslant 1$ which we call the 
\emph{coroot integers} of $G$. It will be convenient to set $n_{0}^{\vee}:=1$. The coroot integers 
$n_0^\vee,\dots,n_r^\vee$ can be found, for instance, in the appendix of \cite{BFM} and they are also listed in Table \ref{table:coroot}.

Within $\t$ we find the coroot lattice $Q^\vee$ defined as the $\Z$-span of 
$\{\alpha_1^\vee,\dots,\alpha_r^\vee\}\subseteq \t$, and the integral lattice 
$\Lambda:=\ker(\exp\co \t\to T)$ which contains $Q^\vee$. In general, these two lattices 
determine the fundamental group of $G$ by the well known formula
\[
\pi_1(G)\cong \Lambda/Q^\vee\,,
\]
see \cite[IX \S 4.9 Theorem 2(b)]{Bourbaki}. Since we assume that $G$ is simply--connected, 
we have that $\Lambda=Q^\vee$.

Let $A=A(\Delta)\subseteq \t$ be the closed alcove contained in the closed Weyl chamber determined 
by $\Delta$ and that contains $0\in \t$. The alcove $A$ is an $r$-dimensional simplex supported by 
the hyperplanes $\{\alpha_j=0\}$ for $1\leqslant j\leqslant r$ and the affine hyperplane 
$\{\alpha_{0}=-1\}$. (If $G$ were semisimple, then $A$ would be the product of the alcoves of its 
simple factors.)  As explained in \cite[IX \S 5.2 Corollary 1]{Bourbaki}, since $G$ is compact and 
simply--connected, the exponential map induces a homeomorphism 
\[
A\xrightarrow{\cong} T/W\, .
\]
In particular, every 
$x\in G$ is conjugate to an element of the form $\exp(\tilde{x})$ for a uniquely determined 
$\tilde{x}\in A$. Thus, for the purpose of describing $Z_G(x)$ we may assume that 
$x=\exp(\tilde{x})$ for some $\tilde{x}\in A$. In this case we choose $T\leqslant Z_G(x)$ as a maximal 
torus for $Z_G(x)$.

Let $\tilde{\Delta}:=\Delta\cup \{\alpha_0\}$ be the extended set of simple roots. Abusing 
notation slightly, we will sometimes regard $\tilde{\Delta}$ simply as the set of indices 
$\{0,1,\dots,r\}$ of the extended set of simple roots. Let 
$\tilde{\Delta}(x)\subseteq \tilde{\Delta}$ be the proper subset defined by
\begin{equation} \label{eq:deltax}
\tilde{\Delta}(x):=\{ \alpha\in\tilde{\Delta}\mid \tilde{x} 
\textnormal{ lies in the wall of }A\textnormal{ determined by }\alpha\}\, .
\end{equation}
Let $Q^\vee(x)\leqslant Q^\vee$ be the sublattice of the coroot lattice spanned by 
$\{\alpha^\vee_i \mid i \in \tilde{\Delta}(x)\}$. Then $\tilde{\Delta}(x)$ is a system of simple 
roots for $Z_G(x)$ relative to $T$, and $Q^\vee(x)$ is the corresponding coroot lattice, see for 
example \cite[V.2 Proposition 2.3(ii)]{BD}. As $Z_G(x)$ has the same integral lattice as $G$ we have
\begin{equation}  \label{eq:pi1zgx}
\pi_1(Z_G(x))\cong \Lambda/ Q^\vee(x)\cong Q^\vee/Q^\vee(x)\, .
\end{equation}
The following lemma is a combination of Corollary 3.1.3 and Proposition 7.6.1 of \cite{BFM}. 
A proof of the first part may be found in \cite[Theorem 1]{KS}.

\begin{lemma} \label{lem:pi1derived}
The fundamental group of $DZ_G(x)$ is isomorphic to the torsion subgroup of $Q^\vee/Q^\vee(x)$ 
and is a cyclic group of order
\[
n^\vee(x):=\textnormal{gcd}\{n_i^\vee\mid i\in \tilde{\Delta}\backslash \tilde{\Delta}(x)\}\,.
\] 
A representative in $Q^\vee$ for the generator of $\pi_1(DZ_G(x))$ is
\[
\zeta(x):= \frac{1}{n^\vee(x)} \sum_{i\in\tilde{\Delta}\backslash 
\tilde{\Delta}(x)} n^\vee_i \alpha_i^\vee \, .
\]
Moreover, if $x'=\exp(\tilde{x}')$ is such that $\tilde{x}'\in A$ and $Z_G(x)\leqslant Z_G(x')$, 
then the inclusion $Z_G(x)\to Z_G(x')$ induces an injection $\pi_1(DZ_G(x))\to \pi_1(DZ_G(x'))$ sending
\[
\zeta(x)\mapsto \frac{n^\vee(x')}{n^\vee(x)} \zeta(x')\,.
\]
\end{lemma}
\begin{proof}
By \cite[IX \S 4.6 Corollary 3]{Bourbaki} the inclusion $DZ_G(x)\to Z_G(x)$ induces an isomorphism of 
$\pi_1(DZ_G(x))$ onto the torsion subgroup of $\pi_1(Z_G(x))\cong Q^\vee/Q^\vee(x)$.

Writing the coroot lattice as 
$Q^\vee=(\bigoplus_{i\in\tilde{\Delta}}\Z\langle \alpha_i^\vee\rangle)/\Z\langle 
\sum_{i\in\tilde{\Delta}}n_i^\vee\alpha_i^\vee\rangle$, we may write $Q^\vee/Q^\vee(x)$ in the form
\[
Q^\vee/Q^\vee(x)\cong \left(\bigoplus_{i\in \tilde{\Delta}\backslash 
\tilde{\Delta}(x)}\Z\langle \alpha_i^\vee\rangle \right) / \Z\langle\sum_{i\in\tilde{\Delta}
\backslash\tilde{\Delta}(x)}n_i^\vee \alpha^\vee_i\rangle =  
\left(\bigoplus_{i\in \tilde{\Delta}\backslash \tilde{\Delta}(x)}\Z\langle \alpha_i^\vee\rangle \right) 
/ \Z\langle n^\vee(x)\zeta(x)\rangle\, .
\]
Since $\zeta(x)$ is a linear combination of 
$\{\alpha^\vee_i \mid i\in\tilde{\Delta}\backslash\tilde{\Delta}(x)\}$ 
with coprime coefficients, it can be completed to a basis of 
$\bigoplus_{i\in\tilde{\Delta}\backslash\tilde{\Delta}(x)}\Z\langle \alpha_i^\vee\rangle$. In this 
basis it becomes obvious that 
$Q^\vee/Q^\vee(x)\cong \Z^{|\tilde{\Delta}\backslash\tilde{\Delta}(x)|-1}\oplus \Z/n^\vee(x)$, 
where $\Z/n^\vee(x)$ is generated by the image of $\zeta(x)$.

For the second part notice that $Z_G(x)\leqslant Z_G(x')$ implies that every wall of $A$ 
containing $\tilde{x}$ also contains $\tilde{x}'$, hence $\tilde{\Delta}(x)\subseteq \tilde{\Delta}(x')$. 
Therefore, $Q^\vee(x)\leqslant Q^\vee(x')$ and by naturality of the isomorphism (\ref{eq:pi1zgx}) 
the map $\pi_1(Z_G(x))\to \pi_1(Z_G(x'))$ corresponds to the projection $
Q^\vee/Q^\vee(x)\to Q^\vee/Q^\vee(x')$. On the torsion subgroups this projection maps
\[
\zeta(x)\mapsto \frac{1}{n^\vee(x)}\sum_{i\in\tilde{\Delta}
\backslash\tilde{\Delta}(x')}n_i^\vee\alpha_i^\vee=\frac{n^\vee(x')}{n^\vee(x)}\zeta(x')\,,
\]
which is an injection $\Z/n^\vee(x)\to \Z/n^\vee(x')$.
\end{proof}

Let $(x,y)\in \Hom(\Z^2,G)$ and let $\tilde{y}$ denote a lift of $y\in Z_G(x)$ in the universal 
covering group $\widetilde{Z_G(x)}$. The component group $\pi_0(Z_G(x,y))$ may be described as follows.

\begin{lemma} \label{lem:pi0inpi1}
For $z\in Z_G(x,y)$ let $[z]\in \pi_0(Z_G(x,y))$ denote the path component determined by $z$. 
Let $\tilde{z}$ be any lift of $z$ in $\widetilde{Z_G(x)}$. Then the map
\[
\delta_y \co \pi_0(Z_G(x,y)) \to \pi_1(DZ_G(x)) \leqslant Z(\widetilde{DZ_G(x)})
\]
defined by $\delta_y([z])=[\tilde{y},\tilde{z}]$ is an injective group homomorphism.
\end{lemma}
\begin{proof}
We follow \cite[Section 7.3]{BFM}. Consider the universal covering sequence
\[
1\to Q^\vee/Q^\vee(x) \to \widetilde{Z_G(x)}\to Z_G(x)\to 1\, .
\]
The sequence is acted upon by the cyclic group $\langle \tilde{y}\rangle$ through conjugation 
by $\tilde{y}$ on $\widetilde{Z_G(x)}$ and by $y$ on $Z_G(x)$, leaving invariant the central 
subgroup $Q^\vee/Q^\vee(x)$. Passing to fixed points and noting that 
$Z_G(x)^{\langle \tilde{y}\rangle}\cong Z_G(x,y)$ yields the exact sequence
\[
1\to Q^\vee/Q^\vee(x)\to \widetilde{Z_G(x)}^{\langle\tilde{y}\rangle} \to 
Z_G(x,y)\xrightarrow{\delta} Q^\vee/Q^\vee(x)\,,
\]
where the connecting homomorphism $\delta$ is defined by 
$\delta(z)=\tilde{z}^{\tilde{y}}\tilde{z}^{-1}=[\tilde{y},\tilde{z}]$.

Now $\widetilde{Z_G(x)}^{\langle\tilde{y}\rangle}$ is connected, by \cite[IX \S 5.3 Corollary 1]{Bourbaki}, 
because it is the centralizer of $\tilde{y}$ in the simply--connected group $\widetilde{Z_G(x)}$. 
Therefore, $\widetilde{Z_G(x)}^{\langle\tilde{y}\rangle}$ maps to the identity component of $Z_G(x,y)$. 
Since $Q^\vee/Q^\vee(x)$ is discrete, and by exactness, $\delta$ descends to an injective map 
$\delta_y\co \pi_0(Z_G(x,y))\to Q^\vee/Q^\vee(x)$.

To finish the proof, note that $\pi_0(Z_G(x,y))$ is finite, since $Z_G(x,y)$ is compact, so 
$\delta_y$ factors through the torsion subgroup of $Q^\vee/Q^\vee(x)$. By Lemma \ref{lem:pi1derived}, 
the latter is identified with $\pi_1(DZ_G(x))$.
\end{proof}

Lemma \ref{lem:componentsisotropy} is now immediate:

\begin{proof}[Proof of Lemma \ref{lem:componentsisotropy}]
Let $(x,y)\in \Hom(\Z^2,G)$. By Lemmas \ref{lem:pi1derived} and \ref{lem:pi0inpi1}, $\pi_0(Z_G(x,y))$ 
is a subgroup of a cyclic group of order $n^\vee(x)$. A look at the coroot diagrams in the appendix of 
\cite{BFM} (or at Table \ref{table:coroot}) shows that $1\leqslant n^\vee(x) \leqslant 6$. 
But then $\pi_0(Z_G(x,y))$ must also be cyclic, and of order at most six.
\end{proof}

For later use we record a further consequence of the preceding lemmas, a special case of 
\cite[Corollary 7.6.2]{BFM}.

\begin{lemma} \label{lem:pi0injective}
Let $x=\exp(\tilde{x})$ and $x'=\exp(\tilde{x}')$ for some $\tilde{x},\tilde{x}'\in A$, 
and let $y,y'\in T$. Suppose that $Z_G(x)\leqslant Z_G(x')$ and $Z_G(x,y)\leqslant Z_G(x',y')$. 
Then the map $\pi_0(Z_G(x,y))\to \pi_0(Z_G(x',y'))$ induced by the inclusion is injective.
\end{lemma}
\begin{proof}
By naturality of the connecting homomorphism, there is a commutative diagram
\[
\xymatrix{
\pi_0(Z_G(x,y)) \ar[r]^-{\delta_y} \ar[d] & \pi_1(DZ_G(x)) \ar[d] \\
\pi_0(Z_G(x',y')) \ar[r]^-{\delta_{y'}} & \pi_1(DZ_G(x'))
}
\]
in which the left hand vertical map is induced by the inclusion $Z_G(x,y)\leqslant Z_G(x',y')$, 
and the one on the right is induced by the inclusion $Z_G(x)\to Z_G(x')$. The assertion follows, 
because $\delta_y$ is injective by Lemma \ref{lem:pi0inpi1}, and $\pi_1(DZ_G(x))\to \pi_1(DZ_G(x'))$ 
is injective by Lemma \ref{lem:pi1derived}.
\end{proof}

\subsection{Equivariant cell structure} \label{sec:equivariant}

We will now describe $\Hom(\Z^2,G)$ as a $G$-equivariant CW-complex. This will enable us to compute 
the $p$-localization of $H_\ast^G(\Hom(\Z^2,G);\underline{\pi}_0)$ as the homology of a certain 
$G$-subcomplex of $\Hom(\Z^2,G)$, see Corollary \ref{cor:decomposition}.

The $G$-CW-structure on $\Hom(\Z^2,G)$ is obtained from the simplicial structure of the Weyl alcove 
$A$ as follows. Recall that $\Rep(\Z^2,G)\cong T^2/W$ and $T/W\cong A$. Let
\[
p_i\co \Hom(\Z^2,G)\to A\,,\quad i=1,2
\]
be the composition of the quotient map $\pi\co \Hom(\Z^2,G)\to \Rep(\Z^2,G)$ and the projection 
onto the $i$-th component $T^2/W \to A$. Let $\mathscr{F}_n$, $n=0,\dots,r$, denote the set of 
$n$-dimensional faces of $A$.

$A$ has a standard CW-structure whose set of $n$-cells is $\mathscr{F}_n$. Let $(A\times A)^{(n)}$ 
be the $n$-skeleton of $A\times A$ in the product CW-structure. Define
\[
\Hom(\Z^2,G)^{(n)}:=(p_1\times p_2)^{-1}((A\times A)^{(n)})\, .
\]
This defines an increasing sequence of $G$-spaces
\begin{equation} \label{eq:gcw}
\Hom(\Z^2,G)^{(0)} \subseteq \Hom(\Z^2,G)^{(1)}\subseteq \cdots \subseteq 
\Hom(\Z^2,G)^{(2r)}=\Hom(\Z^2,G)
\end{equation}
such that $\Hom(\Z^2,G)^{(n)}$ is obtained from $\Hom(\Z^2,G)^{(n-1)}$ by attaching a set of 
equivariant $n$-cells, as we now explain.

\begin{notation*}
To simplify the notation, we identify $A$ with a subset of $T$ without making the exponential map 
explicit. For a face $\sigma\subseteq A$ we let $b(\sigma)\in \sigma$ denote its barycenter. 
Since the centralizer $Z_G(x)$ of some $x\in G$ equals the stabilizer of $x$ under the conjugation 
action of $G$ on itself, we may write $G_x$ instead of $Z_G(x)$. Moreover, since a face 
$\sigma\subseteq A$ is pointwise fixed if and only if $b(\sigma)$ is fixed, we shall write 
$G_{\sigma}$ for $G_{b(\sigma)}$. Similarly, we write
\begin{flushleft}
\def\arraystretch{1.2}
\begin{tabular}{rl}
$W_{\sigma}$ & for the isotropy group of $b(\sigma)\in T$ under the action of $W$, \\
$G_{(x,y)}$ & for the centralizer $Z_G(x,y)$, \\
$G_{(\sigma,w\tau)}$ & for $G_{(b(\sigma),wb(\tau))}$ where $w\in W$ and $\sigma,\tau$ are faces of $A$.
\end{tabular}
\end{flushleft}
\end{notation*}

For each pair of faces $\sigma,\tau$ of $A$, let $\mathscr{C}(\sigma,\tau)$ denote a complete set of 
representatives for the double cosets $W_{\sigma}\backslash W /W_{\tau}$. For $n\geqslant 0$ the 
indexing set $J_n$ of the $G$-$n$-cells is
\[
J_n=\bigsqcup_{\substack{(\sigma,\tau)\in \mathscr{F}_i\times \mathscr{F}_j \\ i+j=n}} 
\mathscr{C}(\sigma,\tau)\,.
\]

The $G$-$n$-cells $\{e^n_\alpha\mid \alpha \in J_n\}$ are built from the faces of the alcove $A$ 
in the following fashion. Given $\alpha=(\sigma,\tau,w)\in J_n$, the closed $G$-$n$-cell 
$e_\alpha^n\subseteq \Hom(\Z^2,G)$ is of the form 
$e^n_{\alpha}=\phi^{n}_\alpha(G/G_{(\sigma,w \tau)}\times \sigma\times \tau)$ where the characteristic 
map $\phi_\alpha^n$ is given by
\begin{alignat*}{2}
&& \phi_\alpha^n\co G/G_{(\sigma,w\tau)}\times \sigma\times \tau & \to \Hom(\Z^2,G) \\
&& (gG_{(\sigma,w\tau)},x,y) & \mapsto (x,wy)^g\,.
\end{alignat*}
Here the superscript $g$ indicates simultaneous conjugation by $g$.

In the definition of $\phi^n_\alpha$ it must be checked that the right hand side is independent of 
the choice of representative for the coset $gG_{(\sigma,w\tau)}$. To see this notice that 
$G_\sigma\leqslant G_x$ and $G_{w\tau}=G_\tau^{\tilde{w}}\leqslant G_y^{\tilde{w}}=G_{wy}$ where 
$\tilde{w}\in N_G(T)$ is a lift of $w\in W$. Therefore,
\begin{equation} \label{eq:isotropyboundary}
G_{(\sigma,w\tau)}=G_\sigma \cap G_{w\tau}\leqslant G_x\cap G_{wy}=G_{(x,wy)}\,,
\end{equation}
showing that $\phi^n_\alpha$ is well defined.

\begin{lemma} \label{lem:cw}
The filtration (\ref{eq:gcw}) is a $G$-CW-structure on $\Hom(\Z^2,G)$ whose set of $G$-$n$-cells is 
$\{e_\alpha^n\mid \alpha\in J_n\}$.
\end{lemma}
\begin{proof}
Set $\Hom(\Z^2,G)^{(-1)}:=\emptyset$ and assume $n\geqslant 0$. Let $P$ denote the pushout of
\[
\bigsqcup_{(\sigma,\tau,w)\in J_n} G/G_{(\sigma,w\tau)}\times\sigma\times\tau 
\longleftarrow \bigsqcup_{(\sigma,\tau,w)\in J_n}  G/G_{(\sigma,w\tau)}\times 
\partial(\sigma\times \tau) \xrightarrow{\sqcup_{J_n}f^n_\alpha}  \Hom(\Z^2,G)^{(n-1)}
\]
where the attaching maps $\{f^n_\alpha \mid \alpha\in J_n\}$ arise as the restriction of the 
characteristic maps $\{\phi^n_\alpha \mid \alpha\in J_n\}$ to the boundary of the complex 
$\sigma\times \tau$. We must show that the map
\[
h\co P\to \Hom(\Z^2,G)^{(n)}
\]
induced by the characteristic maps and the inclusion 
$\Hom(\Z^2,G)^{(n-1)}\hookrightarrow \Hom(\Z^2,G)^{(n)}$ is a homeomorphism. In fact, as $P$ is 
compact and $\Hom(\Z^2,G)^{(n)}$ is Hausdorff, it is enough to show $h$ is a bijection.

Clearly, the image of $h$ contains $\Hom(\Z^2,G)^{(n-1)}\subseteq \Hom(\Z^2,G)^{(n)}$. 
Now suppose that $(z_1,z_2)\in \Hom(\Z^2,G)^{(n)}\backslash \Hom(\Z^2,G)^{(n-1)}$. Since $G$ is 
assumed simply--connected, there exists $g\in G$ such that $(z_1,z_2)=(x,w'y)^g$ for some $w'\in W$ 
and uniquely determined $x,y\in A$. Furthermore, there are unique faces $\sigma\in\mathscr{F}_i$ 
and $\tau\in \mathscr{F}_{n-i}$ such that $(x,y)$ is in the relative interior of $\sigma\times \tau$. 
Now suppose that $w\in \mathscr{C}(\sigma,\tau)$ represents the double coset determined by $w'$. 
Then $w'=awb$ for $a\in W_{\sigma}$ and $b\in W_{\tau}$. Thus, $(x,w'y)=(x,wy)^{\tilde{a}}$ where 
$\tilde{a}\in N_G(T)$ is a lift of $a$. Now 
$\phi^n_{(\sigma,\tau,w)}\co (g\tilde{a}G_{(\sigma,w\tau)},x,y)\mapsto (z_1,z_2)$ showing that 
$h$ is surjective.

Now suppose that $p,p'\in P$ and $h(p)=h(p')$. If either $p$ or $p'$ is represented by an element 
of $\Hom(\Z^2,G)^{(n-1)}$ then so is the other. It follows that $p=p'$ as 
$\Hom(\Z^2,G)^{(n-1)}\hookrightarrow \Hom(\Z^2,G)^{(n)}$ is injective.

If neither $p$ nor $p'$ lifts to $\Hom(\Z^2,G)^{(n-1)}$, then $h(p)$ and $h(p')$ each lie in 
the image of a charactersitic map. This means that there are 
$(\sigma,\tau,w),(\sigma',\tau',w')\in J_n$ and $g,g'\in G$ such that $h(p)=(x,wy)^g$ and 
$h(p')=(x',w'y')^{g'}$ for some $(x,y)\in\textnormal{int}(\sigma\times \tau)$ and 
$(x',y')\in\textnormal{int}(\sigma'\times \tau')$ (where $\textnormal{int}$ denotes the 
relative interior of a cell). Then
\[
(x,wy)\equiv (x',w'y') \textnormal{ modulo }W\,,
\]
which implies that $x=x'$ and $y=y'$ (by projecting to $A\times A$), and further that $\sigma=\sigma'$ 
and $\tau=\tau'$ as every point of $A\times A$ lies in the relative interior of a unique cell. Let 
$w''\in W$ be such that $(x,wy)=(w''x,w''w'y)$. Then $w''\in W_{\sigma}$ and $w^{-1}w''w'\in W_{\tau}$. 
This implies that $w\in \mathscr{C}(\sigma,\tau)$ and $w'\in \mathscr{C}(\sigma,\tau)$ represent 
the same double coset, hence $w=w'$. Finally, $(x,wy)^g=(x,wy)^{g'}$ implies that $g\equiv g'$ 
modulo $G_{(\sigma,w\tau)}$, and therefore $p=p'$. It follows that $h$ is injective.
\end{proof}

\begin{remark} \label{rem:cwgeneral}
In the same way, one can construct a $G$-CW-structure on $\Hom(\Z^k,G)_{\BONE}$ for any $k\geqslant 1$. 
The $G$-$n$-cells are then indexed over $k$-tuples 
$(\sigma_1,\dots,\sigma_k)\in \mathscr{F}_{i_1}\times \cdots \times \mathscr{F}_{i_k}$ such that 
$\sum_j i_j=n$ and a complete set of representatives $\mathscr{C}(\sigma_1,\dots,\sigma_k)$ for the 
$W$-orbits of the diagonal $W$-set $W/W_{\sigma_1}\times \cdots \times W/W_{\sigma_k}$. 
(When $k=2$ the latter reduces to a set of representatives for the double cosets 
$W_{\sigma_1}\backslash W / W_{\sigma_2}$.) This, in fact, gives the indexing set of the $n$-cells 
in a (non-equivariant) CW-structure on $T^k/W$. A proof analogous to that of Lemma \ref{lem:cw} 
then shows that this CW-structure can be lifted to a $G$-equivariant CW-structure on 
$\Hom(\Z^k,G)_{\BONE}$.
\end{remark}

Let $\pi_0(G_{(x,y)})_{(p)}=\pi_0(G_{(x,y)})\otimes_{\Z} \Z_{(p)}$ denote the localization of the 
abelian group $\pi_0(G_{(x,y)})$ (see Lemma \ref{lem:componentsisotropy}) at $p$. Relative to the 
CW-structure just described we have:
\begin{lemma} \label{lem:subcomplex}
Let $p$ be a prime. The subspace $X_G(p)\subseteq \Hom(\Z^2,G)$ defined by
\[
X_G(p):=\{(x,y)\in \Hom(\Z^2,G)\mid \pi_0(G_{(x,y)})_{(p)}\neq 0\}
\]
is a $G$-subcomplex of $\Hom(\Z^2,G)$.
\end{lemma}
\begin{proof}
It is clear that $X_G(p)$ is a union of open $G$-cells. We must show that it is also a union of 
closed $G$-cells. Let $\alpha=(\sigma,\tau,w)\in J_n$ and write
\[
\partial e^n_\alpha:=e^n_\alpha\cap \Hom(\Z^2,G)^{(n-1)}\,,\quad\; 
\textnormal{int}(e^n_\alpha):=e^n_\alpha\backslash \partial e^n_\alpha\,.
\]
Suppose that $\textnormal{int}(e^n_\alpha)\subseteq X_G(p)$. This means that 
$\pi_0(G_{(\sigma,w\tau)})_{(p)}\neq 0$. We are going to show that 
$\partial e^n_\alpha\subseteq X_G(p)$.

Suppose that $(z_1,z_2)\in \partial e^n_\alpha$. Then there is $g\in G$ such that 
$(z_1,z_2)=(x,wy)^g$ where $(x,y)\in \partial(\sigma\times \tau)$. In particular, 
$\pi_0(G_{(z_1,z_2)})\cong \pi_0(G_{(x,wy)})$.

As noted in (\ref{eq:isotropyboundary}) we have that $G_\sigma\leqslant G_x$ and 
$G_{(\sigma,w\tau)}\leqslant G_{(x,wy)}$. By Lemma \ref{lem:pi0injective}, the map 
$\pi_0(G_{(\sigma,w\tau)})\to \pi_0(G_{(x,wy)})$ induced by the inclusion is injective. 
The map remains injective after $p$-localization, hence $\pi_0(G_{(x,wy)})_{(p)}\neq 0$. 
But this implies that $(z_1,z_2)\in X_G(p)$, which finishes the proof.
\end{proof}

Let us write $(\underline{\pi}_0)_{(p)}$ for the coefficient system obtained by localizing 
$\underline{\pi}_0$ objectwise at $p$.

\begin{corollary} \label{cor:decomposition}
Let $\mathcal{P}$ denote the set of primes which divide at least one coroot integer of $G$. 
Then there is an isomorphism
\[
H_\ast^G(\Hom(\Z^2,G);\underline{\pi}_0)\cong 
\bigoplus_{p\in \mathcal{P}} H_\ast^G(X_G(p);(\underline{\pi}_0)_{(p)})\, .
\]
\end{corollary}
\begin{proof}
We know from Lemma \ref{lem:componentsisotropy} that $\underline{\pi}_0$ takes values in 
finite abelian groups. Hence, it splits as a direct sum of its localizations at the various primes.

On the other hand, we derive from Lemmas \ref{lem:pi1derived} and \ref{lem:pi0inpi1} that 
$(\underline{\pi}_0)_{(p)}$ is trivial unless $p$ divides a coroot integer of $G$. Therefore, 
$\underline{\pi}_0\cong \bigoplus_{p\in\mathcal{P}} (\underline{\pi}_0)_{(p)}$.

The corollary is now a consequence of Lemma \ref{lem:subcomplex}; by definition of $X_G(p)$, 
the inclusion of the chain complex for $H_\ast(X_G(p);(\underline{\pi}_0)_{(p)})$ into the chain 
complex for $H_\ast(\Hom(\Z^2,G);(\underline{\pi}_0)_{(p)})$ is an isomorphism.
\end{proof}

To understand why the decomposition of Corollary \ref{cor:decomposition} is useful, we must take a 
closer look at the coefficient system $(\underline{\pi}_0)_{(p)}$.

\begin{lemma} \label{lem:equivarianthomologyquotient}
Let $p\in \mathcal{P}$. Suppose that $G\notin\{E_7,E_8\}$ or that $p>2$. Then
\[
H_\ast^G(X_G(p);(\underline{\pi}_0)_{(p)}) \cong H_\ast(R_G(p);\Z/p)
\]
where $R_G(p):=X_G(p)/G$.
\end{lemma}
\begin{proof}
Let $G$ and $p$ be fixed. The assumption that either $G\notin \{E_7,E_8\}$ or that $p>2$ will only 
become relevant towards the end of the proof.

Let $\mathcal{O}_G$ denote the orbit category of $G$, whose objects are the homogeneous spaces $G/H$ for closed subgroups $H\leqslant G$ and whose morphisms are the $G$-equivariant maps. For the construction of the chain complex $C_\ast(X_G(p);\underline{\pi}_0)$ only a subcategory 
$\mathcal{O}_G'$ of $\mathcal{O}_G$ is relevant, namely the subcategory generated 
by the equivariant maps that enter into the definition of the differentials. To prove the lemma 
it suffices to show that under the stated assumptions the restriction of $(\underline{\pi}_0)_{(p)}$ 
to $\mathcal{O}_G'$ is naturally isomorphic to the constant coefficient system $\underline{\Z/p}$.

The maps defining the differential $C_n(X_G(p);\underline{\pi}_0)\to C_{n-1}(X_G(p);\underline{\pi}_0)$ 
arise by considering the composite of an attaching map $f^n_\alpha$ for an $n$-cell $e^n_\alpha$ and 
the collapse of the complement of an $(n-1)$-cell $e^{n-1}_\beta$ into a point \cite{willson}. 
In our situation this looks as follows. Let $\alpha=(\sigma,\tau,w)\in J_n$, let 
$\beta=(\sigma',\tau',w')\in J_{n-1}$, and let $e^n_\alpha$ respectively $e^{n-1}_\beta$ be the 
corresponding closed cells of $\Hom(\Z^2,G)$. The pair $(e^n_\alpha,e^{n-1}_\beta)$ can contribute 
to the differential only if $\textnormal{int}(e^{n-1}_\beta)\cap \partial e^n_\alpha\neq \emptyset$. 
Let us assume that the intersection is non-empty. It follows, then, that 
$\sigma'\times\tau'\subseteq \partial(\sigma\times\tau)$ and that 
$e^{n-1}_\beta\subseteq \partial e^n_\alpha$.

Now $\textnormal{int}(e^{n-1}_\beta)$ contains the homeomorphic image (under the characteristic map 
$\phi^{n-1}_\beta$ of $e^{n-1}_\beta$) of the orbit 
$G/G_{(\sigma',w'\tau')}\times \{(b(\sigma'),b(\tau'))\}$. The preimage of this orbit under the 
attaching map $f^n_\alpha$ is $G/G_{(\sigma,w\tau)} \times\{(b(\sigma'),b(\tau'))\}$. The composite map
\[
G/G_{(\sigma,w\tau)} \times\{(b(\sigma'),b(\tau'))\} 
\xrightarrow{(\phi^{n-1}_\beta)^{-1}\circ f^n_\alpha} G/G_{(\sigma',w'\tau')}\times \{(b(\sigma'),b(\tau'))\}
\]
is a $G$-equivariant map and is therefore determined by the image of $eG_{(\sigma,w\tau)}$. 
Recalling the definition of $f^n_\alpha$ and $\phi^{n-1}_\beta$ we find that 
$eG_{(\sigma,w\tau)}\mapsto g^{-1}G_{(\sigma',w'\tau')}$ where $g$ is any element of $G$ satisfying
\begin{equation} \label{eq:conditionog}
(b(\sigma'),wb(\tau'))^g=(b(\sigma'),w'b(\tau'))\, .
\end{equation}
(The existence of such $g$ is implicit in the assumption that 
$e^{n-1}_\beta\subseteq \partial e^n_\alpha$.) In particular, we have that $g\in G_{\sigma'}$.

Now let $\mathcal{O}_G'$ be the subcategory of $\mathcal{O}_G$ generated by the $G$-maps
\begin{alignat*}{2}
&& G/G_{(\sigma,w\tau)} & \to G/G_{(\sigma',w'\tau')} \\
&& eG_{(\sigma,w\tau)} & \mapsto g^{-1}G_{(\sigma',w'\tau')}
\end{alignat*}
where $g$ satisfies condition (\ref{eq:conditionog}), and $\sigma'\subseteq \sigma$ and 
$\tau'\subseteq \tau$. Then $\mathcal{O}_G'$ includes all morphisms needed to form the chain 
complex $C_\ast(X_G(p);\underline{\pi}_0)$.

We are going to show that the restriction of $(\underline{\pi}_0)_{(p)}$ to $\mathcal{O}'_G$ is 
naturally isomorphic to the constant coefficient system $\underline{\Z/p}$. Consider a morphism in 
$\mathcal{O}_G'$ and let
\[
\pi_0(G_{(\sigma,w\tau)})\to \pi_0(G_{(\sigma',w'\tau')})
\]
be the map obtained by applying $\underline{\pi}_0$ to it. It can be factored as the map induced 
by the inclusion $G_{(\sigma,w\tau)}\to G_{(\sigma',w\tau')}$ and the map induced by conjugation 
$(-)^g\co G_{(\sigma',w\tau')}\to G_{(\sigma',w'\tau')}$. This yields the upper row in the following 
diagram:
\begin{equation} \label{dgr:naturalisomorphism}
\xymatrix{
\pi_0(G_{(\sigma,w\tau)}) \ar[r] \ar[d]^-{\delta_{wb(\tau)}} & \pi_0(G_{(\sigma',w\tau')}) 
\ar[r]^-{(-)^g} \ar[d]^-{\delta_{wb(\tau')}} & \pi_0(G_{(\sigma',w'\tau')}) \ar[d]^-{\delta_{w'b(\tau')}} \\
\pi_1(DG_\sigma) \ar[r] & \pi_1(DG_{\sigma'}) \ar@{=}[r] & \pi_1(DG_{\sigma'}) \\
\Z/p \ar@{=}[rr] \ar[u]_-{\epsilon_\sigma \zeta(b(\sigma))} && \Z/p \ar[u]_-{\epsilon_{\sigma'}\zeta(b(\sigma'))}.
}
\end{equation}
The first map in the middle row is the one induced by the inclusion $G_\sigma\to G_{\sigma'}$. Let us 
show that the upper half of the diagram commutes. The left hand square commutes by naturality of the 
connecting homomorphism. To see that the right hand square commutes, let $\tilde{g}$ denote a lift of 
$g$ in the universal cover $\widetilde{G_{\sigma'}}$, and let $z\in G_{(\sigma',w\tau')}$. 
By definition of the connecting homomorphism (see Lemma \ref{lem:pi0inpi1}), we obtain
\[
\delta_{w'b(\tau')}(z^g)=[\widetilde{w'b(\tau')},\widetilde{z^g}]=
[\widetilde{(wb(\tau'))^g},\widetilde{z^g}]=[\widetilde{wb(\tau')},\tilde{z}]^{\tilde{g}}
=\delta_{wb(\tau')}(z)\,.
\]
In the second equality we used (\ref{eq:conditionog}), and in the last equality we used the fact 
that the commutator is in the center of $\widetilde{G_{\sigma'}}$.

Next, we are going to show that the vertical arrows in the upper half of the diagram become 
isomorphisms after $p$-localization. Let $(x,y)\in X_G(p)$ with $x\in A$. Then, by 
Lemma \ref{lem:pi0inpi1}, $\pi_1(DG_x)$ contains $p$-torsion, hence $p\mid n^\vee(x)$ as 
$\pi_1(DG_x) \cong \Z/n^\vee(x)$. The set of coroot integers of $G$ displayed in 
Table \ref{table:coroot} is at the same time the set of possible values that $n^\vee(x)$ can 
attain as $x$ ranges over the alcove $A$. If $G\notin \{E_7,E_8\}$ or $p>2$, then $n^\vee(x)$ 
does not contain repeated primes. Hence, $\pi_1(DG_x)_{(p)}\cong \Z/p$. By Lemma \ref{lem:pi0inpi1}, 
the map $\delta_y\co \pi_0(G_{(x,y)})\to \pi_1(DG_x)$ is injective, and it remains so after 
$p$-localization. It follows that
\[
\pi_0(G_{(x,y)})_{(p)}\xrightarrow{\cong} \pi_1(DG_x)_{(p)}\, .
\]

To finish the proof we show that if $\sigma'\subseteq \sigma$ and $p\mid n^\vee(b(\sigma))$, then the map
\[
\pi_1(DG_\sigma)_{(p)}\xrightarrow{\cong} \pi_1(DG_{\sigma'})_{(p)}
\]
induced by the inclusion $G_\sigma\leqslant G_{\sigma'}$ is naturally isomorphic to the identity 
at $\Z/p$. This is achieved by making an appropriate choice of generators. Recall from 
Lemma \ref{lem:pi1derived} that $\pi_1(DG_\sigma)$ is a cyclic group of order $n^\vee(b(\sigma))$ 
generated by $\zeta(b(\sigma))$. Let
\[
\epsilon_\sigma=\begin{cases} 3 & \textnormal{if }n^\vee(b(\sigma))=6\textnormal{ and } p=2, \\ 
2 & \textnormal{if }n^\vee(b(\sigma))=6\textnormal{ and } p=3, \\ 1 & \textnormal{otherwise.} 
\end{cases}
\]
Then $\Z/p\to \pi_1(DG_\sigma)$ sending $1\mapsto \epsilon_\sigma \zeta(b(\sigma))$ is injective 
and becomes an isomorphism after $p$-localization. Moreover, the map 
$\pi_1(DG_\sigma)\to \pi_1(DG_{\sigma'})$ sends 
$\epsilon_\sigma \zeta(b(\sigma))\mapsto \epsilon_{\sigma'}\zeta(b(\sigma'))$ by the second part of 
Lemma \ref{lem:pi1derived}. This demonstrates commutativity of the lower part of 
diagram (\ref{dgr:naturalisomorphism}). As a consequence, (the $p$-localization of) 
diagram (\ref{dgr:naturalisomorphism}) establishes a natural isomorphism of 
$(\underline{\pi}_0)_{(p)}$ with $\underline{\Z/p}$ as asserted.
\end{proof}

Let us comment on what happens when $G \in\{E_7,E_8\}$ and $p=2$. 
Let us restrict the coefficient system $\underline{\pi}_{(0)}$ to the $G$-subcomplex $X_G(2)$ of 
$\Hom(\Z^2,G)$. According to Table \ref{table:coroot} the coefficient system $(\underline{\pi}_0)_{(2)}$ 
can now evaluate to $\Z/2$ and $\Z/4$. Thus, in contrast to Lemma \ref{lem:equivarianthomologyquotient} 
it need not be isomorphic to a constant coefficient system. To handle 
these cases we observe that the argument of Lemma \ref{lem:subcomplex} shows that
\[
X_G(4):=\{(x,y)\in X_G(2) \mid \pi_0(G_{(x,y)})_{(2)}\cong \Z/4\}
\]
is a $G$-subcomplex of $X_G(2)$. Restricted to $X_G(4)$ the coefficient system 
$(\underline{\pi}_0)_{(2)}$ is naturally isomorphic to $\underline{\Z/4}$ and as a consequence 
we obtain the following.

\begin{lemma} \label{lem:e7e8rg4}
Let $G=E_7$ or $G=E_8$. Then
\begin{equation*} \label{eq:RG4}
H^G_\ast(X_G(4);(\underline{\pi}_0)_{(2)})\cong H_\ast(R_G(4);\Z/4)\,,
\end{equation*}
where $R_G(4):=X_G(4)/G$.
\end{lemma}

\subsection{The quotient spaces $R_G(p)$} \label{sec:quotients}
Recall the $G$-CW-complex
\[
X_G(p)=\{(x,y)\in \Hom(\Z^2,G)\mid \pi_0(Z_G(x,y))_{(p)}\neq 0\}
\]
and its orbit space $R_G(p)=X_G(p)/G$. At this point, we have reduced the computation of $H_\ast^G(\Hom(\Z^2,G);\underline{\pi}_0)$ to a 
computation of the non-equivariant homology of $R_G(p)$, with an exception when $G$ is $E_7$ or $E_8$. 
As $R_G(p)$ is the orbit space of a $G$-subcomplex of $\Hom(\Z^2,G)$ it is a subcomplex in 
the induced CW-structure of $T^2/W$. It would be interesting to describe this 
subcomplex, but here we will content ourselves with a calculation of the homology of $R_G(p)$. 
This will be achieved by providing an explicit homotopy equivalence of $R_G(p)$ with a weighted projective 
space.

Let $\mathbf{w}=(w_0,\dots,w_r)$ be a tuple of positive integers. Consider the weighted action of 
the circle group $\SS^1\subseteq \C$ on the unit sphere $\SS^{2r+1}\subseteq \C^{r+1}$ defined by
\[
\lambda\cdot (z_0,\dots,z_r)=(\lambda^{w_0}z_0,\dots,\lambda^{w_r}z_r) \; \quad \; 
(\lambda\in \SS^1,\; (z_0,\dots,z_r)\in \SS^{2r+1})\, .
\]
The quotient space
\[
\mathbb{CP}(\mathbf{w}):=\SS^{2r+1}/\SS^1_\mathbf{w}
\]
is called a \emph{weighted projective space}. Here we use the subscript $\mathbf{w}$ to 
indicate the weighted $\SS^1$-action. The most familiar example arises when $\mathbf{w}=(1,\dots,1)$ 
in which case $\mathbb{CP}(\mathbf{w})=\mathbb{CP}^{r}$ is the usual complex projective space.

\medskip

Recall that, relative to a fixed simply--connected simple compact Lie group $G$, we let 
$\mathcal{P}=\{n_1^\vee,\dots,n_r^\vee\}\backslash\{1,4,6\}$ denote 
the set of primes dividing a coroot integer of $G$. Define
\[
\mathcal{Z}:=\{n_1^\vee,\dots,n_r^\vee\}\backslash \{1,6\}\, .
\]
Thus $\mathcal{Z}=\mathcal{P}$ except when $G$ is $E_7$ or $E_8$ in which case 
$\mathcal{Z}=\mathcal{P}\cup \{4\}$. The objective of this subsection is to prove:

\begin{proposition} \label{prop:weightedprojective}
Let $\mathbf{n}^\vee=(n_0^\vee,\dots,n^\vee_r)$ be the tuple of coroot integers of $G$ and let 
$p\in\mathcal{Z}$ be fixed. Let 
$\mathbf{n}^\vee(p)=(n_{i_0}^\vee,\dots,n_{i_k}^\vee)$ denote the tuple obtained from $\mathbf{n}^\vee$ 
by removing those entries that are not divisible by $p$. Let 
$\iota_p\co \mathbb{CP}(\mathbf{n}^\vee(p))\to \mathbb{CP}(\mathbf{n}^\vee)$ be 
the map defined in homogeneous coordinates by $[y_0,\dots,y_k]\mapsto [z_0,\dots,z_r]$ 
where $z_{l}=y_j$ if $l=i_j$ for some $0\leqslant j\leqslant k$, and $z_l=0$ otherwise. 
Then there is a commutative diagram
\[
\xymatrix{
R_G(p) \ar[r]^-{\simeq} \ar[d] & \mathbb{CP}(\mathbf{n}^\vee(p)) \ar[d]^-{\iota_p} \\
\Rep(\Z^2,G) \ar[r]^-{\simeq} & \mathbb{CP}(\mathbf{n}^\vee)
}
\]
in which the horizontal arrows are homotopy equivalences and the left hand vertical arrow is the 
subspace inclusion.
\end{proposition}

Like any complete toric variety a weighted projective space is simply--connected 
(for the definition of weighted projective space as a toric variety see for example \cite[p. 35]{Fulton}, 
and for the result on the fundamental group see \cite[Section 3.2]{Fulton}). The integral cohomology 
of $\mathbb{CP}(\mathbf{w})$ was computed by Kawasaki in \cite[Theorem 1]{K73}. The homology is 
then obtained from the universal coefficient theorem:
\begin{equation} \label{eq:homologyofweightedprojectivespace}
H_k(\mathbb{CP}(\mathbf{w});\Z)\cong \begin{cases} \Z\,, & \textnormal{if } 
k\leqslant 2r \textnormal{ is even,} \\ 0\,, &\textnormal{otherwise.} \end{cases}
\end{equation}
In view of this result we have:

\begin{corollary} \label{cor:homologyrp}
Let $p\in \mathcal{P}$ and assume that $G\not\in \{E_7,E_8\}$ or that $p>2$. Then
\[
H_k^G(X_G(p);(\underline{\pi}_0)_{(p)})\cong \begin{cases} \Z/p\,, & \textnormal{if } k 
\leqslant 2(\ell-1)\textnormal{ is even}, \\ 0\,, & \textnormal{otherwise},  \end{cases}
\]
where $\ell\geqslant 1$ is the number of coroot integers of $G$ divisible by $p$.
\end{corollary}
\begin{proof}
By Lemma \ref{lem:equivarianthomologyquotient}, 
$H_\ast^G(X_G(p);(\underline{\pi}_0)_{(p)})\cong H_\ast(R_G(p);\Z/p)$. By Proposition 
\ref{prop:weightedprojective}, $R_G(p)$ is homotopy equivalent to a weighted projective space of 
complex dimension $\ell-1$. The homology groups follow therefore from Kawasaki's computation (\ref{eq:homologyofweightedprojectivespace}).
\end{proof}

Our approach to Proposition \ref{prop:weightedprojective} is as follows. The homotopy equivalence 
$\Rep(\Z^2,G)\simeq \mathbb{CP}(\mathbf{n}^\vee)$ will be obtained as a composite of maps,
\begin{equation} \label{eq:equivalencewp}
\Rep(\Z^2,G)\xrightarrow{\cong} A\otimes_{\mathbb{I}} F\xrightarrow{\simeq} 
A\otimes_{\mathbb{I}} \bar{F} \xrightarrow{\cong} \mathbb{CP}(\mathbf{n}^\vee)\, .
\end{equation}
The two spaces in the middle are certain coends that will be defined below, and the homotopy 
equivalence is induced from a natural equivalence of diagrams $F\xrightarrow{\sim} \bar{F}$. 
The equivalence $R_G(p)\simeq \mathbb{CP}(\mathbf{n}^\vee(p))$ will be obtained from this by 
restriction to subspaces, whence the diagram in Proposition \ref{prop:weightedprojective} will 
commute by construction.

In brief, the homotopy equivalence $\Rep(\Z^2,G)\simeq \mathbb{CP}(\mathbf{n}^\vee)$ is given by 
\begin{equation} \label{eq:explicitequivalence}
(x,y) \textnormal{ modulo }G \mapsto (a_0,a_1t_1,\dots,a_rt_r)/\sqrt{\sum_{i=0}^r a_i^2} 
\textnormal{ modulo } \SS^1_{\mathbf{n}^\vee}\, ,
\end{equation}
where $x\in A$ and $(a_0,\dots,a_r)\in \Delta^r$ are the barycentric coordinates of $x$, and  $y\in T$ 
and $(t_1,\dots,t_r)\in (\SS^1)^r$ are the coordinates of $y$ with respect to the one-parameter subgroups 
of $T$ determined by the coroots $\alpha_1^\vee,\dots,\alpha_r^\vee$. This will be clear once we have 
proved Lemma \ref{lem:coendwp}.

\medskip

To describe the first one of the two homeomorphisms in (\ref{eq:equivalencewp}) we will regard 
$\Rep(\Z^2,G)$ and $R_G(p)$ as spaces over the alcove $A$. Fix $p\in \mathcal{Z}$. For an integer $m$ 
let $A(m)\subseteq A$ be the subspace of the fundamental alcove defined by
\[
x\in A(m)\Longleftrightarrow m\mid n^\vee(x)\,.
\]
This is again a simplex, in fact, a face of $A$ as we explain below (Lemma \ref{lem:amface}). 
Note that if $(x,y)\in X_G(p)$ and $x\in A$, then $p\mid n^\vee(x)$, hence $x\in A(p)$. Thus, 
there is a commutative diagram
\[
\xymatrix{
R_G(p) \ar[r] \ar[d]^-{q} & \Rep(\Z^2,G) \ar[d]^-{\textnormal{pr}_1} \\
A(p) \ar[r] & A
}
\]
in which $\textnormal{pr}_1$ is the projection of $\Rep(\Z^2,G)$ onto the first component and 
$q$ is its restriction to the subspace $R_G(p)$. The two horizontal maps are the inclusions. We 
will describe the fibers of $\textnormal{pr}_1$ and $q$.

\begin{notation*}
When $K$ is a group, we write $K/\Ad_K$ for the quotient by the conjugation (or adjoint) action 
of $K$ on itself. The conjugacy class of an element $g\in K$ is denoted by $(g)\in K/\Ad_K$. 
Similarly, given $(x,y)\in \Hom(\Z^2,G)$ we write $((x,y))\in \Rep(\Z^2,G)$ for its equivalence 
class under simultaneous conjugation. When $(x,y)\in T^2$, we will also write $((x,y))\in T^2/W$ 
for its equivalence class under the diagonal action of $W$.
\end{notation*}

Clearly, if $x\in A$, then the assignment $(y)\mapsto ((x,y))$ defines a homeomorphism
\begin{equation} \label{eq:pr1fiber}
Z_G(x)/\Ad_{Z_G(x)}\cong \textnormal{pr}_1^{-1}(x)\,.
\end{equation}
On the other hand, if $x\in A(p)$, then $q^{-1}(x)=\textnormal{pr}_1^{-1}(x)\cap R_G(p)$ 
corresponds to a subspace of $Z_G(x)/\Ad_{Z_G(x)}$. To describe it we rewrite (\ref{eq:pr1fiber}) 
using the finite covering
\begin{alignat*}{2}
 \rho\co && \widetilde{DZ_G(x)}\times Z(Z_G(x))_0 & \to Z_G(x) \\
&& (g,t) & \mapsto u(g)t\,,
\end{alignat*}
where $Z(Z_G(x))_0$ is the identity component of the center of $Z_G(x)$, and 
$u\co \widetilde{DZ_G(x)}\to DZ_G(x)$ is the universal covering (see \cite[IX \S 1.4 Corollary 1]{Bourbaki}). 
There is a finite subgroup $C \leqslant Z(\widetilde{DZ_G(x)})\times Z(Z_G(x))_0$ such that 
$\rho$ descends to an isomorphism
\begin{equation} \label{eq:coveringderivedcenter}
\widetilde{DZ_G(x)}\times_{C}Z(Z_G(x))_0\xrightarrow{\cong} Z_G(x)\, .
\end{equation}
Under this isomorphism (\ref{eq:pr1fiber}) becomes
\begin{equation} \label{eq:pr1fiber2}
\textnormal{pr}_1^{-1}(x)\cong \left(\widetilde{DZ_G(x)}/
\Ad_{\widetilde{DZ_G(x)}}\right)\times_{C} Z(Z_G(x))_0\, .
\end{equation}
Here $C$ acts on the adjoint orbits in the natural way: If $K$ is any group, then an action of the 
center $Z(K)$ on $K/\Ad_K$ is defined by
\begin{equation} \label{eq:actionofcenter}
c\cdot (g)=(cg)\,,\quad (c\in Z(K),\,g\in K)\, .
\end{equation}

We can now describe the fibers of $q\co R_G(p)\to A(p)$.
\begin{lemma} \label{lem:qfiber}
Let $p\in \mathcal{Z}$, and let $x\in A(p)$. Then, there is an 
element $\xi(x)\in Z(\widetilde{DZ_G(x)})$ and a homeomorphism
\[
q^{-1}(x)\cong \left( \widetilde{DZ_G(x)}/\Ad_{\widetilde{DZ_G(x)}}
\right)^{\langle \xi(x)\rangle}\times_{C} Z(Z_G(x))_0\,.
\]
\end{lemma}

Note that both $C$ and $\langle \xi(x)\rangle$ act through subgroups of the center $Z(\widetilde{DZ_G(x)})$. 
As the center is abelian, the two actions commute, and there is an induced action of $C$ on the 
$\langle \xi(x)\rangle$-fixed points.

\begin{proof}
Let us discuss the case where $p$ is a prime. The case $p=4$ is treated analogously, see 
Remark \ref{rem:fiberqp4}. By definition of $R_G(p)$ and the isomorphism (\ref{eq:pr1fiber}) we have that
\begin{equation} \label{eq:qx-1}
q^{-1}(x)\cong\{y\in Z_G(x) \mid \pi_0(Z_G(x,y))_{(p)}\neq 0\}/\Ad_{Z_G(x)}\, .
\end{equation}
To reformulate the condition $\pi_0(Z_G(x,y))_{(p)}\neq 0$ we will use the connecting homomorphism 
defined in Lemma \ref{lem:pi0inpi1}. To this end, let
\[
\xi(x):=\begin{cases} \textnormal{any generator of }
\pi_1(DZ_G(x))_{(p)}\cong \Z/p & \textnormal{if }p>2, \\ 
\textnormal{the unique element of order $2$ in }\pi_1(DZ_G(x))_{(2)} & 
\textnormal{if }p=2 \end{cases}
\]
viewed as an element in the center of $\widetilde{DZ_G(x)}$. Then we have that
\[
\pi_0(Z_G(x,y))_{(p)}\neq 0\Longleftrightarrow \xi(x)\in \Im(\delta_y)  
\Longleftrightarrow \exists \tilde{z}\in \widetilde{Z_G(x)}: [\tilde{y},\tilde{z}]=\xi(x)\, .
\]
Since $\widetilde{Z_G(x)}\cong \widetilde{DZ_G(x)}\times \widetilde{Z(Z_G(x))_0}$, the last condition 
is equivalent to $y\in U(x)\times \widetilde{Z(Z_G(x))_0}$, where
\[
U(x):=\{\tilde{a}\in \widetilde{DZ_G(x)}\mid \exists \tilde{z}\in \widetilde{DZ_G(x)}: 
[\tilde{a},\tilde{z}]=\xi(x)\}\,.
\]
Similarly to (\ref{eq:pr1fiber2}) we then obtain
\[
q^{-1}(x)\cong \left(U(x)/\Ad_{\widetilde{DZ_G(x)}}\right)\times_{C} Z(Z_G(x))_0\, .
\]
Finally, writing out the commutator shows that
\[
\exists \tilde{z}\in \widetilde{DZ_G(x)}: [\tilde{a},\tilde{z}]
=\xi(x) \Longleftrightarrow \xi(x)\cdot (\tilde{a})=(\tilde{a}) \,,
\]
hence $U(x)/\Ad_{\widetilde{DZ_G(x)}}\cong 
\left( \widetilde{DZ_G(x)}/\Ad_{\widetilde{DZ_G(x)}}\right)^{\langle \xi(x)\rangle}$, and the 
lemma follows.
\end{proof}

\begin{remark} \label{rem:fiberqp4}
In the case $p=4$ the defining condition for $q^{-1}(x)$ reads $\pi_0(Z_G(x,y))_{(2)}\cong \Z/4$ 
and one chooses $\xi(x)$ to be any generator of $\pi_1(DZ_G(x))_{(2)}\cong \Z/4$. The rest of 
the proof goes through verbatim.
\end{remark}

Observe that the fibers $q^{-1}(x)$ only depend on the face $\sigma\subseteq A(p)$ such that 
$x\in \textnormal{int}(\sigma)$ and not on the specific point $x$ chosen within 
$\textnormal{int}(\sigma)$. That is, if $x,y\in \textnormal{int}(\sigma)$, then 
$q^{-1}(x)=q^{-1}(y)=q^{-1}(b(\sigma))$, and likewise for $\textnormal{pr}_1$. Together with the 
combinatorial structure of the alcove this allows for a more manageable description of $R_G(p)$ 
and $\Rep(\Z^2,G)$.

To this end, recall that the faces of the alcove $A$ are in one-to-one correspondence with the 
proper subsets $I\subseteq \tilde{\Delta}$ of the extended set of simple roots as follows:
\[
I\subsetneq \tilde{\Delta}\; \longleftrightarrow\; \sigma^{I}\in \mathscr{F}_{r-|I|}\, ,
\]
where
\[
\sigma^I:=\{x\in A\mid \forall \alpha\in I: x\textnormal{ lies in the wall of }
A\textnormal{ determined by }\alpha\}\, .
\]
Since $A$ is a simplex and $\sigma^I$ is an intersection of facets, $\sigma^I$ is a face of $A$.

For example, if $I=\emptyset$, then $\sigma^{\emptyset}=A$. If 
$I=\tilde{\Delta}\backslash \{\alpha_j\}$ for some $j\in \{0,\dots,r\}$, then $\sigma^I$ is the 
vertex opposite the wall of $A$ determined by $\alpha_j$. If $x\in A$ and $I=\tilde{\Delta}(x)$ 
(as defined in Section \ref{sec:componentgroup} eq. (\ref{eq:deltax})), then $\sigma^I$ is the 
minimal face of $A$ containing $x$. Also relevant to us is:

\begin{lemma} \label{lem:amface}
Let $m$ be an integer and let $I_m=\{\alpha_j\in \tilde{\Delta}\mid  m\nmid n_j^\vee\}$. Then 
$\sigma^{I_m}=A(m)$.
\end{lemma}
\begin{proof}
We have that
\begin{equation*}
\begin{split}
x\in A(m) & \Longleftrightarrow m\mid n^\vee(x) \\& \Longleftrightarrow  \forall j\in 
\tilde{\Delta}\backslash \tilde{\Delta}(x): m\mid n_j^\vee \\& 
\Longleftrightarrow I_m\subseteq \tilde{\Delta}(x) \\& \Longleftrightarrow \forall 
\alpha\in I_m : x\textnormal{ lies in the wall of $A$ determined by }\alpha \\& 
\Longleftrightarrow x\in \sigma^{I_m}\, ,
\end{split}
\end{equation*}
by definition of $A(m)$, $n^\vee(x)$, $\tilde{\Delta}(x)$ and $I_m$.
\end{proof}

It follows that $A(m)$ is a face of $A$, hence a simplex.\medskip

Let $\I$ denote the poset of proper subsets of $\tilde{\Delta}$ partially ordered by inclusion. 
Note that when $I,J\in \I$ and $I\subseteq J$, then $\sigma^J$ is a face of $\sigma^I$. Therefore, 
the assignment $I\mapsto \sigma^I$ defines a contravariant functor from $\I$ to $\mathbf{Top}$ 
sending the inclusion $I\subseteq J$ to the opposite inclusion $\sigma^{J}\subseteq \sigma^I$. 
This functor encodes the face structure of the fundamental alcove. Abusing notation, we denote it by
\[
A\co \I^{\op} \to \mathbf{Top}\,, \quad I \mapsto A(I)=\sigma^I\, .
\]
Similarly, for an integer $m$ we denote by
\[
A(m)\co \I(m)^{\op}\to \mathbf{Top}
\]
the restriction of $A$ to the subposet $\I(m)\subseteq \I$ consisting of those subsets $I$ 
containing $I_m$ (see Lemma \ref{lem:amface}). It encodes the face structure of the simplex $A(m)$.

On the other hand, we can define the following covariant functors on $\I$. Recall that if 
$\sigma,\tau$ are faces of $A$ and $\tau\subseteq \sigma$, then there is an inclusion 
$Z_G(b(\sigma))\leqslant Z_G(b(\tau))$. This induces a map 
$Z_G(b(\sigma))/\Ad_{Z_G(b(\sigma))}\to Z_G(b(\tau))/\Ad_{Z_G(b(\tau))}$ (which is, in fact, a 
surjection rather than an inclusion) and thus by (\ref{eq:pr1fiber}) a map
\[
\textnormal{pr}_1^{-1}(b(\sigma))\to \textnormal{pr}_1^{-1}(b(\tau))\,.
\]
Now if $\sigma,\tau$ are faces of $A(p)$ (where $p$ is as in Lemma \ref{lem:qfiber}), then this 
restricts to a map
\[
q^{-1}(b(\sigma))\to q^{-1}(b(\tau))\,.
\]
This is perhaps most easily seen from the description of the fibers of $q$ displayed in (\ref{eq:qx-1}). 
We have that $(y)\in q^{-1}(b(\sigma))$ if and only if $\pi_0(Z_G(b(\sigma),y))_{(p)}\neq 0$. 
But then $\pi_0(Z_G(b(\tau),y))_{(p)}\neq 0$ by Lemma \ref{lem:pi0injective}, hence $(y)\in q^{-1}(b(\tau))$. 
Therefore, there are functors
\[
F\co \I \to \mathbf{Top}\,, \quad I \mapsto \textnormal{pr}_1^{-1}(b(\sigma^{I}))					
\]
and
\[
F'\co  \I(p) \to \mathbf{Top}\,,  \quad  I \mapsto q^{-1}(b(\sigma^{I}))\, . 
\]
We may now identify $\Rep(\Z^2,G)$ and $R_G(p)$ as coends of the functor pairs $(A,F)$ and $(A(p),F')$, 
respectively. To this end, recall

\begin{definition}
Let $\mathcal{C}$ be a small category and $M\co \mathcal{C}\to \mathbf{Top}$ and 
$L\co \mathcal{C}^{\op}\to \mathbf{Top}$ a pair of co- and contravariant functors. The 
\emph{coend} $L\otimes_{\mathcal{C}} M$ is the topological space defined by
\[
L\otimes_{\mathcal{C}} M=\left(\bigsqcup_{c\in \Ob(\mathcal{C})} L(c)\times M(c)\right)/\approx\,,
\]
where the equivalence relation $\approx$ is given by
\[
(L(i)(a),b)\approx (a,M(i)(b))
\]
for all $c,d\in \mathcal{C}$, $i\in \Mor_{\mathcal{C}}(c,d)$, $a\in L(d)$ and $b\in M(c)$.
\end{definition}

\begin{notation*}
For $x\in L(c)$ and $y\in M(c)$ we denote by $x\otimes y$ the equivalence class of $(x,y)$ in 
the coend.
\end{notation*}

The following lemma explains the first one of the two homeomorphisms in (\ref{eq:equivalencewp}).

\begin{lemma} \label{lem:coend}
There is a commutative diagram
\[
\xymatrix{
R_G(p) \ar[r]^-{\cong} \ar[d] & A(p)\otimes_{\I(p)} F' \ar[d] \\
\Rep(\Z^2,G) \ar[r]^-{\cong} & A\otimes_{\I} F
}
\]
in which the two vertical maps are inclusions of subspaces.
\end{lemma}
\begin{proof}
We first construct the bottom map. Each $((x,y))\in \Rep(\Z^2,G)$ is represented by a pair $(x,y)$ 
with $x\in A$. Define
\begin{alignat*}{2}
h\co \Rep(\Z^2,G) & \to A\otimes_{\I} F && \\
 ((x,y)) & \mapsto x\otimes (y) & \quad \quad & (x\in \sigma^{\tilde{\Delta}(x)},\, 
(y)\in Z_G(x)/\Ad_{Z_G(x)}).
\end{alignat*}
To see that $h$ is well defined suppose that $((x,y))=((x',y'))$ and both $x,x'\in A$. Then $x=x'$ and 
there is $g\in Z_G(x)$ such that $y'=y^g$. Hence, $(y)=(y')$ as elements of $Z_G(x)/\Ad_{Z_G(x)}$.

To see that $h$ is surjective observe that $F(\emptyset)=Z_G(b(A))/\Ad_{Z_G(b(A))}=T$, since $b(A)$ is 
a regular point of $G$. Hence, $A\times T$ appears as one of the disjoint summands in the definition of 
the coend. Since for each $\sigma\subseteq A$ the map $T\to Z_G(b(\sigma))/\Ad_{Z_G(b(\sigma))}$ is 
surjective, every element of $A\otimes_{\I}F$ has a representative in $A\times T$ and is therefore in 
the image of $h$.

Finally, to see that $h$ is injective suppose that $x\otimes (y)=x'\otimes (y')$. Since for each inclusion 
$I\subseteq J$ the map $\sigma^J\to \sigma^I$ is injective, we must have $x=x'$. But then, $(y)=(y')$ as 
elements of $Z_G(x)/\Ad_{Z_G(x)}$, so $((x,y))=((x',y'))$.

Together this proves that $h$ is a continuous bijection. Since $\Rep(\Z^2,G)$ is compact and 
$A\otimes_{\I} F$ is Hausdorff, $h$ is a homeomorphism.

To obtain the homeomorphism in the top row of the square, first observe that the inclusion map
\[
\bigsqcup_{I\in \I(p)} \sigma^I\times F'(I) \hookrightarrow \bigsqcup_{I\in\I} \sigma^I\times F(I)
\]
induces a homeomorphism $i$ of $A(p)\otimes_{\I(p)}F'$ onto its image in $A\otimes_{\I}F$. The 
homeomorphism $h$ from before restricts to a bijection of $R_G(p)$ with $i(A(p)\otimes_{\I(p)}F')$. 
Since both $R_G(p)$ and $i(A(p)\otimes_{\I(p)}F)$ carry the subspace topology, this bijection is a 
homeomorphism.
\end{proof}

We now turn our attention to the homotopy equivalence in (\ref{eq:equivalencewp}). To this end, we 
introduce another diagram on $\I$ and $\I(p)$. Given $I\in \I$ let $\t(I)\leqslant \t$ be the 
$\R$-linear span of $\{\alpha^\vee_i\mid i\in I\}$, and let $T(I)\leqslant T$ be the subtorus 
whose Lie algebra is $\t(I)$; in other words, $T(I)$ is the image of $\t(I)$ under 
$\exp\co \t\to T$. Note that $T(I)=T\cap DZ_G(b(\sigma^I))$ can be identified with the maximal 
torus of $DZ_G(b(\sigma^I))$. When $I\subseteq J$, then $T(I)\leqslant T(J)$, and there is a quotient 
map $T/T(I)\to T/T(J)$. This gives rise to a diagram of tori
\[
\bar{F}\co \I \to \mathbf{Top}\,, \quad I \mapsto T/T(I)\, .
\]
We will not distinguish notationally between $\bar{F}$ and its restriction to $\I(p)$.

\begin{lemma} \label{lem:coendequivalence}
There are natural equivalences of diagrams $F\xrightarrow{\simeq} \bar{F}$ and 
$F'\xrightarrow{\simeq} \bar{F}$ which fit into a commutative diagram
\[
\xymatrix{
A(p)\otimes_{\I(p)} F' \ar[r]^-{\simeq} \ar[d] & A(p)\otimes_{\I(p)} \bar{F} \ar[d] \\
A\otimes_{\I} F \ar[r]^-{\simeq} & A\otimes_{\I} \bar{F}
}
\]
where the two vertical maps are inclusions of subspaces.
\end{lemma}

Suppose that $K$ is a compact, simply--connected, simple Lie group. To prove the lemma we need a 
fact concerning the action of the center $Z(K)$ on $K/\Ad_K$ specified in (\ref{eq:actionofcenter}). 
Let $S\leqslant K$ be a maximal torus, $\s$ its Lie algebra and assume that appropriate choices 
have been made that allow us to identify $K/\Ad_K$ with (the closure of) an alcove $A_K\subseteq \s$. Then it is well 
known that the resulting action of $Z(K)$ on $A_K$ is through affine isometries of $\s$ permuting the 
vertices of $A_K$ (see \cite[Section 3.2]{BFM}).
\begin{lemma} \label{lem:alcovecontractible}
Let $K$ be a compact simply--connected semi-simple Lie group with center $Z(K)$. Then $K/\Ad_K$ is 
$Z(K)$-equivariantly contractible. 
\end{lemma}
\begin{proof}
It suffices to prove the lemma in the case where $K$ is simple, because in the general case there is 
$s\geqslant 1$ such that $K\cong K_1\times \dots \times K_s$ where each $K_i$ is simple. Hence, 
$K/\Ad_K \cong K_1/\Ad_{K_1}\times \dots\times K_s/\Ad_{K_s}$, and the action of 
$Z(K)\cong Z(K_1)\times \dots \times Z(K_s)$ is through the action of $Z(K_i)$ on $K_i/\Ad_{K_i}$ 
for every $i=1,\dots,s$.

Assuming that $K$ is simple, we identify $K/\Ad_K$ with a Weyl alcove $A_K$ in the Lie algebra $\s$ of 
a maximal torus for $K$. Then $A_K$ is a simplex given in barycentric coordinates by
\[
A_K=\left\{ \sum_{v\in V} a_v v \mid a_v\in \R_{\geqslant 0} \textnormal{ for all } v\in V 
\textnormal{ and } \sum_{v\in V}a_v=1\right\}\subseteq \s\, ,
\]
where $V\subseteq A_K$ is the set of vertices.

As $Z(K)$ acts on $A_K$ through affine isometries of $\s$ permuting the vertices of $A_K$, the 
action is determined by this permutation action on the vertex set $V$.  In particular, the 
barycenter $b(A_K)=\left(\sum_{v\in V} v\right)/|V|$ is a global fixed point for the $Z(K)$-action. 
Now let $\Gamma\leqslant Z(K)$ be any subgroup. Let $V=V_1\sqcup \dots \sqcup V_{\ell}$ be the 
decomposition of $V$ into orbits with respect to the permutation action of $\Gamma$ on $V$. Then 
a point $x=\sum_{v\in V} a_v v$ of $A_K$ is fixed by $\Gamma$ if and only if $a_v=a_u$ for all 
$v,u\in V_j$ and all $j=1\dots \ell$. Clearly, if $x$ is fixed by $\Gamma$, then so is 
$ab(A_K)+(1-a)x$ for every $a\in [0,1]$. It follows that the fixed point space $A_K^\Gamma$ is a 
star-shaped domain in $\s$ with respect to the barycenter $b(A_K)$ and therefore contractible. 
As this holds for all subgroups $\Gamma\leqslant Z(K)$, $A_K$ is $Z(K)$-equivariantly contractible.
\end{proof}

\begin{proof}[Proof of Lemma \ref{lem:coendequivalence}]
We will only describe the equivalence $F'\xrightarrow{\simeq} \bar{F}$ as the equivalence 
$F\xrightarrow{\simeq} \bar{F}$ follows analogously. By Lemma \ref{lem:qfiber},
\[
F'(I)=q^{-1}(b(\sigma^I))\cong  
\left( \widetilde{DZ_G(b(\sigma^I))}/\Ad_{\widetilde{DZ_G(b(\sigma^I))}}
\right)^{\langle \xi(b(\sigma^I))\rangle}\times_{C} Z(Z_G(b(\sigma^I)))_0\, .
\]
Since both $\langle \xi(b(\sigma^I))\rangle$ and $C$ act through the center of 
$\widetilde{DZ_G(b(\sigma^I))}$, Lemma \ref{lem:alcovecontractible} implies that the projection 
onto $Z(Z_G(b(\sigma^I)))_0$ induces a homotopy equivalence
\[
F'(I)\xrightarrow{\simeq} Z(Z_G(b(\sigma^I)))_0/C\, ,
\]
that is natural in $I$. The latter space can be identified naturally with $Z_G(b(\sigma^I))/DZ_G(b(\sigma^I))$ 
as a look at the isomorphism (\ref{eq:coveringderivedcenter}) shows. On the other hand, as 
$T(I)=T\cap DZ_G(b(\sigma^I))$ the inclusion $T\hookrightarrow Z_G(b(\sigma^I))$ induces a natural isomorphism
\[
T/T(I) \cong Z_G(b(\sigma^I))/DZ_G(b(\sigma^I))\, .
\]
This proves the equivalence $F'\xrightarrow{\simeq} \bar{F}$.

To see that the natural equivalences $F'\xrightarrow{\simeq} \bar{F}$ and $F\xrightarrow{\simeq} \bar{F}$ 
induce homotopy equivalences $A(p)\otimes_{\I(p)} F'\simeq A(p)\otimes_{\I(p)}\bar{F}$ and 
$A\otimes_{\I} F\simeq A\otimes_{\I} \bar{F}$ we recognize the coends as homotopy colimits of the diagrams 
$F'$, $F$ and $\bar{F}$. For this recall that, if $M\co \mathcal{C}\to \mathbf{Top}$ is a small diagram 
of spaces, then a model for the homotopy colimit of $M$ is the coend
\[
\underset{\mathcal{C}}{\hocolim}\, M=B(-/ \mathcal{C})\otimes_{\mathcal{C}} M\,,
\]
where $B(-)$ is the classifying space functor and $c /\mathcal{C}$ denotes the category of objects under 
$c\in\mathcal{C}$, see \cite[XII. \S 2.1]{BK72}. Since the diagrams $A\co I\mapsto \sigma^I$ and 
$I\mapsto B(I/\I)$ are naturally isomorphic, we find that
\[
\Rep(\Z^2,G)\cong A\otimes_{\I} F\cong B(-/\I)\otimes_{\I} F = \underset{\I}{\hocolim}\, F\, ,
\]
and similarly,
\[
R_G(p)\cong \hocolim_{\I(p)} F'\,.
\]
The required homotopy equivalences are now implied by homotopy invariance of homotopy colimits. 
Finally, commutativity of the diagram follows by inspection.
\end{proof}

The following lemma completes the proof of Proposition \ref{prop:weightedprojective}. Essentially, 
this is an identification of the coend $A\otimes_{\I} \bar{F}$ with the weighted projective space 
$\mathbb{CP}(\mathbf{n}^\vee)$. This kind of identification is well known in toric topology (see e.g. \cite[Section 5.3]{WZZ}), 
but we will give a direct proof here.
\begin{lemma} \label{lem:coendwp}
There is a commutative diagram
\[
\xymatrix{
A(p)\otimes_{\I(p)} \bar{F} \ar[r]^-{\cong} \ar[d] & \mathbb{CP}(\mathbf{n}^\vee(p)) \ar[d]^-{\iota_p} \\
A\otimes_{\I} \bar{F} \ar[r]^-{\cong} & \mathbb{CP}(\mathbf{n}^\vee)
}
\]
in which the left hand vertical map is a subspace inclusion.
\end{lemma}
\begin{proof}
The top horizontal map is simply the restriction of the bottom map, so we first construct the latter. 
To this end, we replace $A\otimes_{\I}\bar{F}$ by the homeomorphic identification space 
$(A\times T)/{\approx}$ where
\[
(x,t)\approx (x',t') \Longleftrightarrow x=x'\textnormal{ and } t^{-1}t'\in T(\tilde{\Delta}(x))\,.
\]
The homeomorphism $(A\times T)/{\approx}\to A\otimes_{\I} \bar{F}$ is given by mapping 
$(x,t)\mapsto x\otimes t$.

To define a homeomorphism of $(A\times T)/{\approx}$ with 
$\mathbb{CP}(\mathbf{n}^\vee)=\SS^{2r+1}/\SS^1_{\mathbf{n}^\vee}$ we 
shall view $\SS^{2r+1}$ as the $(r+1)$-fold unreduced join
\[
\SS^{2r+1}\cong \SS^1\ast\cdots \ast \SS^1=(\SS^1)^{\ast(r+1)}\, .
\]
It is convenient to write elements of $(\SS^1)^{\ast(r+1)}$ formally as tuples 
$\langle a_0 t_0, \dots, a_r t_r\rangle$ with
\[
a_i\in [0,1],\, t_i\in \SS^1\textnormal{ for } i=0,\dots,r\textnormal{ and }\sum_{i=0}^r a_i=1\,,
\]
and subject to the identification $0 t=0 t'$ for all $t,t'\in \SS^1$. The homeomorphism with the unit 
sphere $\SS^{2r+1}\subseteq \C^{r+1}$ is then given by
\[
\langle a_0 t_0,\dots,a_r t_r\rangle \mapsto (a_0t_0,\cdots,a_rt_r)/\rVert(a_0t_0,\cdots,a_rt_r)\rVert\, .
\]
Note that this map is $\SS^1_{\mathbf{n}^\vee}$-equivariant for the diagonal (weighted) action on 
$(\SS^1)^{\ast(r+1)}$. We identify points of $A$ with their barycentric coordinates 
$\mathbf{a}=(a_0,\dots,a_r)\in\Delta^r$. Given $t\in T$, we let $(t_1,\dots,t_r)\in (\SS^1)^r$ 
denote the coordinates of $t$ with respect to the 
decomposition
\[
\SS^1_{\alpha_1^\vee} \times \cdots \times \SS^1_{\alpha_r^\vee} \xrightarrow{\cong} \SS^1_{\alpha_1^\vee} \cdots  
\SS^1_{\alpha_r^\vee} = T\,,
\]
where $\SS^1_{\alpha_i^\vee}\leqslant T$ is the one-parameter subgroup determined by the coroot 
$\alpha_i^\vee$.\footnote{Not to be confused with our notation $\SS^1_{\mathbf{w}}$ for the circle group acting on $\SS^{2r+1}$ with weights $\mathbf{w}$.} Define
\begin{alignat*}{1}
\phi\co (A\times T)/{\approx} & \to \SS^{\ast(r+1)}/\SS^1_{\mathbf{n}^\vee} \\
[\mathbf{a},t] & \mapsto \langle a_0 e,a_1 t_1,\dots,a_r t_r\rangle \SS^1_{\mathbf{n}^\vee}\, ,
\end{alignat*}
where $e\in \SS^1$ is the identity element. We claim that $\phi$ is a continuous bijection of the compact 
space $(A\times T)/\approx$ with the Hausdorff space $\SS^{\ast(r+1)}/\SS^1_{\mathbf{n}^\vee}$, 
thus a homeomorphism.

To see that $\phi$ is well-defined it suffices to check that if $a_i=0$ for some $i\in\{0,\dots,r\}$, 
then $\phi([\mathbf{a},ts])=\phi([(\mathbf{a},t)])$ for any $s\in \SS^1_{\alpha_i^\vee}$ and $t\in T$. 
This is clear when $1\leqslant i\leqslant r$, because in this case $\phi([\mathbf{a},t])$ is independent 
of $t_i$ by the identifications made in the join construction. Now suppose that $a_0=0$ and let 
$s \in \SS^1_{\alpha_0^\vee} \leqslant T$. Since $-\alpha_0^\vee=\sum_{i=1}^rn_i^\vee\alpha_i^\vee$, 
we can parametrize $\SS_{\alpha_0^\vee}^1$ by 
$\lambda \mapsto (\lambda^{n_1^\vee},\dots,\lambda^{n_r^\vee})\in  \prod_{i=1}^r \SS^1_{\alpha_i^\vee}$, 
thus $s=(\lambda^{n_1^\vee},\dots,\lambda^{n_r^\vee})$ for some $\lambda\in \SS^1$. Then
\begin{equation*}
\begin{split}
\phi([(\mathbf{a},ts])& =\langle 0 e,a_1  t_1\lambda^{n_1^\vee},
\dots,a_r t_r\lambda^{n_r^\vee}\rangle \SS^1_{\mathbf{n}^\vee} \\
&=\langle 0 \lambda^{-n_0^\vee},a_1 t_1,\dots,a_r t_r\rangle \SS^1_{\mathbf{n}^\vee}\\&=\phi([(\mathbf{a},t])\, .
\end{split}
\end{equation*}
Here we used the weighted action of $\SS^1_{\mathbf{n}^\vee}$ in the second equality. 
We conclude that $\phi$ is well defined and continuous.

It is clear that $\phi$ is surjective, that is, that every element of $(\SS^1)^{\ast(r+1)}$ is 
equivalent to one of the form $\langle a_0 e,a_1 t_1,\dots,a_r t_r\rangle$ modulo $\SS^1_{\mathbf{n}^\vee}$.

To see that $\phi$ is injective suppose that $\phi([\mathbf{a},t])=\phi([\mathbf{a}',t'])$, that is, 
suppose that
\[
\langle a_0 e,a_1 t_1,\dots,a_r t_r\rangle \SS^1_{\mathbf{n}^\vee} 
=\langle a_0' e,a_1' t_1',\dots,a_r' t_r'\rangle \SS^1_{\mathbf{n}^\vee}\, .
\]
Then there exists $\lambda \in \SS^1$ such that
\begin{equation} \label{eq:joinequal}
\langle a_0 e,a_1 t_1',\dots,a_r t_r'\rangle
=\langle a_0 \lambda ,a_1 t_1 \lambda^{n_1^\vee},\dots,a_r t_r\lambda^{n_r^\vee}\rangle\, .
\end{equation}
First, this implies that $\mathbf{a}=\mathbf{a}'$ by the properties of the join construction. 
In addition, we claim that $t^{-1}t'\in T(\tilde{\Delta}(\mathbf{a}))$. From this it will follow that 
$(\mathbf{a},t)\approx (\mathbf{a}',t')$, hence that $\phi$ is injective. As a subset of 
$\{0,\dots,r\}$ the set $\tilde{\Delta}(\mathbf{a})$ consists of those $i$ 
such that $a_i=0$. Consider first the case $0\not\in \tilde{\Delta}(\mathbf{a})$, that is, 
$a_0\neq 0$. From (\ref{eq:joinequal}) we deduce that $\lambda=e$, and $t_i=t_i'$ for all 
$i\not\in \tilde{\Delta}(\mathbf{a})$. Therefore, $t^{-1}t'\in T(\tilde{\Delta}(\mathbf{a}))$. 
In the case $a_0=0$, we deduce that $t_i^{-1}t_i'=\lambda^{n_i^\vee}$ for all 
$i\not\in \tilde{\Delta}(\mathbf{a})$. Therefore, letting 
$s=(\lambda^{n_1^\vee},\dots,\lambda^{n_r^\vee})\in \SS^1_{\alpha_0^\vee}$, we have that 
$t^{-1}t's^{-1}\in T(\tilde{\Delta}(\mathbf{a})\backslash\{0\})$, hence 
$t^{-1}t'\in T(\tilde{\Delta}(\mathbf{a}))$. This completes the proof that $\phi$ is a homeomorphism.

The top horizontal map in the diagram is obtained by restriction of $\phi$. The subspace 
$A(p)\otimes_{\I(p)} \bar{F}$ of $A\otimes_{\I} \bar{F}$ consists of those $[\mathbf{a},t]$ with 
$\mathbf{a}\in A(p)$. By definition, $\mathbf{a}\in A(p)$ is equivalent to $a_i=0$ for all $i$ such that 
$p\nmid n_i^\vee$. Thus $\phi$ maps $A(p)\otimes_{\I(p)} \bar{F}$ homeomorphically onto the subspace
\[
\left\{ \langle a_0t_0,\dots,a_rt_r\rangle \SS_{\mathbf{n}^\vee}^1\in 
\SS^{\ast(r+1)}/\SS^1_{\mathbf{n}^\vee} \mid a_i=0 
\textnormal{ if }p\nmid n_i^\vee\right\}\subseteq \SS^{\ast(r+1)}/\SS^1_{\mathbf{n}^\vee}\,.
\]
This space is homeomorphic to the weighted projective space $\mathbb{CP}(\mathbf{n}^\vee(p))$. 
By inspection, one identifies the inclusion into $\SS^{\ast(r+1)}/\SS^1_{\mathbf{n}^\vee}$ with the map 
$\iota_p\co \mathbb{CP}(\mathbf{n}^\vee(p))\to \mathbb{CP}(\mathbf{n}^\vee)$ described in 
Proposition \ref{prop:weightedprojective}.
\end{proof}

\subsection{The cases $E_7$ and $E_8$ and the proofs of Theorems \ref{thm:mainpairs2intro} and \ref{thm:mainpairsintro}}  \label{sec:e7e8}

Thus far we have determined $H_\ast^G(X_G(p);(\underline{\pi}_0)_{(p)})$ under the assumption that 
$G\not\in \{E_7,E_8\}$ or $p>2$. In this subsection we calculate additional homology groups for 
$G=E_7$ and $G=E_8$ which are needed to prove Theorem \ref{thm:mainpairsintro}.

\begin{lemma} \label{lem:e7e8homology}
Suppose that $G=E_7$ or $G=E_8$. Then
\[
H_k^G(X_G(2);(\underline{\pi}_0)_{(2)}) \cong \begin{cases} \Z/4\,, & \textnormal{ if }k=0, \\ 
0 \,, & \textnormal{ if }k=1. 
\end{cases}
\]
\end{lemma}
\begin{proof}
To start observe that Lemma \ref{lem:e7e8rg4} and Proposition \ref{prop:weightedprojective} imply that  
\begin{equation}\label{caseX(4)}
H_k^G(X_G(4);(\underline{\pi}_0)_{(2)}) \cong \begin{cases} \Z/4\,, & \textnormal{ if }k=0, \\ 
0 \,, & \textnormal{ if }k=1. 
\end{cases}
\end{equation}
In what follows we are going to prove that 
\begin{equation}\label{relativeBredon}
H_{k}^{G}(X_G(2),X_G(4);(\underline{\pi}_0)_{(2)})=0 \ \ \text{ for } \ \ k=0, 1.
\end{equation} 
The result is then obtained using (\ref{caseX(4)}) and (\ref{relativeBredon}) in the 
long exact sequence in Bredon homology associated to the pair $(X_G(2),X_G(4))$
\[
\cdots\to H^{G}_{k}(X_G(4);(\underline{\pi}_0)_{(2)})\to H_{k}^{G}(X_G(2);(\underline{\pi}_0)_{(2)}) 
\to H_{k}^{G}(X_G(2),X_G(4);(\underline{\pi}_0)_{(2)})\to \cdots
\]
A standard argument using excision and the long exact sequence associated to the pair 
$(X_G(2)/X_G(4),X_G(4)/X_G(4))$ shows that (\ref{relativeBredon}) follows from
\[
H_{k}^{G}(X_G(2)/X_G(4);(\underline{\pi}_0)_{(2)})=0 \ \ \text{ for } \ \ k=0, 1.
\] 
We show this next. Let $\mathcal{Q}$  denote the coefficient system whose value 
is $\Z/2$ everywhere, except for $G/G$ where we set $\mathcal{Q}(G/G)=0$, and for which each map $\Z/2\to \Z/2$ is an isomorphism.
The coefficient systems $(\underline{\pi}_0)_{(2)}$ and $\mathcal{Q}$ agree when 
restricted to the quotient complex $X_G(2)/X_G(4)$, hence 
\[
H_{k}^{G}(X_G(2)/X_G(4);(\underline{\pi}_0)_{(2)})\cong H_{k}^{G}(X_G(2)/X_G(4);\mathcal{Q}).
\] 
Therefore, the proof of the lemma reduces to showing that $H_{k}^{G}(X_G(2)/X_G(4);\mathcal{Q})$ vanishes 
when $k=0, 1$. To see this for $k=0$ we observe that the coefficient system $\mathcal{Q}$ satisfies the 
conditions of Lemma \ref{computation spectral sequence}. Moreover, $X_G(2)/X_G(4)$ is path--connected 
as the same is true for $X_G(2)$, and the basepoint $X_G(4)/X_G(4)$ is fixed by the 
$G$-action. Lemma \ref{computation spectral sequence} then implies that 
\begin{equation}\label{stepk=0}
H_{0}^{G}(X_G(2)/X_G(4);\mathcal{Q})=0. 
\end{equation}
To finish we handle the case $k=1$. For this notice that Proposition \ref{prop:weightedprojective} 
implies that, up to homotopy equivalence, both $R_G(2)$ and $R_G(4)$ can be identified with 
weighted projective spaces. Therefore, for the constant coefficient $\underline{\Z/2}$ 
we have 
\begin{equation}\label{trivialstep1}
H_{k}^{G}(X_G(2)/X_G(4);\underline{\Z/2}) \cong 
H_{k}(R_G(2)/R_G(4);\Z/2)\cong \begin{cases} \Z/2\,, & \textnormal{ if }k=0, \\ 
0 \,, & \textnormal{ if }k=1. 
\end{cases}
\end{equation}
Let $\K$ denote the coefficient system which is 
$0$ everywhere, except for $G/G$ where we set $\K(G/G)= \Z/2$. Then we have a short exact sequence 
of coefficient systems 
\[
0\to \K\to \underline{\Z/2}\to \mathcal{Q}\to 0.
\]
Using (\ref{stepk=0}), (\ref{trivialstep1}) and the long exact sequence in Bredon 
homology associated to the above short exact sequence we obtain the short exact sequence 
\begin{equation}\label{longesE}
0\to H_{1}^{G}(X_G(2)/X_G(4);\mathcal{Q}) \to H_{0}^{G}(X_G(2)/X_G(4);\K)
\to  H_{0}^{G}(X_G(2)/X_G(4);\underline{\Z/2}) \to  0.
\end{equation}
By definition of $\K$, $H_{*}^{G}(X_G(2)/X_G(4);\K)$ is the $\Z/2$-homology of the space of $G$-fixed points of $X_G(2)/X_G(4)$. The only point fixed is the basepoint,  so $H_{0}^{G}(X_G(2)/X_G(4);\K)\cong \Z/2$. From the exact sequence (\ref{longesE}) we know that 
\[
H_{0}^{G}(X_G(2)/X_G(4);\K)\to H_{0}^{G}(X_G(2)/X_G(4);\underline{\Z/2})\cong \Z/2
\]
is surjective, hence an isomorphism. Exactness of (\ref{longesE}) thus implies that 
\[
H_{1}^{G}(X_G(2)/X_G(4);\mathcal{Q})=0\,,
\]
as we wanted to show.
\end{proof}

With the calculations of the previous lemma we have gathered all the necessary pieces to 
prove

\begin{theorem} \label{thm:mainpairs}
Let $G$ be a simply--connected and simple compact Lie group. Then
\[
\pi_2(\Hom(\Z^2,G)) \cong \Z\,,
\]
and on this group the quotient map $\Hom(\Z^2,G)\to \Rep(\Z^2,G)$ induces multiplication by the Dynkin index $\textnormal{lcm}\{n_0^\vee,\dots,n^\vee_r\}$ where 
$n_0^\vee,\dots,n_r^\vee \geqslant 1$ are the coroot integers of $G$.
\end{theorem}
\begin{proof}
To calculate $\pi_2(\Hom(\Z^2,G))$ our starting point is Lemma \ref{lem:pretheorem}, where we 
showed that $\pi_2(\Hom(\Z^2,G))\cong \Z\oplus E_{1,1}^2$. By Corollary \ref{cor:decomposition}, 
the group $E_{k,1}^2$ splits for each $k\geqslant 0$ into a direct sum
\[
E_{k,1}^2=H_k^G(\Hom(\Z^2,G);\underline{\pi}_0)\cong 
\bigoplus_{p\in \mathcal{P}} H_k^G(X_G(p);(\underline{\pi}_0)_{(p)})\,,
\]
where $\mathcal{P}$ is the set of primes dividing a coroot integer of $G$. For $k=1$ each direct 
summand is trivial, either by Corollary \ref{cor:homologyrp} or by Lemma \ref{lem:e7e8homology}, 
hence $E_{1,1}^2=0$. This proves that $\pi_2(\Hom(\Z^2,G))\cong \Z$.

Lemma \ref{lem:pretheorem} also showed that the degree\footnote{By \textit{degree} of a map $\Z\to \Z$ we will always mean its absolute value.} of 
$\pi_\ast\co \pi_2(\Hom(\Z^2,G))\to \pi_2(\Rep(\Z^2,G))$ equals the order of the finite group $E_{0,1}^2$. 
When $G$ is neither $E_7$ nor $E_8$, then $E_{0,1}^2\cong \bigoplus_{p \in \mathcal{P}}\Z/p$ by 
Corollary \ref{cor:homologyrp}. On the other hand, when $G=E_7$ or $G=E_8$, then 
Lemma \ref{lem:e7e8homology} implies that 
$E_{0,1}^2\cong\Z/4\oplus \bigoplus_{p\in\mathcal{P}\backslash\{2\}}\Z/p$. 
By inspection (see Table \ref{table:coroot}), the order of $E_{0,1}^2$ equals 
$\textnormal{lcm}\{n_0^\vee,\dots,n_r^\vee\}$, and the proof of the theorem is complete.
\end{proof}

Combining Proposition \ref{prop:weightedprojective} and Theorem \ref{thm:mainpairs} we derive the following more general result:

\begin{theorem} \label{thm:mainpairs2}
Suppose that $G$ is a semisimple compact connected Lie group. Then there is an extension
\[
0\to \Z^s \to H_2(\Hom(\Z^2,G)_{\BONE};\Z) \to H_2(\pi_1(G)^2;\Z)\to 0\,,
\]
where $s\geqslant 0$ is the number of simple factors in the Lie algebra of $G$.
\end{theorem}
\begin{proof}
Consider the Serre spectral sequence for the universal covering sequence
\[
\Hom(\Z^2,\tilde{G})\to \Hom(\Z^2,G)_{\BONE}\to B\pi_1(G)^2\,.
\]
Inspection of the proof of \cite[Lemma 2.2]{Goldman} shows that the action by deck translation of $\pi_1(G)^2$ on $\Hom(\Z^2,\tilde{G})$ is simply $(x,y)\mapsto (ax,by)$ where $a,b\in \pi_1(G)$ are viewed as elements in the center of $\tilde{G}$. We claim that this action is trivial on $H_2(\Hom(\Z^2,\tilde{G});\Z)$. Let us first treat the case when $G$ is simple. Then we can view $H_2(\Hom(\Z^2,\tilde{G});\Z)\cong \Z$ as a $\pi_1(G)^2$-submodule of $H_2(\Rep(\Z^2,\tilde{G});\Z)$ through Theorem \ref{thm:mainpairs}. The action of $Z(\tilde{G})$ on $\Rep(\Z^2,\tilde{G})$ given by translation of either coordinate is trivial on homology, because it extends to an action of a maximal torus on $\SS^{2r+1}/\SS^1$ under the identification in (\ref{eq:explicitequivalence}). Since $H_2(\Hom(\Z^2,\tilde{G});\Z)$ is torsion free, the differential $d_3\co H_3(\pi_1(G)^2;\Z)\to H_2(\Hom(\Z^2,\tilde{G});\Z)$ is trivial, and the extension follows.

The same argument applies when $G$ is semisimple, with the only difference that now $\Hom(\Z^2,\tilde{G})\cong \Hom(\Z^2,G_1)\times \cdots \times \Hom(\Z^2,G_s)$, hence $H_2(\Hom(\Z^2,\tilde{G});\Z)\cong \Z^s$, where $G_1,\dots,G_s$ are the simple factors in $\tilde{G}$.
\end{proof}

Let $G$ be simple, and suppose that the finite group $\pi_1(G)$ is not cyclic; then neither is $H_2(\pi_1(G)^2;\Z)$, and we deduce from Theorem \ref{thm:mainpairs2} that $H_2(\Hom(\Z^2,G)_{\BONE};\Z)$ necessarily has torsion. This happens for $G=PSO(4n)$ for any $n\geqslant 1$, since $\pi_1(PSO(4n))\cong \Z/2\oplus \Z/2$.

As another example consider $SO(3)$. One shows that $H_2(\Hom(\Z^2,SO(3))_{\BONE};\Z)\cong \Z$ (for example, using the description $\Hom(\Z^2,SO(3))_{\BONE}\cong (\SS^2\times_{\Z/2}(\SS^1)^2)/(\SS^2\times_{\Z/2}\ast)$ provided in \cite[Theorem 3.1]{STG}). In this case the extension is $\Z\stackrel{2}{\longrightarrow}\Z\to \Z/2$.\medskip

For $n\geqslant 2$ let $p_{i,j}\co \Hom(\Z^n,G)_{\BONE}\to \Hom(\Z^2,G)$ be the projection onto 
the $i$-th and $j$-th component. Denote by 
$p\co \Hom(\Z^n,G)_{\BONE}\to \prod_{1\leqslant i<j\leqslant n} \Hom(\Z^2,G)$ the map whose 
$(i,j)$-th component is $p_{i,j}$. For later reference we record

\begin{corollary} \label{cor:h2modtorsion}
Let $n\geqslant 2$. Then the map $p\co \Hom(\Z^n,G)_{\BONE}\to \prod^{\binom{n}{2}}\Hom(\Z^2,G)$ 
induces an isomorphism
\[
p_\ast \co \pi_2(\Hom(\Z^n,G)_{\BONE})/\textnormal{torsion} 
\xrightarrow{\cong} \bigoplus^{\binom{n}{2}} \pi_2(\Hom(\Z^2,G))\,.
\]
\end{corollary}
\begin{proof}
For each $1\leqslant i< j\leqslant n$ we have a natural inclusion of the $i$-th and $j$-th component
\[
I_{i,j}:\Hom(\Z^2,G)\to  \Hom(\Z^n,G)_{\BONE}
\]
which inserts the identity in all the other components. Observe that, since $\pi_2(G)=0$, for all $i,j,k,l$ we have that $(p_{i,j})_\ast\circ (I_{k,l})_\ast=0$, unless $i=k$ and $j=l$ in which case $(p_{i,j})_\ast\circ (I_{i,j})_\ast=id$. This implies that $p_{*}$ is surjective. 
By Corollary \ref{cor:rank} and Theorem \ref{thm:mainpairs}, both the domain (modulo torsion) and codomain of $p_\ast$ are free abelian groups of rank $\binom{n}{2}$. 
Since $p_{*}$ is a surjective map of free abelian groups of the same rank it must be 
an isomorphism. 
\end{proof}

\section{Stability for commuting pairs in Spin groups} \label{sec:stability}

In this section we study the stability behavior in the case of Spin groups. 

For $m\geqslant 3$ the group $Spin(m)$ is the universal covering group of $SO(m)$. 
The standard inclusion $SO(m)\hookrightarrow SO(m+1)$, given by block sum with a $1\times 1$ 
identity matrix, induces a map $Spin(m)\to Spin(m+1)$. The purpose of this section is to prove 
Theorem \ref{thm:spinstability} which asserts that for $m\geqslant 5$ the map
\begin{equation} \label{eq:stabilizationmapspin}
\Hom(\Z^2,Spin(m))\to \Hom(\Z^2,Spin(m+1))
\end{equation}
induces an isomorphism of second homotopy groups.

We begin by proving a general homology stability result for representation spaces of spinor 
groups which may be interesting on its own. Only a special case of it will be needed to prove 
Theorem \ref{thm:spinstabilityintro}. Let
\[
i_{\ell-1}^{\textnormal{ev}}\co \Hom(\Z^2,Spin(2\ell-2))\to \Hom(\Z^2,Spin(2\ell))
\]
denote the map obtained by iterating the stabilization map (\ref{eq:stabilizationmapspin}). 
Similarly, define 
\[
i_{\ell-1}^{\textnormal{odd}}\co \Hom(\Z^2,Spin(2\ell-1))\to \Hom(\Z^2,Spin(2\ell+1))\, .
\]
Let $\bar{i}_{\ell-1}^{\textnormal{ev}}$ and $\bar{i}_{\ell-1}^{\textnormal{odd}}$ 
denote the respective maps induced on representation spaces.

\begin{theorem} \label{thm:stabilityrepspin}
\leavevmode
\begin{itemize}
\item[(i)] When $\ell\geqslant 4$, then the map
\[
(\bar{i}_{\ell-1}^{\textnormal{ev}})_\ast\co H_k(\Rep(\Z^2,Spin(2\ell-2));\Z)\to H_k(\Rep(\Z^2,Spin(2\ell));\Z)
\]
is an isomorphism for all $k\leqslant 2\ell-7$ and multiplication by $2$ for $k=2\ell-6$. \smallskip
\item[(ii)] When $\ell\geqslant 3$, then the map
\[
(\bar{i}_{\ell-1}^{\textnormal{odd}})_\ast\co H_k(\Rep(\Z^2,Spin(2\ell-1));\Z)\to H_k(\Rep(\Z^2,Spin(2\ell+1));\Z)
\]
is an isomorphism for all $k\leqslant 2\ell-5$ and multiplication by $2$ for $k=2\ell-4$.
\end{itemize}
\end{theorem}

\begin{proof}
We begin by proving (i). For $\ell\geqslant 3$ the group $Spin(2\ell)$ is described by the root system 
$D_\ell$ (for $\ell=3$ there is an isomorphism $D_3\cong A_3$ yielding the exceptional isomorphism 
$Spin(6)\cong SU(4)$). Let $\{\alpha_1,\dots,\alpha_\ell\}$ be a set of simple roots. It may be chosen 
in such a way that, when $\ell\geqslant 4$, the subset $\{\alpha_2,\dots,\alpha_\ell\}$ determines a 
subroot system of $D_\ell$ of type $D_{\ell-1}$ and the image of $Spin(2\ell-2)$ in $Spin(2\ell)$ is 
the subgroup corresponding to this subroot system.

Now let $\ell\geqslant 4$. Let $\mathbf{n}^{\vee}_{\ell-1}=(1,1,2,\dots,2,1,1)\in\Z^{\ell}$ be the 
tuple of coroot integers of $Spin(2\ell-2)$ (of which $\ell-4$ entries are equal to $2$). By 
Proposition \ref{prop:weightedprojective}, the map $\bar{i}_{\ell-1}^{\textnormal{ev}}$ is 
equivalent to a map $f_{\ell-1} \co \mathbb{CP}(\mathbf{n}^\vee_{\ell-1}) \to \mathbb{CP}(\mathbf{n}^\vee_{\ell})$, 
so that $\deg((\bar{i}_{\ell-1}^{\textnormal{ev}})_\ast)=\deg((f_{\ell-1})_\ast)$. 
Let $(\mathbf{n}^\vee_{\ell-1})'\in \Z^{\ell-2}$ be the tuple consisting of the first $\ell-2$ entries of 
$\mathbf{n}^\vee_{\ell-1}$. We are going to describe $f_{\ell-1}$ explicitly, and show that 
it fits into a commutative diagram
\begin{equation} \label{dgr:wp12}
\xymatrix{
\mathbb{CP}((\mathbf{n}^\vee_{\ell-1})')\ar[d] \ar[dr] & \\
\mathbb{CP}(\mathbf{n}^\vee_{\ell-1}) \ar[r]^-{f_{\ell-1}} & \mathbb{CP}(\mathbf{n}^\vee_{\ell}).
}
\end{equation}
Here the map $\mathbb{CP}((\mathbf{n}^\vee_{\ell-1})')\to \mathbb{CP}(\mathbf{n}^\vee_{\ell-1})$ is the 
``standard inclusion'', described in homogeneous coordinates by 
$[z_0,\dots,z_{\ell-3}]\mapsto [z_0,\dots,z_{\ell-3},0,0]$, and 
$\mathbb{CP}((\mathbf{n}^\vee_{\ell-1})')\to \mathbb{CP}(\mathbf{n}^\vee_{\ell})$ is the ``standard inclusion'' 
$[z_0,\dots,z_{\ell-3}]\mapsto [z_0,\dots,z_{\ell-3},0,0,0]$.

Let us first show how the conclusion of the theorem is obtained from commutativity of the diagram. If $\mathbf{w}\in \Z^{r+1}_{>0}$, $r\geqslant 1$, and $p_{\mathbf{w}}\co \mathbb{CP}^r\to \mathbb{CP}(\mathbf{w})$ is the projection given by 
$[z_0,\dots,z_r]\mapsto [z_0^{w_0},\dots,z_r^{w_r}]$, then the degree of 
$(p_{\mathbf{w}})_\ast\co H_{2k}(\mathbb{CP}^r;\Z)\to H_{2k}(\mathbb{CP}(\mathbf{w});\Z)$ is given by
\begin{equation} \label{eq:degreeofmapofweightedprojectivespaces}
\textnormal{lcm}\left\{\frac{w_{i_0}\cdots w_{i_k}}{\textnormal{gcd}\{w_{i_0},\dots,w_{i_k}\}}
\mid 0\leqslant i_0 <\cdots < i_k\leqslant r \right\}\,,
\end{equation}
see \cite[p. 244]{K73}. This can be used to show that the inclusion $\mathbb{CP}((\mathbf{n}^\vee_{\ell-1})')\to \mathbb{CP}(\mathbf{n}^\vee_{\ell-1})$ induces homology 
isomorphisms in degrees $k\leqslant 2\ell-6$, whereas 
$\mathbb{CP}((\mathbf{n}^\vee_{\ell-1})')\to \mathbb{CP}(\mathbf{n}^\vee_\ell)$ induces isomorphisms in degrees 
$k\leqslant 2\ell-7$ and multiplication by $2$ in degree $k=2\ell-6$. Consequently, 
$f_{\ell-1}\co \mathbb{CP}(\mathbf{n}^\vee_{\ell-1})\to \mathbb{CP}(\mathbf{n}^\vee_\ell)$ 
induces isomorphisms in degrees $k\leqslant 2\ell-7$ and multiplication by $2$ in degree $k=2\ell-6$.

Now we construct $f_{\ell-1}$. Fix a maximal torus $T_\ell\leqslant Spin(2\ell)$, and let $W_\ell$ 
be the corresponding Weyl group. The maximal torus $T_{\ell-1}\leqslant Spin(2\ell-2)$ is the subtorus of 
$T_\ell$ generated by the one-parameter subgroups corresponding to the coroots 
$\alpha_2^\vee,\dots,\alpha_\ell^\vee$. Let $\theta_{\ell-1}\co T_{\ell-1}\hookrightarrow T_{\ell}$ 
denote the inclusion. The map $\bar{i}_{\ell-1}^{\textnormal{ev}}\co \Rep(\Z^2,Spin(2\ell-2))\to \Rep(\Z^2,Spin(2\ell))$ 
may be identified with the map $(T_{\ell-1}\times T_{\ell-1})/W_{\ell-1}\to (T_{\ell}\times T_{\ell})/W_{\ell}$ 
induced by $\theta_{\ell-1}$. Let $A_\ell$ denote the fundamental alcove determined by 
$\{\alpha_1,\dots,\alpha_\ell\}$, and let $A_{\ell-1}$ be the alcove in $\t_{\ell-1}$ determined by 
$\{\alpha_2,\dots,\alpha_\ell\}$. Our aim is to define a map $f_{\ell-1}$ making the following 
diagram commute:
\begin{equation} \label{dgr:fl}
\xymatrix{
(T_{\ell-1}\times T_{\ell-1})/W_{\ell-1} \ar[d]^-{\simeq} 
\ar[r]^-{\bar{i}_{\ell-1}^{\textnormal{ev}}} & (T_{\ell}\times T_{\ell})/W_{\ell}  \ar[d]^-{\simeq}  \\
(A_{\ell-1}\times T_{\ell-1})/{\approx} \ar[r]^-{f_{\ell-1}}  & (A_{\ell}\times T_{\ell})/{\approx}
}
\end{equation}
The spaces in the bottom row are homeomorphic to weighted projective spaces, see Lemma \ref{lem:coendwp}. 
The vertical maps are the equivalences established in the course of proving 
Proposition \ref{prop:weightedprojective}. For example, the right hand map takes a 
$W_\ell$-orbit $((t_1,t_2))$ to the equivalence class under $\approx$ of any representative 
$(x,t)$ of $((t_1,t_2))$ with $x\in A_\ell$ and $t\in T_\ell$.

The inclusion $\theta_{\ell-1}\co T_{\ell-1}\hookrightarrow T_\ell$ descends to a map 
$T_{\ell-1}/W_{\ell-1}\to T_\ell/W_\ell$, hence defines a map of alcoves 
$\bar{\theta}_{\ell-1}\co A_{\ell-1}\to A_\ell$. To define $f_{\ell-1}$ we begin by describing 
$\bar{\theta}_{\ell-1}$ in barycentric coordinates. For this we will work with explicit formulae 
for the root system $D_\ell$ which can be found, for example, in \cite[Plate IV]{Bourbaki46}. 
Note that $D_\ell$ is self-dual, which means that the expressions for roots and weights shown 
in \cite{Bourbaki46} also represent the coroots and coweights, respectively.

We identify the Lie algebra $\mathfrak{t}_\ell$ of $T_\ell$ with $\R^\ell$. With respect to the 
standard basis $\{e_1,\dots,e_\ell\}$ of $\R^\ell$, the simple coroots can be chosen to be
\begin{equation} \label{eq:corootsstandardbasis}
\alpha_i^\vee=e_i-e_{i+1}\;\textnormal{ for }\;i=1,\dots,\ell-1\,,\;\textnormal{ and }\; 
\alpha^\vee_\ell=e_{\ell-1}+e_{\ell}\,.
\end{equation}
The Weyl group $W_\ell$ acts on $\t_\ell$ by signed permutations $e_i\mapsto \pm e_{\sigma(i)}$ 
($\sigma\in \Sigma_\ell$) such that the total number of negative signs is even.  
Note that $\mathfrak{t}_{\ell-1}=\textnormal{span}(\alpha_2^\vee,\dots,\alpha_\ell^\vee)
=\textnormal{span}(e_2,\dots,e_\ell) \subseteq \t_\ell$.

Recall from \cite[VI \S 2.2 Corollary of Proposition 5]{Bourbaki46} that, as a subspace of 
$\mathfrak{t}_\ell$, $A_{\ell}$ is the convex hull
\[
A_{\ell}=\textnormal{Conv}(0,\omega^\vee_1/n_1,\dots,\omega^\vee_\ell/n_\ell)\,,
\]
where $\omega_i^\vee$ is the $i$-th fundamental coweight (defined by $\alpha_j(\omega_i^\vee)=\delta_{ij}$) 
and $n_i$ is the root integer associated to $\alpha_i$ (which in the case of $D_\ell$ equals 
the $i$-th coroot integer $n_i^\vee$). Let us write $v_i:=\omega_i^\vee/n_i$, so that 
$\{0,v_1,\dots,v_\ell\}$ is the set of vertices of $A_\ell$. Let $\{0,u_1,\dots,u_{\ell-1}\}$ 
denote the set of vertices of $A_{\ell-1}\subseteq \t_{\ell-1}\subseteq \t_\ell=\R^\ell$. 
Using \cite[Plate IV]{Bourbaki46} we can write the vertices in terms of the standard basis of $\R^\ell$. 
The result is displayed in Table \ref{table:vertices}.

\begin{table}[h]
\centering
\def\arraystretch{1.5}
\begin{tabular}{ll}
$u_1=e_2$  &  $v_1=e_1$ \\
$u_2=\frac{1}{2}\left(e_2+e_3\right)$   & $v_2=\frac{1}{2}\left(e_1+e_2\right)$ \\
\vdots & \vdots \\
$u_{\ell-3}= \frac{1}{2}\left(e_2+\cdots + e_{\ell-2}\right)$ & $v_{\ell-3}
=\frac{1}{2}\left(e_1+\cdots+e_{\ell-3}\right)$  \\
$u_{\ell-2}= \frac{1}{2}\left(e_2+\cdots + e_{\ell-1}-e_\ell\right)$  & 
$v_{\ell-2}=\frac{1}{2}\left(e_1+\cdots+e_{\ell-2}\right)$ \\
$u_{\ell-1}= \frac{1}{2}\left(e_2+\cdots + e_{\ell-1}+e_\ell\right)$  & 
$v_{\ell-1}=\frac{1}{2}\left(e_1+\cdots+e_{\ell-1}-e_{\ell}\right)$  \\
&   $v_{\ell}=\frac{1}{2}\left(e_1+\cdots+e_{\ell-1}+e_\ell\right)$
\end{tabular}
\vspace{10pt}
\caption{Vertices of $A_{\ell-1}=\textnormal{Conv}(0,u_1,\dots,u_{\ell-1})$ and 
$A_{\ell}=\textnormal{Conv}(0,v_1,\dots,v_\ell)$.} \label{table:vertices}
\end{table}

Let $x\in A_{\ell-1}$. Then $\bar{\theta}_{\ell-1}(x)=wx$, where $w\in W_\ell$ is any element 
such that $w x\in A_\ell$. Let $\sigma^+\in W_\ell$ be the cyclic permutation mapping 
$e_i\mapsto e_{i-1}$ for $i=2,\dots,\ell$, and $e_1\mapsto e_\ell$. Let $\sigma^-\in W_\ell$ be 
the element whose underlying permutation is $\sigma^+$, but which sends $e_1\mapsto -e_\ell$ and 
$e_\ell\mapsto -e_{\ell-1}$. One can check that $\sigma^+ u_i=\sigma^-u_i=v_i$ for all 
$1\leqslant i\leqslant \ell-3$, that $\sigma^+(u_{\ell-2}+u_{\ell-1})
=\sigma^-(u_{\ell-2}+u_{\ell-1})=2v_{\ell-2}$, and that $\sigma^+ u_{\ell-1}=\sigma^- u_{\ell-2}
=(v_{\ell-1}+v_{\ell})/2$.

Now suppose that $x=a_1u_1+\dots+a_{\ell-1}u_{\ell-1} \in A_{\ell-1}$ with $a_i\geqslant 0$ and 
$a_1+\cdots+a_{\ell-1}\leqslant 1$. If $a_{\ell-2}\leqslant a_{\ell-1}$, then, writing
\[
x=a_1u_1+\dots+a_{\ell-3}u_{\ell-3}+a_{\ell-2}(u_{\ell-2}+u_{\ell-1})+(a_{\ell-1}-a_{\ell-2})u_{\ell-1}\,,
\]
we see that
\[
\sigma^+ x=a_1v_1+\cdots+a_{\ell-3}v_{\ell-3}+2a_{\ell-2}v_{\ell-2}+\frac{a_{\ell-1}-a_{\ell-2}}{2}(v_{\ell-1}+v_{\ell})\, ,
\]
which is a convex combination of $0,v_1,\dots,v_\ell$, hence a point in $A_\ell$. On the other hand, 
if $a_{\ell-2}>a_{\ell-1}$, then we find in a similar way that
\[
\sigma^- x=a_1v_1+\cdots+a_{\ell-3}v_{\ell-3}+2a_{\ell-1}v_{\ell-2}+\frac{a_{\ell-2}-a_{\ell-1}}{2}(v_{\ell-1}+v_{\ell})\,,
\]
which is again in $A_{\ell}$. From this we derive a description of the map 
$\bar{\theta}_{\ell-1}\co A_{\ell-1}\to A_\ell$ in barycentric coordinates:
\[
(a_0,\dots,a_{\ell-1})\stackrel{\bar{\theta}_{\ell-1}}{\longmapsto} 
\left(a_0,\dots,a_{\ell-3},2\min\{a_{\ell-2},a_{\ell-1}\},
\frac{|a_{\ell-1}-a_{\ell-2}|}{2},\frac{|a_{\ell-1}-a_{\ell-2}|}{2}\right)\, .
\]

Now let $[x,t]\in (A_{\ell-1}\times T_{\ell-1})/{\approx}$, and let $(a_0,\dots,a_{\ell-1})$ be 
the barycentric coordinates of $x$. Since we want the diagram (\ref{dgr:fl}) to commute, 
we are forced to set
\[
f_{\ell-1}([x,t]):=[\bar{\theta}_{\ell-1}(x),w\theta_{\ell-1}(t)]\,,
\]
where $w=\sigma^+$ if $a_{\ell-2}\leqslant a_{\ell-1}$, and $w=\sigma^-$ if $a_{\ell-2}> a_{\ell-1}$. 
Let $(t_1,\dots,t_{\ell-1})\in (\SS^1)^{\ell-1}$ be the coordinates of $t\in T_{\ell-1}$ 
with respect to $\alpha_2^\vee,\dots,\alpha_\ell^\vee$. Then it is easily verified, using the 
action of $W_\ell$ on the coroots $\alpha_1^\vee,\dots,\alpha_\ell^\vee$ 
(see (\ref{eq:corootsstandardbasis})), that
\[
\sigma^+\theta_{\ell-1}(t)= (t_1,\dots,t_{\ell-3},t_{\ell-1}t_{\ell-2},t_{\ell-1},t_{\ell-1})\,,
\quad \sigma^-\theta_{\ell-1}(t)=(t_1,\dots,t_{\ell-3},t_{\ell-1}t_{\ell-2},t_{\ell-2},t_{\ell-2})\,.
\]
Now it is not difficult to check that $f_{\ell-1}$ is well defined and continuous 
(for the definition of the equivalence relation $\approx$ see the proof of Lemma \ref{lem:coendwp}). 
Moreover, the diagram (\ref{dgr:fl}) commutes by construction.

Using the homeomorphism $\phi$ defined in the proof of Lemma \ref{lem:coendwp} we could give a 
description of $f_{\ell-1} \co \mathbb{CP}(\mathbf{n}^\vee_{\ell-1})\to \mathbb{CP}(\mathbf{n}^\vee_\ell)$ 
in homogeneous coordinates. Relevant to our proof, however, is merely the observation that the 
restriction of $f_{\ell-1}$ to the first $\ell-2$ homogenous coordinates of 
$\mathbb{CP}(\mathbf{n}^\vee_{\ell-1})$ (i.e., setting $a_{\ell-2}=a_{\ell-1}=0$) equals the standard 
inclusion $[z_0,\dots,z_{\ell-3}]\mapsto [z_0,\dots,z_{\ell-3},0,0,0]$; but this is evident from the 
description of $\bar{\theta}_{\ell-1}$ and $w\theta_{\ell-1}$ given above, as these coordinates 
depend only on $a_0,\dots,a_{\ell-3}$ and $t_1,\dots,t_{\ell-3}$. This proves part (i) of the theorem.

Part (ii) for $\ell\geqslant 4$ is proved in the same way. The group $Spin(2\ell+1)$ is described 
by the root system $B_{\ell}$. A basis $\{\alpha_1,\dots,\alpha_\ell\}$ of simple roots can be chosen 
in such a way that the subgroup $Spin(2\ell-1)$ corresponds to the subroot system $B_{\ell-1}$ 
obtained by omitting the simple root $\alpha_1$. A calculation as before shows that the stabilization 
map $\bar{i}_{\ell-1}^{\textnormal{odd}}$ is equivalent to the one induced by the map 
$A_{\ell-1}\times T_{\ell-1}\to A_{\ell}\times T_\ell$ sending
\[
((a_0,\dots,a_{\ell-1}),(t_1,\dots,t_{\ell-1}))\mapsto 
((a_0,\dots,a_{\ell-2},a_{\ell-1},0),(t_1,\dots,t_{\ell-2},t_{\ell-1}^2,t_{\ell-1}))\, .
\]
In particular, the restriction to the first $\ell-1$ homogeneous coordinates is equal to the standard inclusion. 
The conclusion follows again from \cite{K73} and a commutative diagram similar to (\ref{dgr:wp12}). 
For the case $\ell=3$ see Remark \ref{rem:spin5spin7}.
\end{proof}

\begin{remark} \label{rem:oddspin}
Theorem \ref{thm:stabilityrepspin} should be compared with the homology stability result of 
Ramras and Stafa, \cite[Theorem 1.1]{RS2}. From their result one deduces that the rational homology groups of 
$\Rep(\Z^2,Spin(2\ell-1))$ and $\Rep(\Z^2,Spin(2\ell+1))$, and of $\Rep(\Z^2,Spin(2\ell-2))$ and 
$\Rep(\Z^2,Spin(2\ell))$, respectively, are abstractly isomorphic up to homological degree 
$k\leqslant \ell-1$.
\end{remark}

Next we prove stability of $\pi_2$ for spaces of commuting pairs in spinor groups. As in the previous theorem, it is natural to divide 
the analysis into the case of even and odd Spin groups; it will be enough, however, to treat the even 
case from which the general case may be deduced.

\begin{theorem} \label{thm:spinstability}
For all $m\geqslant 5$ the map $Spin(m)\to Spin(m+1)$ induces an isomorphism
\[
\pi_2(\Hom(\Z^2,Spin(m)))\xrightarrow{\cong} \pi_2(\Hom(\Z^2,Spin(m+1)))\, .
\]
\end{theorem}
\begin{proof}
We first focus on the range $m\geqslant 6$. The case $m=5$ will be dealt with separately at the end. 
Let $\ell\geqslant 4$. By Theorem \ref{thm:mainpairs}, all three groups in the sequence
\[
\pi_2(\Hom(\Z^2,Spin(2\ell-2))) \to \pi_2(\Hom(\Z^2,Spin(2\ell-1)))\to \pi_2(\Hom(\Z^2,Spin(2\ell)))
\]
are isomorphic to $\Z$. Thus, if the composite map is an isomorphism, then so are the two component maps. 
To prove the theorem in the range $m\geqslant 6$ it therefore suffices to show that for all 
$\ell\geqslant 4$ the map $i_{\ell-1}^{\textnormal{ev}}\co 
\Hom(\Z^2,Spin(2\ell-2))\to \Hom(\Z^2,Spin(2\ell))$ induces an isomorphism of second homotopy groups.

Consider the commutative diagram
\[
\xymatrix{
\pi_2(\Hom(\Z^2,Spin(2\ell-2)))\ar[r]^-{(i_{\ell-1}^{\textnormal{ev}})_\ast} 
\ar[d]^-{(\pi_{\ell-1})_\ast} & \pi_2(\Hom(\Z^2,Spin(2\ell))) \ar[d]^-{(\pi_{\ell})_\ast} \\
\pi_2(\Rep(\Z^2,Spin(2\ell-2))) \ar[r]^-{(\bar{i}_{\ell-1}^{\textnormal{ev}})_\ast} & \pi_2(\Rep(\Z^2,Spin(2\ell)))
}
\]
in which the maps are the obvious ones. As all groups in the diagram are isomorphic to $\Z$, 
each map is determined by its degree. The degrees of the vertical maps are described by 
Theorem \ref{thm:mainpairs}, from which we deduce that
\begin{equation*} \label{eq:deg}
\deg((i_{\ell-1}^{\textnormal{ev}})_\ast)=\frac{\deg((\bar{i}_{\ell-1}^{\textnormal{ev}})_\ast) 
\deg((\pi_{\ell-1})_\ast)}{\deg((\pi_{\ell})_\ast)}=\begin{cases} 
\deg((\bar{i}_{\ell-1}^{\textnormal{ev}})_\ast)/2 & \textnormal{if } \ell=4, \\ 
\deg((\bar{i}_{\ell-1}^{\textnormal{ev}})_\ast) & \textnormal{if } \ell>4. \end{cases}
\end{equation*}
Theorem \ref{thm:stabilityrepspin} implies that $\deg(\bar{i}_{\ell-1}^{\textnormal{ev}})_\ast=1$ 
if $\ell >4$ and $\deg(\bar{i}_{\ell-1}^{\textnormal{ev}})_\ast=2$ if $\ell=4$, hence 
$\deg((i_{\ell-1}^{\textnormal{ev}})_\ast)=1$ for all $\ell\geqslant 4$. 
This proves the theorem in the range $m\geqslant 6$.

It remains to prove that $Spin(5)\to Spin(6)$ induces an isomorphism as stated. It is enough to show 
that the map $\pi_2(\Hom(\Z^2,Spin(5)))\to \pi_2(\Hom(\Z^2,Spin(6)))$ is surjective as both groups are 
isomorphic to $\Z$. Consider the following sequence of matrix groups
\[
\xymatrix{
Spin(4) \ar[r] & Spin(5) \ar[r] & Spin(6) \\
SU(2)\times SU(2) \ar@{}[u]|{\rotatebox{90}{$\cong$}} & Sp(2) \ar@{}[u]|{\rotatebox{90}{$\cong$}} & 
SU(4) \ar@{}[u]|{\rotatebox{90}{$\cong$}}
}
\]
Recall that $Sp(k)=\{A\in GL(k,\mathbb{H})\mid A^\dagger A=\BONE\}$. In particular, $Sp(1)\cong SU(2)$ 
is the group of quaternions of unit norm. It is wellknown, see \cite[Theorem 5.20]{MT91}, that the first
 map in the sequence is block sum $Sp(1)^2\to Sp(2)$, $(A,B)\mapsto A\oplus B$. The second map is the 
inclusion $Sp(2)\hookrightarrow SU(4)$ resulting from
\[
M(2,\mathbb{H})\to M(4,\C)\,,\; A+\mathbf{j}B \mapsto 
\begin{pmatrix} A & -\bar{B} \\ B & \bar{A} \end{pmatrix}\, ,
\]
see \cite[Theorem 5.21]{MT91}. There is a permutation $P\in \Sigma_4\leqslant U(4)$ such that the 
composition of $SU(2)^2\to Sp(2) \to SU(4)$ is conjugate by $P$ to block sum. As $U(4)$ is 
path--connected the composition is homotopic to block sum. Then its restriction to the first 
$SU(2)$-factor is homotopic to the canonical inclusion of $SU(2)$ into $SU(4)$, which by 
Theorem \ref{thm:casesymplecticandunitary} and Remark \ref{rem:symplecticandunitary} induces 
an isomorphism $\pi_2(\Hom(\Z^2,SU(2)))\cong \pi_2(\Hom(\Z^2,SU(4)))$. As a consequence, 
the map $\pi_2(\Hom(\Z^2,Spin(4)))\to \pi_2(\Hom(\Z^2,Spin(6)))$ is surjective, hence so is 
$\pi_2(\Hom(\Z^2,Spin(5)))\to \pi_2(\Hom(\Z^2,Spin(6)))$.
\end{proof}

\begin{remark} \label{rem:spin5spin7}
To prove the case $\ell=3$ of Theorem \ref{thm:stabilityrepspin} (ii) we consider the sequence of maps
\[
\xymatrix{
\Rep(\Z^2,Spin(6)) \ar[r] \ar@/^1.4pc/[rr]^-{2} & \Rep(\Z^2,Spin(7)) 
\ar[r] \ar@/_1.4pc/[rr]_-{1} & \Rep(\Z^2,Spin(8)) \ar[r] & \Rep(\Z^2,Spin(9)) \\
}
\]
The two labels indicate the degree of the induced map of second homotopy groups as implied 
by the case $\ell=4$ of Theorem \ref{thm:stabilityrepspin}. It follows that the first map in 
the sequence has degree $2$ on second homotopy groups. Now 
$\pi_2(\Rep(\Z^2,Spin(5)))\to \pi_2(\Rep(\Z^2,Spin(6)))$ is an isomorphism, by Theorem 
\ref{thm:spinstability}, hence $\pi_2(\Rep(\Z^2,Spin(5)))\to \pi_2(\Rep(\Z^2,Spin(7)))$ must have degree 2.
\end{remark}

\section{The distinguished role of $SU(2)$} \label{sec:rolesu2}

In this section we continue to work with a fixed simply--connected simple compact Lie group $G$. Let
\[
\nu \co SU(2)\hookrightarrow G
\]
be the embedding that corresponds to the highest root of $G$. Then the induced map
\[
\nu_\ast\co H_3(SU(2);\Z)\xrightarrow{\cong} H_3(G;\Z)
\]
is an isomorphism, see \cite[III Proposition 10.2]{BS52}.

\begin{theorem} \label{thm:su2represents}
The map $\nu \co SU(2)\hookrightarrow G$ induces an isomorphism
\[
\pi_2(\Hom(\Z^2,SU(2)))\cong \pi_2(\Hom(\Z^2,G))\, .
\]
\end{theorem}
\begin{proof}
A standard argument with the semi-algebraic triangulation theorem \cite[Theorem 9.2.1]{BCR} shows 
that we can give $G^2=G\times G$ a CW-structure in such a way that $\Hom(\Z^2,G)$ is a subcomplex, 
hence the inclusion $i\co \Hom(\Z^2,G)\to G^2$ is a cofibration. Let us consider the 
Puppe sequence of $i$,
\[
\Hom(\Z^2,G)\xrightarrow{i} G^2\to G^2 /\Hom(\Z^2,G) \xrightarrow{\partial} 
\Sigma \Hom(\Z^2,G)\xrightarrow{-\Sigma i} \Sigma G^2\, .
\]
By the K{\"u}nneth theorem, and the fact that $G$ is $2$-connected, we have that
\[
H_3(\Sigma G^2;\Z)\cong H_2(G^2;\Z)=0\,.
\]
Moreover, if $j\co G\vee G\to \Hom(\Z^2,G)$ is the inclusion, then the composite map
\[
H_3(G\vee G;\Z)\xrightarrow{j_\ast} H_3(\Hom(\Z^2,G);\Z) \xrightarrow{i_\ast} H_3(G^2;\Z)
\]
is an isomorphism. Hence, $i_\ast$ is surjective. From the long exact homology sequence 
associated to the Puppe sequence we derive the isomorphism
\[
\partial_\ast\co H_3(G^2/\Hom(\Z^2,G);\Z)\xrightarrow{\cong} H_2(\Hom(\Z^2,G);\Z)\, . 
\]
Let $\overline{\nu\times\nu}\co SU(2)^2/\Hom(\Z^2,SU(2))\to G^2/\Hom(\Z^2,G)$ be the map 
induced by $\nu\co SU(2)\hookrightarrow G$ upon passage to quotients.  By naturality 
of the connecting map $\partial$ and by the Hurewicz theorem, it suffices to show that 
there is an isomorphism
\[
(\overline{\nu\times \nu})_\ast \co H_3(SU(2)^2/\Hom(\Z^2,SU(2));\Z)\xrightarrow{\cong} 
H_3(G^2/\Hom(\Z^2,G);\Z)\, .
\]

To see this we consider the commutator map $G^2\to G$, $(x,y)\mapsto [x,y]$. It is 
constant on commuting pairs, hence induces a map
\[
\gamma\co G^2/\Hom(\Z^2,G)\to G\, .
\]
We claim that when $G=SU(2)$ the induced map
\[
\gamma_\ast\co H_3(SU(2)^2/\Hom(\Z^2,SU(2));\Z)\xrightarrow{\cong} H_3(SU(2);\Z)
\]
is an isomorphism. Indeed, observe that
\[
SU(2)^2/\Hom(\Z^2,SU(2))\cong Y^+\,,
\]
where $Y^+$ is the one-point compactification of the space $Y:=SU(2)^2\backslash \Hom(\Z^2,SU(2))$ 
of non-commuting pairs in $SU(2)$. It is known that the commutator map restricted to $Y$ 
is a locally trivial fiber bundle over $SU(2)\backslash \{1\}$ with compact fiber 
$F:=\{(x,y)\in SU(2)^2\mid [x,y]=-1\}$, see \cite[VI 1 (a)]{AMcC}. In particular, as 
$SU(2)\backslash \{1\}$ is contractible (it is a sphere with one point removed), there 
is a homeomorphism $Y\cong (SU(2)\backslash\{1\})\times F$ under which the commutator 
map corresponds to the projection onto the first factor. Therefore, $\gamma$ is equivalent 
to the map $SU(2)\wedge F_+\to SU(2)$ induced by $F_+\to S^0$. Because this map has a section, 
and because $H_3(SU(2)\wedge F_+;\Z)\cong H_2(\Hom(\Z^2,SU(2));\Z)\cong  \Z$, the induced map 
$H_3(SU(2)\wedge F_+;\Z)\to H_3(SU(2);\Z)$ must be an isomorphism. Consequently, $\gamma_\ast$ 
is an isomorphism in the case $G=SU(2)$.

Finally, we contemplate the commutative diagram
\[
\xymatrix{
 H_3(SU(2)^2/\Hom(\Z^2,SU(2));\Z) \ar[r]^-{\gamma_\ast}_-{\cong} 
\ar[d]^-{(\overline{\nu\times \nu})_\ast} & H_3(SU(2);\Z) \ar[d]^-{\nu_\ast}_-{\cong} \\
H_3(G^2/\Hom(\Z^2,G);\Z) \ar[r]^-{\gamma_\ast} & H_3(G;\Z)
}
\]
By Theorem \ref{thm:mainpairs} all groups in the diagram are isomorphic to $\Z$, which forces 
both the left hand vertical arrow as well as the bottom horizontal arrow to be isomorphisms. 
This finishes the proof of the theorem. 
\end{proof}

\begin{remark}
Theorem \ref{thm:su2represents} provides an alternative proof of Theorem \ref{thm:spinstability}. Indeed, as all long roots are conjugate, Theorem \ref{thm:su2represents} holds for any homomorphism $SU(2)\hookrightarrow G$ corresponding to the inclusion of a long root. In $D_n$ all roots have the same length, and thus Theorem \ref{thm:spinstability} can be proved by applying Theorem \ref{thm:su2represents} to the first and the composite map in $\Hom(\Z^2,SU(2))\to \Hom(\Z^2,Spin(2\ell-2))\to \Hom(\Z^2,Spin(2\ell))$.
\end{remark}

In the study of spaces of homomorphisms $\Hom(\Gamma,G)$ it is natural to ask for the properties of the classifying space map
\[
B\co \Hom(\Gamma,G)\to \textnormal{map}_\ast(B\Gamma,BG)\,,
\]
which takes a homomorphism $\phi\co \Gamma\to G$ to the classifying map $B\phi\co B\Gamma\to BG$ of a flat principal $G$-bundle over $B\Gamma$ with holonomy $\phi$. On path--connected components $B$ describes the contribution of flat bundles to the set of isomorphism classes of principal $G$-bundles over $B\Gamma$; on higher homotopy groups there is a similar interpretation in terms of families of flat bundles (cf. \cite[Section 7]{R19}). For $\Gamma=\Z^2$ we have the following corollary of Theorem \ref{thm:su2represents}.

\begin{corollary} \label{cor:classifyingmap}
The classifying space map
\[
B\co \Hom(\Z^2,G)\to \textnormal{map}_\ast(\SS^1\times \SS^1,BG)
\]
induces an isomorphism on $\pi_k$ for all $k\leqslant 2$.
\end{corollary}

Here we identify $B\Z^2\simeq \SS^1\times \SS^1$ up to homotopy.

\begin{proof}
Since $BG$ is $3$-connected, the space $\textnormal{map}_\ast(\SS^1\times \SS^1,BG)$ is path--connected. By adjunction and the homotopy equivalence 
$\SS^k\wedge(\SS^1\times \SS^1)\simeq \SS^{k+2}\vee \bigvee^2 \SS^{k+1}$ (for $k>0$), we have natural isomorphisms
\[
\pi_k(\textnormal{map}_\ast(\SS^1\times \SS^1,BG))\cong 
[\SS^k\wedge (\SS^1\times \SS^1),BG]\cong \pi_{k+1}(G)\oplus \pi_{k}(G)\oplus \pi_k(G)
\]
for all $k>0$. In particular, $\textnormal{map}_\ast(\SS^1\times \SS^1,BG)$ is simply--connected. 
Since $\Hom(\Z^2,G)$ is also path--connected and simply--connected, it only remains to prove the case $k=2$ of the corollary. We shall prove this case by reducing 
first to $G=SU(2)$, and then further to the stable unitary group $U=\colim_{n\to \infty} U(n)$ for 
which the result is known.

By Theorem \ref{thm:su2represents}, by naturality of the classifying space map, and because 
$\nu\co SU(2)\hookrightarrow G$ induces an isomorphism $\pi_3(SU(2))\cong \pi_3(G)$, it suffices to show that
\[
\pi_2(\Hom(\Z^2,SU(2))) \to \pi_2(\textnormal{map}_\ast(\SS^1\times \SS^1,BSU(2)))\cong \pi_3(SU(2))
\]
is an isomorphism. The isomorphism $\pi_3(SU(2))\cong \pi_3(U)$ induced by the inclusion $SU(2)\to U$ 
is classical. On the other hand, we deduce from Theorem \ref{thm:casesymplecticandunitary} and 
Remark \ref{rem:symplecticandunitary} 
that $SU(2)\to U$ also induces an isomorphism $\pi_2(\Hom(\Z^2,SU(2)))\cong \pi_2(\Hom(\Z^2,U))$. 
It is then enough to show that the classifying space map induces an isomorphism
\[
\pi_2(\Hom(\Z^2,U)) \xrightarrow{\cong} \pi_2(\textnormal{map}_\ast(\SS^1\times \SS^1,BU))\, .
\]
This isomorphism is an application of \cite[Theorem 3.4]{R11}.
\end{proof}

Loosely speaking, the corollary says that every principal $G$-bundle over $\SS^2\times (\SS^1)^2$ arises from an $\SS^2$-family of flat bundles over $(\SS^1)^2$, and the associated family of holonomies $\SS^2\to \Hom(\Z^2,G)$ is uniquely determined up to homotopy.

Also observe that since $\pi_2(\Hom(\Z^2,G))\to \pi_2(\Rep(\Z^2,G))$ need not be an isomorphism, the map induced by $B$ on second homotopy groups does not factor through $\Rep(\Z^2,G)$ in general. This is in contrast to the fact that $\pi_0(\Hom(\Gamma,G))\to [B\Gamma,BG]$ factors through $\pi_0(\Rep(\Gamma,G))$ so long as $G$ is connected (cf. \cite[Lemma 2.5]{AC}).

\begin{remark}
Theorem \ref{thm:mainpairs} showed that the image of $\pi_2(\Hom(\Z^2,G))$ in $\pi_2(\Rep(\Z^2,G))$ 
has index $\textnormal{lcm}\{n_0^\vee,\dots,n_r^\vee\}$. With the ideas of the previous theorem we 
can give an alternative explanation of this fact. Without reference to Theorem \ref{thm:mainpairs}, 
the proof of Theorem \ref{thm:su2represents} and the fact that $\pi_2(\Hom(\Z^2,G))$ has rank one 
as an abelian group (Corollary \ref{cor:rank}) show that the map $\nu\co SU(2)\to G$ induces an isomorphism
\[
\pi_2(\Hom(\Z^2,SU(2))) \cong \pi_2(\Hom(\Z^2,G))/\textnormal{torsion}\, .
\]
By Theorem \ref{thm:generalsecondhomology}, $\pi_\ast\co \pi_2(\Hom(\Z^2,SU(2)))\to \pi_2(\Rep(\Z^2,SU(2)))$ 
is an isomorphism, hence the degree of 
$\pi_\ast\co \pi_2(\Hom(\Z^2,G))/\textnormal{torsion}\to \pi_2(\Rep(\Z^2,G))$ equals the degree of 
the map $\pi_2(\Rep(\Z^2,SU(2)))\to \pi_2(\Rep(\Z^2,G))$ induced by $\nu$. To calculate the degree 
one identifies the map $\Rep(\Z^2,SU(2))\to \Rep(\Z^2,G)$ up to equivalence with a map 
$\mathbb{CP}(1,1)\to \mathbb{CP}(\mathbf{n}^\vee)$ through a case-by-case argument with root systems 
similar to the proof of Theorem \ref{thm:spinstability}. One finds that, for each $G$, there exists 
$j\in \{1,\dots,r\}$ such that $\mathbb{CP}(1,1)\to \mathbb{CP}(\mathbf{n}^\vee)$ is homotopic to
\[
\mathbb{CP}(1,1)\xrightarrow{i_j} \mathbb{CP}(1,\dots,1)\xrightarrow{p_{\mathbf{n}^\vee}} 
\mathbb{CP}(\mathbf{n}^\vee)\,,
\]
where $i_j$ is the inclusion $[z_0,z_1] \mapsto [z_0,0,\dots,0,z_1,0,\dots,0]$ with $z_1$ in 
the $j$-th position, and the second map is the projection 
$[z_0,\dots,z_r]\mapsto [z_0^{n_0^\vee},\dots,z_r^{n_r^\vee}]$. By the formula displayed 
in (\ref{eq:degreeofmapofweightedprojectivespaces}), the degree of the composite map on second 
homology groups equals $\textnormal{lcm}\{n_0^\vee,\dots,n_r^\vee\}$, independent of $j$.
\end{remark}

\begin{remark}
Consider a representation $\rho\co G\to SU(N)$ and let $D_{\rho}$ be the Dynkin index of $\rho$, that is the degree of $\rho_\ast\co \pi_3(G)\to \pi_3(SU(N))$. Theorem \ref{thm:su2represents} (or Corollary \ref{cor:classifyingmap}) shows that the degree of $\pi_2(\Hom(\Z^2,G))\to \pi_2(\Hom(\Z^2,SU(N)))$ equals $D_\rho$. On the other hand, as the degree of $\pi_2(\Hom(\Z^2,SU(N)))\to \pi_2(\Rep(\Z^2,SU(N)))$ is one (Theorem \ref{thm:mainpairs}), the degree of $\pi_2(\Rep(\Z^2,G))\to \pi_2(\Rep(\Z^2,SU(N)))$ equals $D_\rho/D$, where $D=\textnormal{gcd}\{D_\rho\mid \rho\co G\to SU(N),\, N\geq 1\}$ is the Dynkin index of $G$, which equals $\textnormal{lcm}\{n_0^\vee,\dots,n_r^\vee\}$ as noted in the introduction.
\end{remark}

\section{Application: TC structures on trivial principal $G$-bundles}\label{TC structures}

In this section we provide a geometric interpretation for the calculations given in this 
article. In particular, we show how examples of non-trivial transitionally commutative (TC) 
structures on trivial principal 
$G$-bundles can be constructed using non-trivial elements in $\pi_{2}(\Hom(\Z^{2},G))$.

\subsection{TC structures on principal $G$-bundles}

We start by reviewing the concept of a TC structure on a principal $G$-bundle. Assume that $G$ 
is a Lie group. We can associate to $G$ a simplicial space, denoted $[B_{com}G]_{*}$ and 
defined by $[B_{com}G]_n:=\Hom(\Z^{n},G)\subset G^n$. 
The face and degeneracy maps in $[B_{com}G]_{*}$ are defined as the restriction of the 
face and degeneracy maps in the classical bar construction $[BG]_{*}$. 
The geometric realization of the simplicial space $[B_{com}G]_*$ 
is denoted by $B_{com}G$ and called the classifying space for commutativity 
in $G$ \cite{ACT,AG}. The space $B_{com}G$ is naturally a subspace of $BG$ and we denote 
by $p_{com}\co E_{com}G\to B_{com}G$ the restriction of the universal 
principal $G$-bundle $p\co EG\to BG$ over $B_{com}G$. The space $B_{com}G$ classifies principal 
$G$-bundles that come equipped with an additional structure that we will refer to as a 
transitionally commutative structure. To explain this further we need to recall some basic 
definitions from bundle theory.

Suppose that $q\co E\to X$ is a principal $G$-bundle with $G$ acting on the right on $E$ and 
that $X$ is a CW-complex.  By local triviality we can find an open cover $\U=\{U_{i}\}_{i\in I}$ 
of $X$ together with trivializations $\varphi_{i}\co q^{-1}(U_{i})\to U_{i}\times G$ for 
every $i\in I$. If $U_{i}$ and $U_{j}$ are such that $U_{i} \cap U_{j}\ne \emptyset$, then for 
every  $x\in U_{i}\cap U_{j}$ and all $g\in G$ we have  
$\varphi_{i} \varphi_{j}^{-1}(x,g)=(x,\rho_{i,j}(x)g)$. Here 
$\rho_{i,j}:U_{i}\cap U_{j}\to G$ is a continuous function called the transition function. 
The different transition functions satisfy the cocycle identity
\[
\rho_{i,k}(x)=\rho_{i,j}(x) \rho_{j,k}(x)
\] 
for every $x\in U_{i}\cap U_{j}\cap U_{k}$.  Now, if for every $i,j,k\in I$ and every 
$x\in U_{i}\cap U_{j}\cap U_{k}$ the elements $\rho_{i,j}(x), \rho_{j,k}(x)$ and 
$\rho_{i,k}(x)$ commute with each other, then we say that $\{\rho_{i,j}\}$ is a 
commutative cocycle. 
 
Following \cite{AG}, we call a principal $G$-bundle $q\co E\to X$ transitionally commutative if it 
admits a trivialization in such a way that the corresponding transition functions define a 
commutative cocycle. Let $f\co X\to BG$ be the classifying map of a 
principal $G$--bundle  $q\co E\to X$  over a finite CW--complex $X$. Then by \cite[Theorem 2.2]{AG} 
it follows that $f$ factors through $\BC G$, up to homotopy, if and only 
if $q$ is transitionally commutative.  

Next we define the notion of a transitionally commutative structure on principal bundles as in 
\cite{SimonPhD} and \cite{RV}.  

\begin{definition}
Suppose that $q:E\to X$ is a transitionally commutative principal $G$-bundle. A 
\textsl{transitionally commutative structure (TC structure)} on $q:E\to X$ is a 
choice of map $\tilde{f}\co X\to B_{com}G$, up to homotopy, such that 
$i \tilde{f}\co X\to BG$ is a classifying map for $q\co E\to X$. Here $i\co B_{com}G\to BG$ 
denotes the natural inclusion. 
\end{definition}

With the previous definition the set of TC structures on principal $G$-bundles over a space 
$X$ is precisely the set $[X,B_{com}G]$ of homotopy classes of maps $X\to B_{com}G$. 

\begin{example}
Consider the trivial principal $G$-bundle $\textnormal{pr}_{1}\co X\times G\to X$. This bundle admits 
a TC structure given by the homotopy class of a constant map $X\to B_{com}G$. 
We refer to this TC structure as the trivial TC structure.
\end{example}

Observe that the definition of TC structures implies that the same underlying principal 
$G$-bundle can admit TC structures in many different ways.

\subsection{Examples of TC structures on the trivial bundle}

Our next goal is to construct non-trivial TC structures for the trivial principal $G$-bundle 
over $\SS^{4}$.

Let $E_{com}G$ denote the homotopy fiber of the inclusion
$i\co B_{com}G\to BG$. As a direct consequence of \cite[Theorem 6.3]{ACT}, the homotopy fiber 
sequence $E_{com}G\to B_{com}G\to BG$ induces for every $n\geqslant 0$ a split short exact sequence 
\begin{equation} \label{eq:sespibcom}
1\to \pi_{n}(E_{com}G)\to \pi_{n}(B_{com}G) \xrightarrow{i_{*}}\pi_{n}(BG)\to 1\,.
\end{equation}
If $G$ is connected, then both $B_{com}G$ and $BG$ are simply--connected. 
In particular, for every $n\geqslant 0$ 
we have that $[\SS^{n},B_{com}G]\cong \pi_{n}(B_{com}G)$ so that $\pi_{n}(B_{com}G)$ 
agrees with the set of TC structures on principal $G$-bundles over $\SS^{n}$. Let $f\co \SS^{n}\to B_{com}G$ 
be a TC structure on the trivial principal $G$-bundle over $\SS^{n}$. 
Then $[f]\in \pi_{n}(B_{com}G) $ belongs 
to 
\[
\Ker(i_{*}\co \pi_{n}(B_{com}G)\to \pi_{n}(BG))\cong \pi_{n}(E_{com}G)\,.
\]
In fact, elements in $\pi_{n}(E_{com}G)$ correspond precisely to 
TC structures on the trivial principal $G$-bundle over $\SS^{n}$.

If $G=SU(m)$ or $G=Sp(k)$, then $E_{com}G$ is $3$-connected by \cite[Proposition 3.2]{AG}. 
This implies that the lowest dimensional sphere for which a non-trivial TC structure on the 
trivial principal $G$-bundle can exist is $\SS^{4}$. For other simply--connected compact Lie groups 
$G$ this may not hold, as a result of the fact that $\Hom(\Z^n,G)$ can be disconnected. However, 
the following variation of $B_{com}G$ was considered in \cite{AG}. Let $[B_{com}G_{\BONE}]_\ast$ 
denote the sub-simplicial space of $[B_{com}G]_\ast$ defined by 
$[B_{com}G_{\BONE}]_n:=\Hom(\Z^n,G)_{\BONE}$. Let  $E_{com}G_{\BONE}$ denote the homotopy fiber 
of the inclusion $B_{com}G_{\BONE}\to BG$. The proof of \cite[Theorem 6.3]{ACT} shows that there 
is a split exact sequence of the form (\ref{eq:sespibcom}) also when $E_{com}G$ and $B_{com}G$ 
are replaced by $E_{com}G_{\BONE}$ and $B_{com}G_{\BONE}$, respectively. Therefore, 
\cite[Proposition 3.2]{AG} implies that $E_{com}G_{\BONE}$ is $3$-connected if $G$ is simply--connected. 
Moreover, $\pi_4(B_{com}G_{\BONE})\cong \pi_4(E_{com}G_{\BONE})\oplus \Z$ if in addition $G$ is 
assumed simple.

\begin{corollary} \label{cor:pi4}
Let $G$ be a simply--connected simple compact Lie group. Then
\[
\pi_4(E_{com}G_{\BONE})\cong \Z\, \textnormal{ and }\, \pi_4(B_{com}G_{\BONE})\cong \Z\oplus \Z\, .
\]
\end{corollary}
\begin{proof}
As explained in \cite{AG} a model for $E_{com}G_{\BONE}$ is the geometric realization of a 
simplicial space $[E_{com}G_{\BONE}]_\ast$ with $n$-simplices 
$[E_{com}G_{\BONE}]_n:=\Hom(\Z^n,G)_{\BONE}\times G$. To keep our proof short we refer to \cite{AG} 
for more details about the simplicial structure. The filtration of $E_{com}G_{\BONE}$ arising from 
the simplicial structure leads to a spectral sequence (see \cite[Theorem 11.14]{MayGeometry}) 
which takes the form
\[
E^2_{p,q}=H_p(H_q([E_{com}G_{\BONE}]_\ast;\Z)) \Longrightarrow H_{p+q}(E_{com}G_{\BONE};\Z)\, .
\]
The proof of \cite[Proposition 3.3]{AG} shows that $E^2_{p,q}=0$ for all $p,q\geqslant 0$ with 
$0< p+q\leqslant 3$. Let us determine $E^2_{p,q}$ for $p+q=4$. Since $\Hom(\Z^n,G)_{\BONE}\times G$ 
is path--connected and simply--connected for all $n\geqslant 0$, we have that $E_{p,0}^2=0$ for all 
$p\geqslant 1$ and $E_{p,1}^2=0$ for all $p\geqslant 0$.

The chain complex computing $E^2_{1,3}$ reads
\[
\cdots \to H_3(\Hom(\Z^2,G)\times G;\Z)\xrightarrow{d_2} H_3(G\times G;\Z) \xrightarrow{d_1} H_3(G)\to 0\,,
\]
with differentials $d_k=\sum_{i=0}^{k}(-1)^i\partial_i$ where 
$\partial_i\co [E_{com}G]_{k}\to [E_{com}G]_{k-1}$ is the $i$-th face map. Let $i_2\co G\to G\times G$ 
be the inclusion $x\mapsto (1,x)$. Then a short calculation shows that $\ker(d_1)= \Im((i_2)_\ast)$. 
Moreover, if $i_2'\co G\to \Hom(\Z^2,G) \times G$ denotes the inclusion $x\mapsto (1,1,x)$, 
then $d_2(i_2')_\ast=(i_2)_\ast$. Hence, $E_{1,3}^2=\ker(d_1)/\Im(d_2) =0$.

The chain complex computing $E_{2,2}^2$ reads
\[
\cdots \to H_2(\Hom(\Z^3,G)_{\BONE}\times G;\Z)\xrightarrow{d_3} H_2(\Hom(\Z^2,G)\times G;\Z) \xrightarrow{d_2} 0\, ,
\]
because $G$ is assumed simply--connected and therefore $H_2(G\times G;\Z)=0$ by the K{\"u}nneth theorem. 
Now $H_2(\Hom(\Z^2,G)\times G;\Z)\cong H_2(\Hom(\Z^2,G);\Z)\cong \Z$, by Theorem \ref{thm:mainpairs}. 
As a consequence, $d_3$ factors through $H_2(\Hom(\Z^3,G)_{\BONE};\Z)/\textnormal{torsion}$.

We claim that $d_3=0$, hence $E^2_{2,2}\cong \Z$. To see this let 
$j\co \bigvee^3 \Hom(\Z^2,G)\to \Hom(\Z^3,G)_{\BONE}$ be the map induced by $(x,y)\mapsto (x,y,1)$, 
$(x,y)\mapsto (x,1,y)$, and $(x,y)\mapsto (1,x,y)$. By Corollary \ref{cor:h2modtorsion}, 
the map $j_\ast$ induced by $j$ on second homology groups is an isomorphism modulo torsion. Therefore, 
to see that $d_3=0$ it is enough to show that $d_3j_\ast=0$, but this is an easy calculation.

Finally, it is clear that the differential $d_1\co H_4(G\times G;\Z)\to H_4(G;\Z)$ in the chain complex 
computing $E_{0,4}^2$ is surjective, so $E_{0,4}^2=0$.

It follows that the only non-trivial group in total degree $4$ of the $E^2$-page is $E^2_{2,2}$ and 
there are no non-trivial differentials originating from or arriving at $E^2_{2,2}$. Hence, we conclude 
that $H_4(E_{com}G_{\BONE};\Z)\cong E^{2}_{2,2}\cong \Z$. The isomorphism $\pi_4(E_{com}G_{\BONE})\cong \Z$ 
is obtained from the Hurewicz theorem and the fact that $E_{com}G_{\BONE}$ is $3$-connected.
\end{proof}

We show next how an element of $\pi_{2}(\Hom(\Z^{2},G))\cong \Z$ 
can be used to construct a TC structure on the trivial principal 
$G$-bundle over $\SS^{4}$.

Let $\beta:\SS^{2}\to \Hom(\Z^{2},G)$ be any map. If $\beta_{1}$ and $\beta_{2}$ are the components of $\beta$, 
then we have that $\beta_{1}(x),\beta_{2}(x)\in G$ commute for all $x\in \SS^{2}$. We are going to use 
the functions $\beta_{1}$ and $\beta_{2}$ to construct a commutative cocycle with values in $G$. 
The construction we give follows the idea of \cite[Section 3]{RV} where the case $G=O(2)$ was studied.
Consider 
\[
\SS^{4}=\{(x_{0},x_{1},x_{2},x_{3},x_{4})\in \R^{5}~|~ 
x_{0}^{2}+x_{1}^{2}+x_{2}^{2}+x_{3}^{2}+x_{4}^{2}=1\}.
\]
We can cover $\SS^{4}$ using the closed sets $C_{1}$, $C_{2}$ and $C_{3}$ 
given by
\begin{align*}
C_{1}&=\{(x_{0},x_{1},x_{2},x_{3},x_{4})\in \SS^{4} \mid x_{0}\leqslant 0\}\,,\\
C_{2}&=\{(x_{0},x_{1},x_{2},x_{3},x_{4})\in \SS^{4} \mid x_{0}\geqslant 0,\, x_{4}\geqslant 0\}\,,\\
C_{3}&=\{(x_{0},x_{1},x_{2},x_{3},x_{4})\in \SS^{4} \mid x_{0}\geqslant 0,\, x_{4}\leqslant 0\}\,.
\end{align*}
Notice that 
\[
C_{1}\cap C_{2}\cap C_{3}=\{(x_{0},x_{1},x_{2},x_{3},x_{4})\in \SS^{4} \mid x_{0}= 0,\, x_{4}= 0\}
\cong \SS^{2}.
\]
From now on we identify $\SS^{2}$ with $C_{1}\cap C_{2}\cap C_{3}$. In addition, observe that 
$C_{1}\cap C_{2}\cong \D^{3}$ and under this identification the boundary $\SS^{2}$ 
corresponds to $C_{1}\cap C_{2}\cap C_{3}$. The same is true for  $C_{1}\cap C_{3}$
and $C_{2}\cap C_{3}$.

Recall that $\pi_{2}(G)= 0$, hence $\beta_1\co \SS^2\to G$ is null-homotopic. Since 
$C_{1}\cap C_{2}\cong \D^{3}$, we can find a continuous map $\rho_{1,2}\co C_{1}\cap C_{2}\to G$ 
such that  $\left.\rho_{1,2}\right|_{C_{1}\cap C_{2}\cap C_{3}}\co C_{1}\cap C_{2}\cap C_{3}\to G$ 
agrees with $\beta_{1}$. Similarly, the choice of a null-homotopy of $\beta_2$ defines a continuous map 
$\rho_{2,3}\co C_{2}\cap C_{3}\to G$ such that 
$\left.\rho_{2,3}\right|_{C_{1}\cap C_{2}\cap C_{3}}\co C_{1}\cap C_{2}\cap C_{3}\to G$ agrees with $\beta_2$. 
To define the transition function $\rho_{1,3}\co C_1\cap C_3\to G$ we consider the retraction 
$r\co C_{2}\to C_{2}\cap C_{3}$ defined by
\[
r(x_{0},x_{1},x_{2},x_{3},x_{4}) := 
\left(\sqrt{1-x_{1}^{2}-x_{2}^{2}-x_{3}^{2}},x_1,x_{2},x_{3},0 \right).
\]
Then define $\rho_{1,3}\co C_{1}\cap C_{3}\to G$ by
\[
\rho_{1,3}(0,x_{1},x_{2},x_{3},x_{4})
:=\rho_{1,2}(0,x_{1},x_{2},x_{3},-x_{4})\rho_{2,3}(r(0,x_{1},x_{2},x_{3},-x_{4})).
\]  
For $i>j$ we set $\rho_{i,j}:=\rho_{j,i}^{-1}$. Notice that if $x\in C_{1}\cap C_{2}\cap C_{3}$, 
then $x=(0,x_{1},x_{2},x_{3},0)$ and 
$r(x)=x$. Therefore, for $x\in C_{1}\cap C_{2}\cap C_{3}$ we have that 
$\rho_{1,3}(x)=\rho_{1,2}(x)\rho_{2,3}(x)$ by definition of $\rho_{1,3}$. Thus, $\{\rho_{i,j}\}$ 
satisfies the cocycle condition. Moreover, since $\rho_{1,2}(x)=\beta_{1}(x)$ and  
$\rho_{2,3}(x)=\beta_{2}(x)$ for all $x\in C_{1}\cap C_{2}\cap C_{3}$, we conclude that 
$\{\rho_{i,j}\}$ defines a commutative cocycle relative to the closed cover 
$\mathcal{C}:=\{C_{1},C_{2},C_{3}\}$ of $\SS^4$. 

Let $E_\beta$ be the space defined by 
\[
E_\beta:=(C_{1}\times G\sqcup C_{2}\times G\sqcup C_{3}\times G)/{\sim}\,,
\]
where $(j,x,g)\sim (i,x,\rho_{ij}(x)g)$. The projection map 
$E_\beta\to \SS^{4}$ induced by $(j,x,g)\mapsto x$ defines a principal $G$-bundle by \cite[Lemma 3.1]{RV}.

\begin{lemma} \label{lem:trivialbundle}
The principal $G$-bundle $E_\beta \to \SS^4$ is trivial.
\end{lemma}
\begin{proof}
By \cite[Lemma 3.2]{RV}, the principal $G$-bundle $E_\beta$ is isomorphic to the bundle obtained 
using the clutching function 
$\varphi\co C_{1}\cap (C_{2}\cup C_{3})\cong \SS^{3}\to G$ given by 
\begin{equation*}
\varphi(x)=\begin{cases}
\rho_{1,2}(x)\rho_{2,3}(r(x))\,, &\text{if } x\in C_{1}\cap C_{2}\,,\\
\rho_{1,3}(x)\,, &\text{if } x\in C_{1}\cap C_{3}\,.
\end{cases}
\end{equation*}					
By construction this function satisfies 
$\varphi(0,x_{1},x_{2},x_{3},x_{4})=\varphi(0,x_{1},x_{2},x_{3},-x_{4})$. 
This implies that $\varphi$ factors through $C_1\cap C_2\cong \D^{3}$, hence $\varphi$ is null-homotopic. 
Therefore, the principal $G$-bundle $E_\beta$ is trivial.
\end{proof}

Let $N_\ast(\mathcal{C})$ denote the {\v C}ech nerve of the closed cover $\mathcal{C}$ of $\SS^4$. 
The commutative cocycle $\{\rho_{i,j}\}$ defines a simplicial map 
$[f_\beta]_\ast \co N_{*}(\mathcal{C})\to [B_{com}G_{\BONE}]_{*}$ sending $x\in C_{i_0}\cap \dots \cap C_{i_n}$ 
to $(\rho_{i_0,i_1}(x),\dots,\rho_{i_{n-1},i_n}(x))\in \Hom(\Z^n,G)_{\BONE}$. Upon geometric realization 
we obtain a map  
\[
f_\beta:=|[f_\beta]_\ast|\co |N_\ast(\mathcal{C})|\to B_{com}G_{\BONE}\, .
\] 
We can choose open neighbourhoods $U_i\supset C_i$, $i=1,2,3$, such that for every choice of indices the 
inclusion $C_{i_0}\cap \dots \cap C_{i_n}\hookrightarrow U_{i_0}\cap \dots \cap U_{i_n}$ is a homotopy 
equivalence. Let $\mathcal{U}=\{U_1,U_2,U_3\}$ be the resulting open cover of $\SS^4$. The map of 
{\v C}ech nerves $N_\ast(\mathcal{C})\to N_\ast(\mathcal{U})$ is a levelwise homotopy equivalence of 
proper simplicial spaces, hence induces a homotopy equivalence 
$|N_\ast(\mathcal{C})|\simeq |N_\ast(\mathcal{U})|$. Since $\mathcal{U}$ is numerable, the natural map 
$|N_{*}(\mathcal{U})|\to \SS^4$ is a homotopy equivalence by \cite[Proposition 4.1]{Segal1}. We 
conclude that $|N_\ast(\mathcal{C})|\to \SS^4$ is a homotopy equivalence. Let 
$\lambda\co \SS^4\to |N_\ast(\mathcal{C})|$ be a homotopy inverse. The argument of \cite[Lemma 3.3]{RV} 
shows that $f_\beta \lambda\co \SS^4\to B_{com}G_{\BONE}$ is a TC structure on $E_\beta$, i.e., that 
$if_\beta \lambda\co \SS^4\to BG$ classifies $E_\beta$. Since $E_\beta$ is trivial, by 
Lemma \ref{lem:trivialbundle}, $if_\beta \lambda$ is null homotopic. Thus, $f_\beta\lambda$ factors, 
up to homotopy, through the homotopy fiber $E_{com}G_{\BONE}$. \medskip

Let us construct another commutative cocycle on $\SS^4$, this time representing a generator of 
$\pi_4(BG)\cong \Z$. To this end, we cover $\SS^4$ by the contractible closed sets
\begin{align*}
D_{1}&=\{(x_{0},x_{1},x_{2},x_{3},x_{4})\in \SS^{4} \mid x_{0}\leqslant 0\}\,,\\
D_{2}&=\{(x_{0},x_{1},x_{2},x_{3},x_{4})\in \SS^{4} \mid x_{0}\geqslant 0\}\,.
\end{align*}
Observe that $D_1\cap D_2\cong \SS^3$. Let $\tau_{1,2}\co D_{1}\cap D_{2} \to G$ represent a generator 
of $\pi_3(G)$. To be concrete, we choose $\tau_{1,2}=\nu$, where $\nu\co SU(2)\to G$ is the map defined in 
Section \ref{sec:rolesu2}. Then $\tau_{1,2}$ defines trivially a commutative cocycle relative to the 
closed cover $\mathcal{D}:=\{D_1,D_2\}$. As before, the cocycle $\{\tau_{i,j}\}$ defines a simplicial 
map $[g_\nu]_\ast \co N_\ast(\mathcal{D})\to [B_{com}G_{\BONE}]_\ast$ and thus a map 
$g_\nu \co |N_\ast(\mathcal{D})|\to B_{com}G_{\BONE}$. Upon choosing a homotopy equivalence
 $\lambda'\co \SS^4\to |N_\ast(\mathcal{D})|$ one obtains a classifying map $ig_\nu\lambda'\co \SS^4\to BG$ 
for the bundle clutched by $\nu\co \SS^3\to G$. Since $\nu$ represents a generator of $\pi_3(G)$, 
the homotopy class of $ig_\nu \lambda'$ generates $\pi_4(BG)$.\medskip

From now on we tacitly identify $|N_\ast(\mathcal{C})|$ and $|N_\ast(\mathcal{D})|$ with $\SS^4$ and 
drop the homotopy equivalences $\lambda$ and $\lambda'$ from the notation.

\begin{theorem}
Let $[\beta]\in \pi_2(\Hom(\Z^2,G))$ and $[\nu]\in \pi_3(G)$ be generators. Then the TC structures 
$[f_\beta]$ and $[g_\nu]$ generate $\pi_4(B_{com}G_{\BONE})\cong \Z\oplus \Z$. In particular, 
$f_{\beta}$ lifts to a generator of $\pi_4(E_{com}G_{\BONE})\cong \Z$.
\end{theorem}
\begin{proof}
The proof is by comparison of the spectral sequences associated to the simplicial spaces 
$N_\ast(\mathcal{C})$, $N_\ast(\mathcal{D})$ and $[B_{com}G_{\BONE}]_\ast$. Let us first consider the 
spectral sequence
\[
{}_{\mathcal{C}}E^2_{p,q}=H_p(H_q(N_\ast(\mathcal{C});\Z)) \;\Longrightarrow\; H_{p+q}(|N_\ast(\mathcal{C})|;\Z)\,. 
\]
In each degree $N_\ast(\mathcal{C})$ is a disjoint union of contractible spaces and spaces 
homeomorphic to $\SS^2$. This readily shows that ${}_{\mathcal{C}}E^2_{p,q}=0$ unless $q=0,2$. 
In the case $q=0$ the chain complex $H_0(N_\ast(\mathcal{C});\Z)$ computes the homology of a $2$-simplex, 
hence ${}_{\mathcal{C}}E^2_{4,0}=0$. In the case $q=2$ we observe that 
$H_2(N_2(\mathcal{C});\Z)\cong H_2(C_1\cap C_2\cap C_3;\Z)\cong \Z$, hence ${}_{\mathcal{C}}E^2_{2,2}$ 
is a quotient of $\Z$. In fact, we must have ${}_{\mathcal{C}}E^2_{2,2}\cong \Z$, because 
$H_4(|N_\ast(\mathcal{C})|;\Z)\cong \Z$. For the same reason, ${}_{\mathcal{C}}E^2_{2,2}$ is not hit 
by any non-zero differential. Since ${}_{\mathcal{C}}E^2_{0,3}=0$, there is no non-zero differential 
leaving ${}_{\mathcal{C}}E^2_{2,2}$ either. We conclude that 
${}_{\mathcal{C}}E^\infty_{2,2}\cong H_2(C_1\cap C_2\cap C_3;\Z)\cong \Z$ is the only non-zero group in 
total degree $4$.

The analysis of the spectral sequence $\{{}_{\mathcal{D}}E^\ast_{p,q}\}$ calculating 
$H_\ast(|N_\ast(\mathcal{D})|;\Z)$ is very similar; one finds that 
${}_{\mathcal{D}}E_{1,3}^\infty\cong H_3(D_1\cap D_2;\Z)\cong \Z$ is the only non-zero group in total degree $4$.

Finally, we consider the spectral sequence
\[
E^2_{p,q}=H_p(H_q([B_{com}G_{\BONE}];\Z))\; \Longrightarrow \; H_{p+q}(B_{com}G_{\BONE};\Z)\, .
\]
Since $[B_{com}G_{\BONE}]_0$ is a one point space, we trivially have that $E^2_{0,q}=0$ for all $q>0$. 
Because $\Hom(\Z^n,G)_{\BONE}$ is path--connected and simply--connected for all $n\geqslant 0$, we further 
have that $E^2_{4,0}=E^2_{3,1}=0$. On the other hand, we find that $E_{1,3}^2$ is a quotient of 
$H_{3}(G;\Z)\cong \Z$, and likewise $E^2_{2,2}$ is a quotient of $H_2(\Hom(\Z^2,G);\Z)\cong \Z$ 
(Theorem \ref{thm:mainpairs}). In Corollary \ref{cor:pi4} we showed that 
$H_4(B_{com}G_{\BONE};\Z)\cong \Z\oplus \Z$; but this is possible only if $E^2_{1,3}\cong \Z$ and 
$E_{2,2}^2\cong \Z$ and none of these is hit by a non-zero differential. Furthermore, for degree 
reasons and because $E_{0,3}^2=0$, there are no non-zero differentials originating from either 
$E^2_{1,3}$ or $E^2_{2,2}$. Hence, $E^\infty_{1,3}\cong H_3(G;\Z)\cong \Z$ and 
$E^\infty_{2,2}\cong H_2(\Hom(\Z^2,G);\Z)\cong \Z$.

The simplicial maps $[f_\beta]_\ast\co N_\ast(\mathcal{C})\to [B_{com}G_{\BONE}]_\ast$ and 
$[g_\nu]_\ast\co N_\ast(\mathcal{D})\to [B_{com}G_{\BONE}]_\ast$ give rise to the following diagram 
of extensions:

\[
\xymatrix{
0 \ar[r] & 0 \ar[r] \ar[d] & H_{4}(|N_\ast(\mathcal{C})|;\Z) \ar[r] \ar[d]^-{(f_\beta)_\ast} & 
{}_{\mathcal{C}}E_{2,2}^{\infty}\cong \Z \ar[r] \ar[d]^-{\cong}_-{\beta_\ast} & 0 \\
0 \ar[r] & E_{1,3}^{\infty}\cong \Z  \ar[r] &  H_{4}(B_{com}G_{\BONE};\Z) \ar[r] & 
E_{2,2}^{\infty}\cong \Z \ar[r] & 0 \\
0 \ar[r] & {}_{\mathcal{D}}E_{1,3}^{\infty}\cong \Z \ar[r] \ar[u]_-{\cong}^-{\nu_\ast} & 
H_{4}(|N_\ast(\mathcal{D})|;\Z) \ar[r] \ar[u]_-{(g_\nu)_\ast} & 0 \ar[u] \ar[r] & 0
}
\]
The preceding discussion shows that the map ${}_{\mathcal{D}}E_{1,3}^\infty\to E^\infty_{1,3}$ 
can be identified with the map $(\nu)_\ast\co H_3(\SS^3;\Z)\to H_3(G;\Z)$, and 
${}_{\mathcal{C}}E^\infty_{2,2}\to E^\infty_{2,2}$ may be identified with 
$\beta_\ast\co H_2(\SS^2;\Z)\to H_2(\Hom(\Z^2,G);\Z)$; both are isomorphisms by choice. The 
commutative diagram together with the Hurewicz theorem imply that $[f_\beta]$ and $[g_\nu]$ 
generate $\pi_4(B_{com}G_{\BONE})$. Since $if_\beta$ is null homotopic, 
$[f_\beta]\in \Ker(i_\ast\co \pi_4(B_{com}G_{\BONE})\to \pi_4(BG))\cong \pi_4(E_{com}G_{\BONE})$ 
and it is clear that $[f_{\beta}]$ generates $\pi_4(E_{com}G_{\BONE})$.
\end{proof}

\subsection{An explicit generator of $\pi_2(\textnormal{Hom}(\Z^2,G))$} We finish our discussion by constructing a map
$\beta \co \SS^2\to \Hom(\Z^2,G)$ whose homotopy class generates the group 
$\pi_2(\Hom(\Z^2,G))\cong \Z$. In theory, this enables us to write down an explicit commutative 
cocycle on $\SS^4$ (relative to the closed cover $\{C_1, C_2,C_3\}$ described above)  which 
represents the generator of $\pi_4(E_{com}G_{\BONE})\cong \Z$.

By Theorem \ref{thm:su2represents} it suffices to construct $\beta$ in the case $G=SU(2)$; 
composition with the embedding $\nu\co SU(2)\to G$ yields a generator for general $G$. Let 
$T\leqslant SU(2)$ be the maximal torus 
consisting of all diagonal matrices and identify it with the unit circle $\SS^1\subseteq \C$. 
Under this identification the action of the Weyl group $W\cong \mathbb{Z}/2$ corresponds 
to complex conjugation on $\mathbb{S}^1$. Let $I=[0,1]$ denote the unit interval and let 
$\Delta=\{(s,t)\in I^2 \mid  s\leqslant t\}$. Our model for $\mathbb{S}^2$ will be the 
boundary of the prism $P=\Delta\times I$. Consider the continuous map
\begin{alignat*}{2}
r	& \co \Delta	&& \to (\mathbb{S}^1)^2 \subset \Hom(\Z^2,SU(2)) \\
	& (s,t)		&& \mapsto  (e^{2\pi i s}, e^{2\pi i t})\, ,
\end{alignat*}
whose image is the closure of a fundamental domain for the diagonal $\mathbb{Z}/2$--action on 
$(\mathbb{S}^1)^2$. The basic idea to construct the desired homotopy class is to choose a 
null homotopy of
\[
r|_{\partial \Delta}\co \partial \Delta \to \Hom(\Z^2,SU(2))\, ,
\]
which exists because $\Hom(\Z^2,SU(2))$ is simply--connected. Up to homotopy, this induces a map 
$\Delta /\partial \Delta\cong \mathbb{S}^2 \to \Hom(\Z^2,SU(2))$. However, some care must be 
taken as the choice of null homotopy will in general affect the resulting homotopy class.

Let $\rho\co (I,\partial I)\to (\mathbb{S}^1,1)$ represent a generator of $\pi_1(\mathbb{S}^1)$ 
and let $h \co (I,\partial I)\times I\to (SU(2),1)$ be \emph{any} fixed null homotopy of 
$i \rho$, where $i\co (\mathbb{S}^1,1)\hookrightarrow (SU(2),1)$ is the inclusion. 
Let $i_{1} \co SU(2)\to \Hom(\Z^2,SU(2))$ denote the inclusion of the first factor. 
Similarly, let $i_2$ denote the inclusion of the second factor. 
Let $d\co SU(2)\to \Hom(\Z^2,SU(2))$ be the diagonal map. 
We now extend $r$ to a continuous map
\[
\beta \co \partial P\to \Hom(\Z^2,SU(2))
\]
as follows: Define $\beta|_{\Delta\times \{0\}}:=r$ and let $\beta|_{\Delta\times \{1\}}$ be the 
constant map with value $(1,1)$. The boundary $\partial \Delta$ is the union of the three 
intervals
\[
\Delta_{01}=\{s=0\} \,, \quad \Delta_{02}=\{s=t\}\,, \quad \Delta_{12}=\{t=0\}\,,
\]
each of which is identified with $I$. Define 
$\beta|_{\Delta_{01}\times I}:=i_2 h$, $\beta|_{\Delta_{02}\times I}:=d h$, and 
$\beta|_{\Delta_{12}\times I}:=i_1 h$. By inspection, these maps can be glued together and 
define a continuous map $\beta\co \partial P \to \Hom(\Z^2,SU(2))$.

\begin{proposition}\label{prop:geometricgenerator}
The homotopy class of $\beta$ generates $\pi_2(\Hom(\Z^2,SU(2)))$.
\end{proposition}
\begin{proof}
By Theorem \ref{thm:mainpairs} it suffices to show that the composite map
\[
\partial P\xrightarrow{\beta} \Hom(\Z^2,SU(2)) \xrightarrow{\pi} \Rep(\Z^2,SU(2))
\]
represents a generator of $\pi_2(\Rep(\Z^2,SU(2)))$. Since $\Rep(\Z^2,SU(2)) \cong \mathbb{S}^2$, 
it is enough to verify that 
$\pi \beta$ has degree $\pm 1$. Let $\textnormal{int}(\Delta)$ denote the relative interior of the bottom face of $\partial P$, and let $U\subseteq \SS^2$ be the image of $\textnormal{int}(\Delta)$ under $\pi\beta$. Then $U$ is open and $\pi\beta$ maps $\partial P\backslash \textnormal{int}(\Delta)$ into $\SS^2\backslash U$.  Any $x\in U$ has a unique preimage under $\pi\beta$. By considering local degrees it follows that $\pi \beta$ has degree $\pm 1$.
\end{proof}

Let $\nu\co SU(2)\to G$ be the embedding corresponding to the highest root of $G$.

\begin{corollary}
The homotopy class of $(\nu\times \nu)\beta$ generates $\pi_2(\Hom(\Z^2,G))$.
\end{corollary}
\begin{proof}
This is immediate from Proposition \ref{prop:geometricgenerator} and Theorem \ref{thm:su2represents}.
\end{proof}

\begin{remark}
In \cite{AG,AGLT} a K-theory group was introduced, defined for a finite CW complex $X$ by 
$\tilde{K}_{com}(X):=[X,B_{com}U]$. In \cite[Proposition 5.2]{Simon} it was shown that the 
inclusion $SU(2)\to U$ induces an isomorphism $\pi_4(B_{com}SU(2))\cong \pi_4(B_{com}U)$, and 
$\tilde{K}_{com}(\SS^4)\cong \Z\oplus \Z$. Thus, the results of this section can be used to 
construct explicit commutative cocycles over $\SS^4$, relative to the closed covers 
$\mathcal{C}$ and $\mathcal{D}$, representing the two generators of $\tilde{K}_{com}(\SS^4)$.
\end{remark}

\section*{Funding}

JMG acknowledges the support provided by the 
Max Planck Institute for Mathematics and Minciencias through grant number  FP44842-013-2018 
of the Fondo Nacional de Financiamiento para la Ciencia, la Tecnolog\'ia y la Innovaci\'on. SG received funding from the European Union's Horizon 2020 research and innovation programme under the Marie Sklodowska-Curie grant agreement No. 846448. This project was also supported by the Danish National Research Foundation through the Copenhagen Centre for Geometry and Topology (DNRF151). AA acknowledges support from NSERC prior to October 2019.

\section*{Acknowledgements}

We thank Shrawan Kumar and Burt Totaro for their valuable feedback. We also thank the referee for their reading of the manuscript and their comments.

\end{document}